\documentclass[11pt]{elsarticle}

\usepackage{lineno,hyperref}
\usepackage[]{hyperref}

\modulolinenumbers[5]

\journal{Computer Physics Communications}
\usepackage[margin=2.5cm]{geometry}
\biboptions{sort&compress}

\makeatletter
\providecommand{\doi}[1]{%
	\begingroup
	\let\bibinfo\@secondoftwo
	\urlstyle{rm}%
	\href{http://dx.doi.org/#1}{%
		doi:\discretionary{}{}{}%
		\nolinkurl{#1}%
	}%
	\endgroup
}
\makeatother

\usepackage{graphicx}
\usepackage{dcolumn}
\usepackage{bm}
\usepackage{hyperref}
\usepackage{epstopdf}
\usepackage{amsmath,esint}
\usepackage{amsmath}
\usepackage{amsfonts}
\usepackage{amssymb}
\usepackage{graphicx}
\usepackage{amsmath}
\usepackage{amsthm}
\usepackage{eucal}
\usepackage{amssymb}
\usepackage{mathrsfs}
\usepackage{float}
\usepackage{hyperref}
\usepackage{lineno}
\usepackage{multirow}
\usepackage{booktabs}
\usepackage{graphicx}
\usepackage{subfigure}
\usepackage{dashrule}
\usepackage{adjustbox}
\usepackage[pdftex,dvipsnames]{xcolor}  
\usepackage{graphicx}
\usepackage{dcolumn}
\usepackage{bm}
\usepackage{natbib}
\usepackage{amssymb}
\usepackage{amsthm}
\usepackage{mathtools}
\usepackage{empheq}
\usepackage[most]{tcolorbox}
\usepackage{mathtools}
\usepackage{cleveref}
\usepackage{soul}
\usepackage{array}

\definecolor{bleuUCLfonce}{rgb}{ .13, .52, .86}

\usepackage{tikz}
\usetikzlibrary{shapes,arrows,chains}
\usetikzlibrary{calc}
\usetikzlibrary{positioning}

\usepackage{tgpagella}
\usepackage{newpxmath}

\renewcommand{\vec}[1]{\mathbf{#1}}

\usepackage{xspace}
\let\storeBeta=\beta
\renewcommand\beta{\relax\ifmmode{\storeBeta}\else{$\storeBeta$}\fi\xspace}

\let\storeAlpha=\alpha
\renewcommand\alpha{\relax\ifmmode{\storeAlpha}\else{$\storeAlpha$}\fi\xspace}

\newcommand{\abs}[1]{\left| #1 \right|} 
\renewcommand{\d}[2]{\frac{\mathrm{d} #1}{\mathrm{d} #2}} 
\newcommand{\pd}[2]{\frac{\partial #1}{\partial #2}} 
 
\let\baraccent=\= 
\renewcommand{\=}[1]{\stackrel{#1}{=}} 

\newcount\colveccount
\newcommand*\colvec[1]{
	\global\colveccount#1
	\begin{pmatrix}
		\colvecnext
	}
	\def\colvecnext#1{
		#1
		\global\advance\colveccount-1
		\ifnum\colveccount>0
		\\
		\expandafter\colvecnext
		\else
	\end{pmatrix}
	\fi
}

\newcommand{\norm}[1]{\left\lVert#1\right\rVert}

\def\tphi{\widetilde{\phi}}
\def\tmu{\widetilde{\mu}}
\def\trhok{\widetilde{\rho}^{\, k}}
\def\tetak{\widetilde{\eta}^{\, k}}
\def\tp{\widetilde{p}}
\def\tvi{\widetilde{v}_i}
\def\tvj{\widetilde{v}_j}
\def\tJi{\widetilde{J}_i}
\def\tJj{\widetilde{J}_j}
\def\tpsi{\widetilde{\psi}}
\def\tvecv{\widetilde{\vec{v}}}

\hypersetup{
	colorlinks,
	linkcolor={blue!50!black},
	citecolor={blue!50!black},
	urlcolor={blue!80!black},
	anchorcolor = {blue!80!black},
	filecolor = {blue!80!black},
	menucolor = {blue!80!black},
	runcolor = {blue!80!black}
}

\usepackage{tcolorbox}
\tcbuselibrary{theorems}
\usepackage{varwidth}
\tcbuselibrary{skins}
\tcbuselibrary{breakable}
\tcbsetforeverylayer{shield externalize}
\newtcbtheorem[crefname={theorem}{theorems}]{theorem}{Theorem}%
{	enhanced,
	frame empty,
	interior empty,
	colframe=black,
	coltitle=black,
	fonttitle=\bfseries,
	colbacktitle=gray!10!white,
	borderline={0.5mm}{0mm}{gray!15!white},
	borderline={0.5mm}{0mm}{black!50!white,dashed},
	attach boxed title to top center={yshift=-2mm},
	boxed title style={boxrule=0.4pt},
	attach boxed title to top left={%
		yshift*=-\tcboxedtitleheight/2, 
		xshift=5mm},
	varwidth boxed title}{thrm}

\newtcbtheorem[crefname={lemma}{lemmas}]{lemma}{Lemma}%
{	enhanced,
	frame empty,
	interior empty,
	colframe=black,
	coltitle=black,
	fonttitle=\bfseries,
	colbacktitle=gray!10!white,
	borderline={0.5mm}{0mm}{gray!15!white},
	borderline={0.5mm}{0mm}{black!50!white,dashed},
	attach boxed title to top center={yshift=-2mm},
	boxed title style={boxrule=0.4pt},
	attach boxed title to top left={%
		yshift*=-\tcboxedtitleheight/2, 
		xshift=5mm},
	varwidth boxed title}{lem}

\newtcbtheorem[crefname={proposition}{propositions}]{proposition}{Proposition}%
{	enhanced,
	frame empty,
	interior empty,
	colframe=black,
	coltitle=black,
	fonttitle=\bfseries,
	colbacktitle=gray!10!white,
	borderline={0.5mm}{0mm}{gray!15!white},
	borderline={0.5mm}{0mm}{black!50!white,dashed},
	attach boxed title to top center={yshift=-2mm},
	boxed title style={boxrule=0.4pt},
	attach boxed title to top left={%
		yshift*=-\tcboxedtitleheight/2, 
		xshift=5mm},
	varwidth boxed title}{prop}

\newtcbtheorem[crefname={definition}{definitions}]{definition}{Definition}%
{	enhanced,
	frame empty,
	interior empty,
	colframe=black,
	coltitle=black,
	fonttitle=\bfseries,
	colbacktitle=gray!10!white,
	borderline={0.5mm}{0mm}{gray!15!white},
	borderline={0.5mm}{0mm}{black!50!white,dashed},
	attach boxed title to top center={yshift=-2mm},
	boxed title style={boxrule=0.4pt},
	attach boxed title to top left={%
		yshift*=-\tcboxedtitleheight/2, 
		xshift=5mm},
	varwidth boxed title}{defn}

\newtcbtheorem[crefname={claim}{claims}]{claim}{Claim}%
{	enhanced,
	frame empty,
	interior empty,
	colframe=black,
	coltitle=black,
	fonttitle=\bfseries,
	colbacktitle=gray!10!white,
	borderline={0.5mm}{0mm}{gray!15!white},
	borderline={0.5mm}{0mm}{black!50!white,dashed},
	attach boxed title to top center={yshift=-2mm},
	boxed title style={boxrule=0.4pt},
	attach boxed title to top left={%
		yshift*=-\tcboxedtitleheight/2, 
		xshift=5mm},
	varwidth boxed title}{claim}

\newtcbtheorem[crefname={corollary}{corollarys}]{corollary}{Corollary}%
{	enhanced,
	frame empty,
	interior empty,
	colframe=black,
	coltitle=black,
	fonttitle=\bfseries,
	colbacktitle=gray!10!white,
	borderline={0.5mm}{0mm}{gray!15!white},
	borderline={0.5mm}{0mm}{black!50!white,dashed},
	attach boxed title to top center={yshift=-2mm},
	boxed title style={boxrule=0.4pt},
	attach boxed title to top left={%
		yshift*=-\tcboxedtitleheight/2, 
		xshift=5mm},
	varwidth boxed title}{cor}

\tcolorboxenvironment{proof}{
	blanker,
	breakable,
	left=5mm,
	borderline west={0.5mm}{0pt}{black!50!white}
}

\newtheorem{remark}{Remark}

\usepackage{pgfplots}
\pgfplotsset{
	table/search path={Figures},
}
\usetikzlibrary{arrows,shapes}
\pgfplotsset{compat=1.3}
\pgfplotsset{
	discard if/.style 2 args={
		x filter/.code={
			\edef\tempa{\thisrow{#1}}
			\edef\tempb{#2}
			\ifx\tempa\tempb
			
			\fi
		}
	},
	discard if not/.style 2 args={
		x filter/.code={
			\edef\tempa{\thisrow{#1}}
			\edef\tempb{#2}
			\ifx\tempa\tempb
			\else
			
			\fi
		}
	}
}
\usetikzlibrary{plotmarks}
\usetikzlibrary{spy}
\usetikzlibrary{arrows,shapes,plotmarks}
\usepgfplotslibrary{groupplots}

\usetikzlibrary{calc}

\usetikzlibrary{matrix}

\usepackage{xargs}                      
\usepackage{siunitx}

\usepackage[colorinlistoftodos,prependcaption,textsize=normalsize,disable]{todonotes}

\newcounter{bgunsure}
\newcounter{bgchange}
\newcounter{bginfo}
\newcounter{bgimprovement}
\newcounter{bgthiswillnotshow}
\newcommandx{\bgunsure}[2][1=]{\refstepcounter{bgunsure}\todo[linecolor=red,backgroundcolor=red!25,bordercolor=red,#1]{[BG\thebgunsure:] #2}}
\newcommandx{\bgchange}[2][1=]{\refstepcounter{bgchange}\todo[linecolor=blue,backgroundcolor=blue!25,bordercolor=blue,#1]{[BG\thebgchange:] #2}}
\newcommandx{\bginfo}[2][1=]{\refstepcounter{bgimprovement}\todo[linecolor=OliveGreen,backgroundcolor=OliveGreen!25,bordercolor=OliveGreen,#1]{[BG\thebgimprovement:] #2}}
\newcommandx{\bgimprovement}[2][1=]{\refstepcounter{bginfo}\todo[linecolor=Plum,backgroundcolor=Plum!25,bordercolor=Plum,#1]{[BG\thebginfo:] #2}}
\newcommandx{\bgthiswillnotshow}[2][1=]{\refstepcounter{bgthiswillnotshow}\todo[disable,#1]{[BG\thebgthiswillnotshow:] #2}}

\newcommand{\Stampede}{\href{https://www.tacc.utexas.edu/systems/stampede2}{Stampede2}}

\newcounter{mkunsure}
\newcounter{mkchange}
\newcounter{mkinfo}
\newcounter{mkimprovement}
\newcounter{mkthiswillnotshow}
\newcommandx{\mkunsure}[2][1=]{\refstepcounter{mkunsure}\todo[linecolor=red,backgroundcolor=red!25,bordercolor=red,#1]{[MAK\themkunsure:] #2}}
\newcommandx{\mkchange}[2][1=]{\refstepcounter{mkchange}\todo[linecolor=blue,backgroundcolor=blue!25,bordercolor=blue,#1]{[MAK\themkchange:] #2}}
\newcommandx{\mkinfo}[2][1=]{\refstepcounter{mkimprovement}\todo[linecolor=OliveGreen,backgroundcolor=OliveGreen!25,bordercolor=OliveGreen,#1]{[MAK\themkimprovement:] #2}}
\newcommandx{\mkimprovement}[2][1=]{\refstepcounter{mkinfo}\todo[linecolor=Plum,backgroundcolor=Plum!25,bordercolor=Plum,#1]{[MAK\themkinfo:] #2}}
\newcommandx{\mkthiswillnotshow}[2][1=]{\refstepcounter{mkthiswillnotshow}\todo[disable,#1]{[MAK\themkthiswillnotshow:] #2}}

\newcounter{vcunsure}
\newcounter{vcchange}
\newcounter{vcinfo}
\newcounter{vcimprovement}
\newcounter{vcthiswillnotshow}
\newcommandx{\vcunsure}[2][1=]{\refstepcounter{vcunsure}\todo[linecolor=red,backgroundcolor=red!25,bordercolor=red,#1]{[VC\themkunsure:] #2}}
\newcommandx{\vcchange}[2][1=]{\refstepcounter{vcchange}\todo[linecolor=blue,backgroundcolor=blue!25,bordercolor=blue,#1]{[VC\themkchange:] #2}}
\newcommandx{\vcinfo}[2][1=]{\refstepcounter{vcimprovement}\todo[linecolor=OliveGreen,backgroundcolor=OliveGreen!25,bordercolor=OliveGreen,#1]{[VC\themkimprovement:] #2}}
\newcommandx{\vcimprovement}[2][1=]{\refstepcounter{vcinfo}\todo[linecolor=Plum,backgroundcolor=Plum!25,bordercolor=Plum,#1]{[VC\themkinfo:] #2}}
\newcommandx{\vcthiswillnotshow}[2][1=]{\refstepcounter{vcthiswillnotshow}\todo[disable,#1]{[VC\themkthiswillnotshow:] #2}}

\newcounter{mfunsure}

\newcounter{mfimprovement}
\newcounter{mfthiswillnotshow}
\newcommandx{\mfunsure}[2][1=]{\refstepcounter{mfunsure}\todo[linecolor=red,backgroundcolor=red!25,bordercolor=red,#1]{[MF\themkunsure:] #2}}
\newcommandx{\mfchange}[2][1=]{\refstepcounter{mfchange}\todo[linecolor=blue,backgroundcolor=blue!25,bordercolor=blue,#1]{[MF\themkchange:] #2}}
\newcommandx{\mfinfo}[2][1=]{\refstepcounter{mfimprovement}\todo[linecolor=OliveGreen,backgroundcolor=OliveGreen!25,bordercolor=OliveGreen,#1]{[MF\themkimprovement:] #2}}
\newcommandx{\mfimprovement}[2][1=]{\refstepcounter{mfinfo}\todo[linecolor=Plum,backgroundcolor=Plum!25,bordercolor=Plum,#1]{[MF\themkinfo:] #2}}
\newcommandx{\mfthiswillnotshow}[2][1=]{\refstepcounter{mfthiswillnotshow}\todo[disable,#1]{[MF\themkthiswillnotshow:] #2}}

\newcommand{\dendro}{\textsc{Dendro}}
\newcommand{\petsc}{\textsc{Petsc}}
\newcommand{\dendroFour}{\textsc{Dendro4}}
\newcommand{\dendroFive}{\textsc{Dendro5}}
\newcommand{\oTo}{\textsc{o2o}}
\newcommand{\oTn}{\textsc{o2n}}

\usepackage{listings}

\usepackage{xparse}

\NewDocumentCommand{\codeword}{v}{%
	\texttt{\textcolor{blue}{#1}}%
}

\definecolor{codegreen}{rgb}{0,0.6,0}
\definecolor{codegray}{rgb}{0.5,0.5,0.5}
\definecolor{codepurple}{rgb}{0.58,0,0.82}
\definecolor{backcolour}{rgb}{0.95,0.95,0.92}

\lstdefinestyle{mystyle}{
	backgroundcolor=\color{backcolour},   
	commentstyle=\color{codegreen},
	keywordstyle=\color{magenta},
	stringstyle=\color{codepurple},
	basicstyle=\ttfamily\footnotesize,
	breakatwhitespace=false,         
	breaklines=true,                 
	captionpos=b,                    
	keepspaces=true,                  
	showspaces=false,                
	showstringspaces=false,
	showtabs=false,                  
	tabsize=2
}

\definecolor{sq_b1}{RGB}{37,52,148}
\definecolor{sq_b2}{RGB}{44,127,184}
\definecolor{sq_b3}{RGB}{65,182,196}
\definecolor{sq_b4}{RGB}{127,205,187}
\definecolor{sq_b5}{RGB}{199,233,180}
\definecolor{sq_b6}{RGB}{255,255,204}

\definecolor{sq_r1}{RGB}{189,0,38}
\definecolor{sq_r2}{RGB}{240,59,32}
\definecolor{sq_r3}{RGB}{253,141,60}
\definecolor{sq_r4}{RGB}{254,178,76}
\definecolor{sq_r5}{RGB}{254,217,118}
\definecolor{sq_r6}{RGB}{255,255,178}

\definecolor{sq_g1}{RGB}{0,104,55}
\definecolor{sq_g2}{RGB}{49,163,84}
\definecolor{sq_g3}{RGB}{120,198,121}
\definecolor{sq_g4}{RGB}{173,221,142}
\definecolor{sq_g5}{RGB}{217,240,163}
\definecolor{sq_g6}{RGB}{255,255,204}

\definecolor{div_c1}{RGB}{230,171,2}
\definecolor{div_c2}{RGB}{102,166,30}
\definecolor{div_c3}{RGB}{231,41,138}
\definecolor{div_c4}{RGB}{117,112,179}
\definecolor{div_c5}{RGB}{217,95,2}
\definecolor{div_c6}{RGB}{27,158,119}
\definecolor{div_c7}{RGB}{215,48,39}

\definecolor{div_d1}{RGB}{215,25,28}
\definecolor{div_d2}{RGB}{253,174,97}
\definecolor{div_d3}{RGB}{255,255,191}
\definecolor{div_d4}{RGB}{171,217,233}
\definecolor{div_d5}{RGB}{44,123,182}

\allowdisplaybreaks
\usetikzlibrary{external}
\begin{document}

\begin{frontmatter}

\title{A fully-coupled framework for solving Cahn-Hilliard Navier-Stokes equations: Second-order, 
energy-stable numerical methods on adaptive octree based meshes}


\author[isuMechEAddress,isuMathAddress]{Makrand A. Khanwale\fnref{MAKFootnote}}
\ead{khanwale@iastate.edu}

\author[isuMechEAddress]{Kumar Saurabh}
\ead{maksbh@iastate.edu}

\author[utahAddress]{Milinda Fernando\fnref{MFFootnote}}
\ead{milinda@cs.utah.edu}

\author[curtinAddress,minResAddress,cicAddress]{Victor M. Calo}
\ead{Victor.Calo@Curtin.edu.au}

\author[utahAddress]{Hari Sundar}
\ead{hari@cs.utah.edu}

\author[isuMathAddress]{James A. Rossmanith}
\ead{rossmani@iastate.edu}

\author[isuMechEAddress]{Baskar~Ganapathysubramanian\corref{correspondingAuthor}}
\ead{baskarg@iastate.edu}

\cortext[correspondingAuthor]{Corresponding author}
\fntext[MAKFootnote]{Presently at Center for Turbulence Research, Stanford University, CA, USA}
\fntext[MFFootnote]{Presently at Oden Institute for Computational Engineering and Sciences, University of Texas at Austin, TX, USA}

\address[isuMechEAddress]{Department of Mechanical Engineering, Iowa State University, Iowa, USA 50011}
\address[isuMathAddress]{Department of Mathematics, Iowa State University, Iowa, USA 50011}
\address[utahAddress]{School of Computing, The University of Utah, Salt Lake City, Utah, USA 84112}
\address[curtinAddress]{Department of Applied Geology, Western Australian School of Mines, Faculty of Science and Engineering, Curtin University, Kent Street, Bentley, Perth, WA 6102, Australia}
\address[minResAddress]{Mineral Resources, Commonwealth Scientific and Industrial Research Organization (CSIRO), Kensington, Perth, WA 6152, Australia}
\address[cicAddress]{Curtin Institute for Computation, Curtin University, Kent Street, Bentley, Perth, WA 6102, Australia}

\begin{abstract}
	We present a fully-coupled, implicit-in-time framework for solving a thermodynamically-consistent Cahn-Hilliard Navier-Stokes system that models two-phase flows. In this work, we extend the block iterative method presented in Khanwale et al. [{\it Simulating two-phase flows with thermodynamically consistent energy stable Cahn-Hilliard Navier-Stokes equations on parallel adaptive octree based meshes}, J. Comput. Phys. (2020)], to a fully-coupled, provably second-order accurate scheme in time, while maintaining energy-stability. The new method requires fewer matrix assemblies in each Newton iteration resulting in faster solution time. The method is based on a fully-implicit Crank-Nicolson scheme in time and a pressure stabilization for an equal order Galerkin formulation. That is, we use a conforming continuous Galerkin (cG) finite element method in space equipped with a residual-based variational multiscale (RBVMS) procedure to stabilize the pressure.  We deploy this approach on a massively parallel numerical implementation using parallel octree-based adaptive meshes.  We present comprehensive numerical experiments showing detailed comparisons with results from the literature for canonical cases, including the single bubble rise, Rayleigh-Taylor instability, and lid-driven cavity flow problems.  We analyze in detail the scaling of our numerical implementation.
\end{abstract}
 

\end{frontmatter}



\section{Introduction}
\label{sec:introduction}
\tikzexternaldisable

\subsection{Motivation}
Understanding the fundamental mechanisms of phase interactions is critical to developing accurate and efficient models of two-phase flows. In particular, insights into the phase interactions may lead to accurate, low-cost coarse-scale models for large systems (e.g., chemical/biological reactors); and optimization-based design of micro-scale systems (e.g., bio-microfluidics and advanced manufacturing using multiphase flows).  To achieve this understanding, we need physically accurate models that fully capture the phase interactions through a consistent description of the interfacial processes and the interface evolution.  Models using coupled Cahn-Hilliard Navier-Stokes (CHNS) equations bear the promise of providing such a description~\citep{ Anderson1998}.  

Specifically, applications in bio-microfluidics and advanced manufacturing are characterized by two-phase flows with a steady state or a periodic set of states~\citep{Teh2008,Dangla2013,Stoecklein2014,Stoecklein2016a}.  Therefore, simulations of such flows require numerical schemes designed to allow for large time steps.  With this objective in mind we focus on design of fully-implicit non-linear time scheme which allows for larger time steps.  We list the following key contributions of this paper. 
\begin{enumerate}
	\item Design second order fully-coupled fully implicit time stepping scheme building on the analysis in~\citet{Khanwale2020}.  
	\item Analyze the time-scheme for energy stability analysis which is an important property of the CHNS model. 
	\item Deploy the numerical method in a large scale parallel computational framework which utilizes highly efficient octree based adaptive meshes. 
\end{enumerate}

\subsection{Background}
  
Similar to level set methods~\citep{Osher2006,Yan2018}, CHNS models track the boundary between two phases using a smooth function, referred to as the \textit{phase field} function.  This approach allows for a diffuse transition between the physical properties from one phase to the other and circumvents modeling the jump discontinuities at the interface. The use of the Cahn-Hilliard (CH) equations to track the phase field offers many advantages, including mass conservation, thermodynamic consistency, and a free-energy-based description of surface tension with a well-established theory from non-equilibrium thermodynamics \citep{ Jacqmin1996, Jacqmin2000}.  Carefully designed numerical schemes allow the discrete numerical solutions of these CHNS models to inherit these continuous properties.
	
CHNS models couple the momentum equation governing an ``averaged mixture velocity'' with the interface-tracking Cahn-Hilliard equation. Even assuming that each fluid phase is incompressible, the resulting mixture velocity may not be solenoidal; the pointwise incompressibility depends on the averaging. Volume averaging, at least under strictly isothermal conditions, usually results in a solenoidal mixture velocity. On the other hand, mass averaging results in a non-solenoidal mixture velocity resulting in quasi-incompressible models (e.g.,~\citet{ Guo2017, Shokrpour2018}, and references therein). Here, we use the volume averaging strategy. The resulting solenoidal mixture velocity is a useful property in the subsequent numerical method development.
	
In~\citet{Khanwale2020} we proposed an energy-stable and mass-conserving discretization of the thermodynamically-consistent CHNS model; that approach used a block iterative method for solving the two sets of equations (i.e., Cahn-Hilliard (CH) and Navier-Stokes (NS)).  Therefore, the fully discrete system resulted in two non-linear systems of algebraic equations, one corresponding to the discretized version of the momentum equations, the other corresponding to the discretized Cahn-Hilliard equations.  Both sets of equations used an implicit time-stepping strategy alongside an internal (within each block iteration) Newton's method for solving the resulting non-linear algebraic equations.  While this strategy decoupled the implementation challenges of the CH and NS systems, it required multiple matrix assemblies due to the multiple block iterations within each time-step.  \citet{Xu2020} and \citet{Saurabh2020} showed that the matrix assembly and preconditioner setup in a framework such as \citet{Khanwale2020} can be very expensive. 

In the current work we seek to improve the results of~\citep{ Khanwale2020} in three key aspects. 
	\begin{enumerate}
		\item \textbf{Second-order energy-stable scheme:} We extend the time integration scheme presented in~\citep{ Khanwale2020} to second-order accuracy while maintaining energy-stability and mass conservation. Furthermore, to address computational efficiency concerns, we use a fully-coupled method instead of using a block-iterative approach.  The coupling entails solving the fully discretized CHNS system. Therefore, within each Newton iteration, we solve a linear system with six degrees of freedom per node (three velocities, one pressure, two phase field variables).
		
		\item \textbf{VMS-based stabilization for conforming Galerkin elements:} We extend the variational multiscale (VMS) based treatment in~\citep{ Khanwale2020} to the fully-coupled approach with conforming Galerkin elements.  We use VMS stabilization to circumvent the discrete inf-sup condition in equal order polynomial representations for velocity and pressure (e.g.,~\citet{ Volker2016}). This is especially important in the case of adaptive $h$-refinement~\citep{ Burstedde2011, Sundar2007, Dendro}, which is extensively used here.  
		
		\item \textbf{Scalable octree-based adaptive mesh:} We apply the proposed method to problems where we need sufficient resolution of the interfacial length scales to capture the interface dynamics accurately. To make this computationally tractable, we implement the proposed numerical scheme inside the~\dendroFive~\cite{ Dendro5} adaptive mesh refinement framework, which efficiently resolves the interface dynamics in 3D systems with highly deformable interfaces. This implementation extends the previous implementation of \cite{Khanwale2020}, which was done in the older ~\dendroFour~\cite{ Dendro} framework.
	\end{enumerate}

\subsubsection{Second-order energy-stable scheme}
There are two main approaches for CHNS modeling depending on the averaging used to define the mixture velocity.  Examples of CHNS models based on volume-averaged velocity include~\citet{ Kim2004a, Feng2006, Shen2010a, Shen2010b, Dong2012, Chen2016, Dong2016}, and \citet{Zhu2019}.  \citet{ Kim2004a}~used a strategy similar to block iteration, while~\citet{ Feng2006} used a fully-coupled approach similar to one we adopt in this work.  \citet{Shen2010a, Shen2010b,Chen2021}~used a block-solve strategy with linearized time-schemes that reduce their discretization to a sequence of elliptic equations for the velocity and phase fields.  Subsequently, \citet{ Han2015}~also used a 
block-iterative strategy with an energy-stable time scheme but with a non-linear scheme.  \citet{ Chen2016}~showed an energy-stable time scheme with a fully-coupled solver. \citet{ Guo2017}~recently reported a detailed analysis for a mass averaged mixture velocity CHNS system using a fully coupled strategy. \citet{Fu2021} presented a linear fully-coupled BDF2 based scheme which uses Scalar Auxiliary Variable approach to linearize the Cahn-Hilliard equations.  The SAV approach modifies the energy law and~\mbox{\citet{Fu2021}} prove stability of this modified energy law. \citet{Fu2021} use an SUPG/PSPG approach for pressure stabilization in contrast to the more fundamental VMS approach in the current paper.  Indeed the SUPG/PSPG approach is can be derived from the more general VMS approach (see ~\citet{ hughes2018multiscale}). In addition to pressure stabilization VMS provides a natural leeway into modeling high-Reynolds number flows~\citep{ Hughes1995,Bazilevs2007} as it uses a filtering approach similar to large-eddy simulation (LES).

We emphasize that the focus in~\citet{Khanwale2020} and the current paper is on non-linear schemes. This is in contrast to recent elegant work in designing linear schemes for the CH-NS equations~\citep{Shen2010a, Shen2010b, Dong2012, Chen2016, Dong2016, Zhu2019}. Non-linear schemes in general allows using larger time-steps. This feature is important, particularly for problems where we are most interested in rapidly reaching a steady state (or a periodic set of states).  There are many critical applications in bio-microfluidics where one needs to determine steady state (or limit cycle) behavior of the two-phase system~\citep{Mutlu2018, Stoecklein2018}. Linearized schemes~\citep{Shen2010a, Shen2010b, Shen2015,Zhu2019} generally have a stringent time-step restriction and therefore may not be well suited for this particular class of applications. In contrast, such linear schemes are particularly well-suited for problems with naturally small time steps, like tracking instabilities\footnote{It would be very valuable to the bio-microfluids community to have a comparative assessment of optimized implementations of both linear/de-coupled vs non-linear/coupled to study the trade-off between smaller times step requirement versus increased computational complexity per time step. This is beyond the scope of this paper; however, we have made our code open-source to encourage such an analysis.}.
Here, in~\cref{subsec:time_scheme}, we prove that a second-order extension of the time scheme presented in~\citep{ Khanwale2020} for the fully-coupled solver is energy-stable and mass conserving.  The benefit of such a time integration scheme is that it does not require storage of more than one previous time step, while still providing accuracy and ensuring energy stability.  

\subsubsection{VMS-based stabilization for conforming Galerkin elements}
Discretizing the momentum equations in the CHNS model with solenoidal velocity requires velocity-pressure pairs that satisfy the discrete inf-sup condition (e.g., Ladyzhenskaya-Babuska-Brezzi (LBB) stable).   However, we prefer to use standard conforming Galerkin finite elements to leverage parallel adaptive meshing tools. These methods circumvent the discrete inf-sup condition by using stabilization approaches such as \textit{grad-div} stabilizations.  We build such a stabilization using a VMS approach~\citep{ Hughes2000, Ahmed2017}.  The VMS approach 
involves the use of an adjustable constant, $\tau$, which requires careful design for a fully-coupled system.  In this work, we extend the formulation based on the Residual-Based Variational Multiscale Method (RBVMS)~\cite{ Bazilevs2007} we presented in~\citep{ Khanwale2020} to the fully coupled approach in~\cref{subsec:space_scheme}. 
	
\subsubsection{Scalable octree-based adaptive mesh generation} 
Adaptive spatial discretizations are popular in computational sciences~\citep{ Coupez2013, Hachem2013, Hachem2016} to improve efficiency and resolution quality. In some applications (e.g., ~\citep{ MasadoDKT, mantle, Fernando2018_GR}) an adaptive spatial discretization is the key to make those simulations feasible on modern supercomputers. In distributed-memory computations, adaptive discretizations introduce additional computational challenges such as load-balancing, low-cost mesh generation, and mesh-partitioning. The scalability of the algorithms used for mesh generation and partitioning is crucial, especially when the represented solution requires frequent re-meshing. Octrees~\citep{ Sundar2008, BursteddeWilcoxGhattas11, Fernando2018_GR, MasadoDKT} are widely used in the community due to their simplicity and their extreme parallel scalability. In~\citep{ Khanwale2020}, we used \dendroFour~\cite{ Dendro} as the underlying parallel octree library. \dendroFour~is a parallel octree library that supports linear finite element computations on adaptive octrees. \dendroFive~\cite{ Dendro5} extends \dendroFour~and supports higher-order finite difference (FD), finite volume (FV), and finite element (FE) discretizations on fully adaptive octrees. In the present work, we use \dendroFive~\cite{ Dendro5} as our primary parallel octree mesh library. From now on, we use \dendro~to refer to \dendroFive, unless otherwise specified. 
	
\dendro~\cite{ Dendro5} is a freely available open-source library that is currently used by several research communities to tackle problems in computational relativity~\cite{ Fernando2018_GR}, relativistic fluid dynamics, and other computational science applications. \dendro~octree generation and partitioning is based on the \textsc{TreeSort}~\cite{ Fernando:2017} algorithm.  Octant neighborhood information is needed to perform numerical computations on topological octrees. To compute these neighborhood data structures, we use the \textsc{TreeSearch} algorithm, which has better scalability compared to traditional binary search approaches~\cite{ Fernando2018_GR}. \dendro \, enforces a 2:1 balancing constraint that ensures that adjacent octants differ by at most a factor of 2 in size. This constraint imposition uses top-down and bottom-up traversals with minor modifications to the \textsc{TreeSort} algorithm.
We detail the adaptive meshing and scalability of our framework in~\cref{sec:octree_mesh} and~\cref{sec:scaling}, respectively.   

\section{Governing equations}
\label{sec:governing_equations}

	We consider a bounded domain $\Omega \subset \mathbb{R}^d$, for $d = 2,3$ containing two immiscible fluids, and the time interval, $[0, T]$. Let $\rho_{+}$ ($\eta_{+}$ ) and $\rho_{-}$ ($\eta_{-}$) denote the specific density (viscosity) of the two phases. Let the phase field function, $\phi$, be the variable  that tracks the location of the phases and varies smoothly between $+1$ and $-1$.  The non-dimensional density is 
	\begin{equation}
	\rho(\phi) = \alpha\phi + \beta, \qquad \text{where} \quad \alpha = \frac{\rho_+ - \;\rho_-}{2\rho_+}, \quad \beta = \frac{\rho_+ + \;\rho_-}{2\rho_+}.
	\label{eqn:alpha_beta_defn}
	\end{equation}
Note that our non-dimensional form uses the specific density/viscosity of fluid 1 as the non-dimensionalizing density/viscosity. Similarly, the non-dimensional viscosity is
\begin{equation}
\eta(\phi) = \gamma\phi + \xi, \qquad \text{where} \quad \gamma = \frac{\eta_+ - \;\eta_-}{2\eta_+}, \quad \xi = \frac{\eta_+ + \;\eta_-}{2\eta_+}.
\end{equation}
The governing equations in their non-dimensional form are as follows:

\begin{align}
\begin{split}
	\text{Momentum Eqns:} & \quad \pd{\left(\rho(\phi) v_i\right)}{t} + \pd{\left(\rho(\phi)v_iv_j\right)}{x_j} + \frac{1}{Pe}\pd{\left(J_jv_i\right)}{x_j} +\frac{Cn}{We} \pd{}{x_j}\left({\pd{\phi}{x_i}\pd{\phi}{x_j}}\right) \\
	& \quad \quad \quad + 
	\frac{1}{We}\pd{p}{x_i} - \frac{1}{Re}\pd{}{x_j}\left({\eta(\phi)\pd{v_i}{x_j}}\right) - \frac{\rho(\phi)\hat{{g_i}}}{Fr} = 0, \label{eqn:nav_stokes} 
	\end{split} \\
	\text{Thermo Consistency:} & \quad J_i = \frac{\left(\rho_- - \rho_+\right)}{2\;\rho_+ Cn} \, m(\phi) 
	\, \pd{\mu}{x_i},\\
	\text{Solenoidality:} & \quad \pd{v_i}{x_i} = 0, \label{eqn:solenoidality}\\
	\text{Continuity:} & \quad \pd{\rho(\phi)}{t} + \pd{\left(\rho(\phi)v_i\right)}{x_i}+
	\frac{1}{Pe} \pd{J_i}{x_i} = 0, \label{eqn:cont}\\
	\text{Chemical Potential:} & \quad \mu = \psi'(\phi) - Cn^2 \pd{}{x_i}\left({\pd{\phi}{x_i}}\right)  ,\label{eqn:mu_eqn} 
	\\ 
	\text{Cahn-Hilliard Eqn:} & \quad \pd{\phi}{t} + \pd{\left(v_i \phi\right)}{x_i} - \frac{1}{PeCn} \pd{}{x_i}\left({m(\phi){\pd{\mu}{x_i}}}\right) = 0. 
	\label{eqn:phi_eqn}
\end{align}
Note that we use Einstein notation throughout this work; in this notation $v_i$ represents the $i^{\text{th}}$ component of the vector $\vec{v}$, and any repeated index is implicitly summed over.
In the above equations, $\vec{v}$ is the volume-averaged mixture velocity, $p$ is the volume-averaged pressure, $\phi$ is the phase field (interface tracking variable), and $\mu$ is the chemical potential.  Non-dimensional mobility $m(\phi)$ is assumed to be a constant with a value of one.  The non-dimensional parameters are as follows: Peclet, $Pe = \frac{u_{r} L_{r}^2}{m_{r}\sigma}$; Reynolds, $Re = \frac{u_{r} L_{r}}{\nu_{r}}$; Weber, $We = \frac{\rho_{r}u_{r}^2 L_{r}}{\sigma}$; Cahn, $Cn = \frac{\varepsilon}{L_{r}}$; and Froude, $Fr = \frac{u_{r}^2}{gL_{r}}$, with $u_{r}$ and $L_r$ denoting the reference velocity and length, respectively. Here, $m_{r}$ is the reference mobility, $\sigma$ is the scaling interfacial tension, $\nu_{r} = \eta_{+}/\rho_{+}$, $\varepsilon$ is the interface thickness, $g$ is gravitational acceleration. $\hat{\vec{g}}$ is a unit vector defined as $\left(0, -1, 0\right)$ denoting the direction of gravity and  $\psi(\phi(\vec{x}))$ is a known free-energy function.  
%
%
%
In particular, we use the polynomial form of the free energy density defined as follows:
\begin{align}
		\psi(\phi) = \frac{1}{4}\left( \phi^2 - 1 \right)^2 \qquad \text{and} \qquad
		\psi'(\phi) = \phi^3 - \phi.
		\label{eqn:free_energy}
\end{align}

The system of equations \cref{eqn:nav_stokes} -- \cref{eqn:phi_eqn} has a dissipative law given by: 
\begin{equation}
\d{E_{\text{tot}}}{t} = -\frac{1}{Re}  \int_{\Omega} \frac{\eta(\phi)}{2} \norm{\nabla \vec{v}}_F^2 \mathrm{d}\vec{x} - \frac{Cn}{We} \int_{\Omega}m(\phi) \norm{\nabla \mu}^2  \mathrm{d}\vec{x} \le 0,
\end{equation}
where the total energy is 
\begin{equation}
E_{\text{tot}}(\vec{v},\phi, t) = \frac{1}{2}\int_{\Omega}\rho(\phi) \norm{\vec{v}}^2 \mathrm{d}\vec{x} + \frac{1}{CnWe}\int_{\Omega} \left(\psi(\phi) + \frac{Cn^2}{2} \norm{\nabla\phi}^2 + \frac{1}{Fr} \rho(\phi) y \right) \mathrm{d}\vec{x}.
\label{eqn:energy_functional}
\end{equation}
The norms used in the above expression are the Euclidean vector norm and the Frobenius matrix norm:
\begin{equation}
\norm{\vec{v}}^2 := \sum_i \abs{v_i}^2 \qquad \text{and} \qquad 
\norm{\nabla\vec{v}}^2_F := \sum_i \sum_j \abs{\frac{\partial v_i}{\partial x_j}}^2.
\end{equation}

\begin{remark}
A realistic interface thickness (parametrized by the Cahn number) is in the nanometer range; resolving this scale is computationally intractable, as all the other scales in the problem are orders of magnitude larger. Therefore, an ansatz that diffuse interface models follow is that the solution tends to the real physics in the limit of $Cn \rightarrow 0$.  This limiting process progressively reduces the Cahn number from large to small until the dynamics become independent of the Cahn number.  However, the choice of Cahn number ($Cn$) determines the Peclet number ($Pe$), which is given by $Pe = \frac{u_{r} L_{r}^2}{m\sigma}$. $Pe$ represents the ratio of the advection timescale to the diffuse interface relaxation time to its equilibrium $\tanh$ profile (a purely computational construct).  \citet{ Magaletti2013} reported a careful asymptotic analysis of these timescales, which suggests a $1/Pe = \alpha Cn^2$ scaling. We use this scaling with $\alpha = 3$.
\end{remark}



\section{Numerical method and its properties}
\label{sec:numerical_tecniques}

We extend the scheme of~\citet{ Khanwale2020} to a fully-implicit, fully-coupled Crank-Nicolson time-marching scheme for the system of~\cref{eqn:nav_stokes} -- \cref{eqn:phi_eqn}.  This extension delivers better accuracy and energy-stability for larger time-steps while only storing data structures of one previous time-step. While this advantage may not be seem significant for smaller cases,  it is very useful for large-scale simulations with billions of unknowns. 

Let $\delta t$ be a time-step; let $t^k := k \delta t$; thus, we define the following time-averages: 
\begin{equation}
\label{eqn:averaging1}
	\widetilde{\vec{v}}^{k} := \frac{\vec{v}^{k} + \vec{v}^{k+1}}{2}, \quad \tp^{k} := \frac{{p}^{k+1}+{p}^{k}}{2},
	\quad \tphi^{k} := \frac{{\phi}^{k+1} + {\phi}^{k}}{2}, \quad \text{and} \quad \tmu^{k} := \frac{{\mu}^{k+1} + {\mu}^{k}}{2},
\end{equation}
and the following function evaluations:
\begin{equation}
\label{eqn:psi_ave_def}
	\tpsi^k := \psi\left( \tphi^k \right), \qquad 
	\tpsi'^{k} := \psi'\left( \tphi^k \right), \qquad
	\trhok := \rho\left(\tphi^k\right), \qquad \text{and} \qquad
	\tetak := \eta\left(\tphi^k\right).
\end{equation}
Using these temporal values, we define our time-discretized weak form of
the Cahn-Hilliard Navier-Stokes (CNHS) equations. 

\begin{definition}{}{variational_form_sem_disc}
	Let $(\cdot,\cdot)$ be the standard $L^2$ inner product. We state the time-discrete variational problem as follows: find $\vec{v}^{k+1}(\vec{x}) \in \vec{H}_0^1(\Omega)$, $p^{k+1}(\vec{x})$, $\phi^{k+1}(\vec{x})$, $\mu^{k+1}(\vec{x})$ $\in {H}^1(\Omega)$ such that
	\begin{align}
	\begin{split}
	\text{Momentum Eqns:} & \quad \left(w_i, \, \trhok \, \frac{v^{k+1}_i - v^k_i}{\delta t}\right) + \left(w_i, \, \trhok \, \tvj^{k} \, \pd{\tvi^{k}}{x_j}\right) \\ & \quad +
	\frac{1}{Pe}\left(w_i, \, \tJj^{k} \, \pd{\tvi^{k}}{x_j}\right) -
	\frac{Cn}{We} \left(\pd{w_i}{x_j}, \, {\pd{\tphi^{k}}{x_i}\pd{\tphi^{k}}{x_j}}\right)  -
	\frac{1}{We}\left(\pd{w_i}{x_i}, \, {\tp^{k}} \right) \\ & \quad + \frac{1}{Re}\left(\pd{w_i}{x_j}, \,\tetak {\pd{\tvi^{k}}{x_j}} \right) -
	\left(w_i,\frac{\trhok \, \hat{g_i}}{Fr}\right) = 0, 
	\label{eqn:nav_stokes_var_semi_disc}
	\end{split} \\ 
	\text{Thermo Consistency:} & \quad \tJi^{k} = \frac{\left(\rho_- - \rho_+ \right)}{2\;\rho_+ Cn} \, \pd{\tmu^{k}}{x_i}, \label{eqn:thermo_consistency_semi_disc} \\
	\text{Solenoidality:} & \quad 
	\quad \left(q, \, \pd{v^{k+1}_{i}}{x_i}\right) = 0,
	\label{eqn:cont_var_semi_disc} \\
	\text{Chemical Potential:} & \quad 
	-\left(q,\tmu^{k}\right) + \left(q, \tpsi'^{k} \right) + Cn^2 \left(\pd{q}{x_i}, \, {\pd{\tphi^{k}}{x_i}}\right)   = 0, \label{eqn:mu_eqn_var_semi_disc}\\
	\text{Cahn-Hilliard Eqn:} & \quad \left(q, \frac{\phi^{k+1} - \phi^k}{\delta t} \right) -
	 \left(\pd{q}{x_i}, \, \tvi^{k} \tphi^{k} \right) + \frac{1}{PeCn} \left(\pd{q}{x_i}, \, {\pd{\tmu^{k}}{x_i}} \right) = 0,
	\label{eqn:phi_eqn_var_semi_disc}\\
	\text{Continuity:} & \quad
	\left(\frac{\rho\left(\phi^{k+1}\right) - \rho\left(\phi^{k}\right)}{\delta t}, \, q\right) + \left(\pd{\left(\trhok \tvj^k \right)}{x_j}, \, q\right)-
	\left(\frac{1}{Pe} \tJj, \, \pd{q}{x_j}\right) = 0,
	\label{eqn:cont_consv_var_semi_disc}
	\end{align}
	$\forall \vec{w} \in \vec{H}^1_0(\Omega)$, $\forall q \in H^1(\Omega)$, given $\vec{v}^{k} \in \vec{H}_0^1(\Omega)$, and $\phi^{k},\mu^{k} \in H^1(\Omega)$.
\end{definition}

\begin{remark}
	In~\cref{defn:variational_form_sem_disc} for the momentum equations, the boundary terms in the variational form are zero because the velocity and the basis functions live in $\vec{H}^1_0(\Omega)$. Also we use the no flux boundary condition for $\phi$ and $\mu$, which makes boundary terms i.e.,  $\left(q, {\pd{\tphi^{k}}{x_i}} \hat{n}_i\right)$ and $\left(q, {\pd{\mu^{k}}{x_i}} \hat{n}_i\right)$, go to zero.  We use these boundary conditions for all the proofs in this paper.  The numerical examples in the paper also use these boundary conditions unless explicitly noted otherwise.
\end{remark}

\begin{remark} While $\phi \in [-1, 1]$ in the continuous equations, the discrete $\phi$ may violate these bounds. These bound violations may not change the dynamics of $\phi$  adversely, but they could lose the strict positivity of some quantities which depend on $\phi$  (e.g., mixture density $\rho(\phi)$ and viscosity $\eta(\phi)$). This effect is especially significant for high density and viscosity contrasts.
We fix this issue by saturation scaling (i.e., we pull back the value of $\phi$ only for the calculation of density and viscosity). We, therefore, define $\phi^*$ for the mixture density and viscosity calculations, where $\phi^*$ is: 
\begin{equation}
\phi^* := 
\begin{cases}
\phi, &\quad \text{if} \;\; \abs{\phi} \leq 1, \\
\mathrm{sign}(\phi), &\quad \text{otherwise.} 
\end{cases}
\label{eqn:phi_for_density}
\end{equation} 
We note that this pull back strategy to keep the density and viscosity realizable is a common technique in Cahn-Hilliard Navier Stokes models~\citep{Shen2010a, Shen2010b, Dong2012, Chen2016, Dong2016, Guo2017, Zhu2019}. 
\label{rmk:phi_pullback}
\end{remark}

\subsection{Energy-stability of the time-stepping scheme}
\label{subsec:time_scheme}
In this subsection, we prove the energy-stability of the time time-stepping scheme.  The result for energy stability in~\citep{ Khanwale2020} is limited to the case of equal density.  Here we try to extend the proof to unequal densities.  For completeness, we recall some of the crucial results from~\citep{ Khanwale2020} which are common.   We begin with mass conservation.
\begin{proposition}{Mass conservation}{mass_conservation}
	The scheme of~\cref{eqn:nav_stokes_var_semi_disc} -- \cref{eqn:phi_eqn_var_semi_disc} with the following boundary conditions: 
	\begin{equation}
	\pd{\tmu}{x_i} \hat{n}_i\Bigl\rvert_{\partial \Omega} = 0, \quad \pd{\tphi}{x_i} \hat{n}_i\Bigl\rvert_{\partial \Omega} = 0, \quad \widetilde{\vec{v}}^{k}\Bigl\rvert_{\partial \Omega} = \vec{0},
	\end{equation}
	where $\hat{\vec{n}}$ is the outward normal to the boundary $\partial \Omega$,
	is globally mass conservative: 
	\begin{equation}
	\int_{\Omega} \phi^{k+1} \, \mathrm{d}\vec{x}  = \int_{\Omega} \phi^{k} \, \mathrm{d}\vec{x}.
	\end{equation}
\end{proposition} 
We verify this claim numerically in~\cref{subsec:manfactured_soln_result,subsec:single_rising_drop_2D}.
\begin{lemma}{Weak equivalence of forcing}{forcing_equivalent}
	The forcing term due to Cahn-Hilliard in the momentum equation,~\cref{eqn:nav_stokes_var_semi_disc}, with the test function $w_i = \delta t \, \tvi^k$, becomes
	\begin{equation}
	\frac{Cn}{We}\left( \pd{}{x_j}\left({\pd{\tphi^k}{x_i}\pd{\tphi^{k}}{x_j}}\right), \delta t \, \tvi^k\right) = \frac{\delta t}{WeCn}\left({\tphi}^k\pd{\tmu^{k}}{x_i},\tvi^k\right),
	\label{eqn:lemma_weak_equiv_forcing}
	\end{equation}
	$\forall \;\; \tphi^k$, $\tmu^{k} \in H^1(\Omega)$, and $\forall \;\; \widetilde{\vec{v}}^k \in  \vec{H}_{0}^1(\Omega)$, where $\vec{v}^k, \vec{v}^{k+1}, p^k, p^{k+1}, \phi^k, \phi^{k+1}, \mu^{k},\mu^{k+1}, $ satisfy \cref{eqn:nav_stokes_var_semi_disc} -- \cref{eqn:phi_eqn_var_semi_disc}.
\end{lemma}
\begin{lemma}{}{free_en_taylor}
	The following identity holds: 
	\begin{equation}
	\begin{split}
	\left({\tpsi'^{k}\left(\vec{x}\right)}, \phi^{k+1}\left(\vec{x}\right) - \phi^k\left(\vec{x}\right)\right) &= \left(\psi(\phi^{k+1}\left(\vec{x}\right)) - \psi(\phi^{k}\left(\vec{x}\right)), 1\right) \\
	&-  \frac{1}{24} \left(\d{^3\psi}{\phi^3}\Biggl|_{\phi=\lambda\left(\vec{x}\right)},
	\left(\phi^{k+1}\left(\vec{x}\right) - \phi^k\left(\vec{x}\right)\right)^3\right),
	\end{split}
	\label{eq:cubic_term}
	\end{equation}
	for some $\lambda\left(\vec{x}\right)$ between $\phi^k\left(\vec{x}\right)$ and $\phi^{k+1}\left(\vec{x}\right)$.
\end{lemma}

\begin{lemma}{} {correction_estimate}
	The following estimate holds:
	\begin{align}
	\frac{1}{24} \left| \left( \d{^3\psi}{\phi^3}\Biggl|_{\phi=\lambda\left(\vec{x}\right)}, \left(\phi^{k+1}\left(\vec{x}\right) - \phi^k\left(\vec{x}\right)\right)^3 \right) \right| \le  \left( C_m \, L_{\text{max}}^3 \, \, P_{\text{max}} \right) \delta t^3,
	\label{eqn:estimate_corr_cl}
	\end{align}
	where 
	\begin{equation}
	\label{eqn:pmax}
	 P_{\text{max}} := \frac{1}{24} \max_{\vec{x} \in \Omega} 
	\, \left|  \d{^3\psi}{\phi^3}\Biggl|_{\phi=\lambda\left(\vec{x}\right)} \right|,
	\end{equation}
    $\lambda(x)$ is a value between $\phi^k(\vec{x})$ and $\phi^{k+1}(\vec{x})$,
	$L_{\text{max}}$ is a global maxima of Lipschitz constants for $\phi$ as a function of time, and $C_m$ is the volume of the physical domain:
	\[
	\left| \phi^{k+1}\left(\vec{x}\right) - \phi^k\left(\vec{x}\right) \right| \le L\left(\vec{x}\right) \, \delta t, \qquad
	C_m := \int_{\Omega} \mathrm{d}\vec{x},
	\qquad \text{and} \qquad L_{\text{max}} := \max_{\vec{x} \in \Omega} \left(L\left(\vec{x}\right)\right).
	\]
	In the current work, the free energy potential is given by \cref{eqn:free_energy}, which
	results in the following simplification of \cref{eqn:pmax}:
	\begin{equation}
	\begin{split}
	P_{\text{max}} = \frac{1}{4}  \max_{\vec{x} \in \Omega} 
	\, \left|  \lambda\left(\vec{x}\right) \right| \le \frac{1}{4} \max_{\vec{x} \in \Omega} \max
	\left\{ \left| \phi^k\left(\vec{x}\right) \right|, \, \left| \phi^{k+1}\left(\vec{x}\right) \right| \right\}.
	\end{split}
	\label{eqn:estimate_corr_cl_simplified}
	\end{equation}
\end{lemma}
\begin{remark}
    While there is no rigorous maximum principle for CHNS that guarantees that $|\phi^k|,|\phi^{k+1}| \le 1$, and hence $P_{\text{max}} \le 0.25$, there is ample computational evidence to suggest that $|\phi^k|,
    |\phi^{k+1}| \lessapprox 1$, at least for initial and boundary data that is of interest in many applications. Hence, in practice we find that $P_{\text{max}} \lessapprox 0.25$. It is important to note that we don't actually need $\phi$ to be bounded by 1 for Lemma 3.  We only need that $\phi^k$ and $\phi^{k+1}$ 
    are each bounded by {\bf some finite constant} over the spatial domain: ${\vec{x} \in \Omega}$.
	
	\textcolor{black}{As $\phi^k$ and $\phi^{k+1}$ are solutions to the Cahn-Hilliard equation, we have proved their continuity in the context of existence in~\citet{Khanwale2020} similar to ~\citet{Han2015}. In \citet{Khanwale2020,Han2015} the analysis is limited to proving $\phi$ is continuous and bounded (requirements for existence), in addition, there are several works in the literature that analyze the precise regularity of $\phi$ beyond just continuity~\citep{Conti2018,Giorgini2019}.  Therefore, if $\phi^k$ and $\phi^{k+1}$ are continuous functions on a bounded domain $\Omega$.  Continuous functions achieve bounds on a bounded domain.\todo[color=cyan]{R1: \#1.1}}
\end{remark}

A major difference between the time integration schemes in~\citep{ Khanwale2020} and this work is the evaluation of mixture density and viscosity at the average of $\phi$ ($\tphi^k$). 

\begin{remark}
	For this proof we add 
	\begin{equation}
	\frac{1}{2} \left(\frac{\rho\left(\phi^{k+1}\right) - \rho\left(\phi^{k}\right)}{\delta t}\tvi^k\, , w_i\right) 
	+ \frac{1}{2}\left(\pd{\left(\trhok \tvj^k \right)}{x_j}\tvi^k\, , w_i\right)
	+ \frac{1}{2}\left(\frac{1}{Pe} \pd{\tJj^k}{x_j}\tvi^k\, , w_i\right) = 0
	\label{eqn:consv_added}
	\end{equation}
	to \cref{eqn:nav_stokes_var_semi_disc}.  This is equivalent to adding zero to the equation on the time-discrete continuous-space setting.
	\label{rem:consv_add}
\end{remark}

To analyze the scheme accordingly, we present the following results. 
\begin{lemma}{}{evolution_term_estimate}
	The variational temporal term from the Cahn-Hilliard contribution in the momentum equation, 
	\cref{eqn:nav_stokes_var_semi_disc}, can be written as follows:  
	\begin{equation}
	\begin{split}
	\left(\trhok \left(v^{k+1}_i - v^k_i\right), \tvi^k\right) 
	& + \frac{1}{2} \left(\left(\rho\left(\phi^{k+1}\right) - \rho\left(\phi^{k}\right)\right)\tvi^k, \tvi^k\right) = \\
	&\quad\quad\quad\quad\quad\quad\quad\quad  
	\frac{1}{2}\int_{\Omega}\left[ \rho\left(\phi^{k+1}\right) \norm{{\vec{v}}^{k+1}}^2
	 - \rho\left(\phi^{k}\right)\norm{\vec{v}^{k}}^2 \right]  \, d\vec{x} \\
	&\quad\quad\quad\quad\quad\quad\quad
	- \frac{1}{8}\int_{\Omega}\left[ \rho\left(\phi^{k+1}\right) - \rho\left(\phi^{k}\right) \right] \norm{\vec{v}^{k+1} - \vec{v}^{k} }^2 \, d\vec{x}
	\end{split}
	\label{eqn:evolution_term_lemma}
	\end{equation}
	$\forall \;\; \tphi^k$, $\phi^{k}$, $\phi^{k+1} \in  H^1(\Omega)$, and $\forall \;\; {\vec{v}}^k, {\vec{v}}^{k+1} \in \vec{H}_{0}^1(\Omega)$, where  ${\vec{v}}^k, {\vec{v}}^{k+1}, \phi^k, \phi^{k+1}$ satisfy \cref{eqn:nav_stokes_var_semi_disc} -- \cref{eqn:phi_eqn_var_semi_disc},
	and
	\begin{equation}
	\begin{split}
	\norm{\vec{v}}^2 &:= \sum_i \abs{v_i}^2.
	\end{split}
	\end{equation}
\end{lemma}
\begin{proof} {} {}
We start with the left-hand side of \cref{eqn:evolution_term_lemma}:
	\begin{equation}
	\begin{split}
	 \left(\trhok \left(v^{k+1}_i - v^k_i\right), \tvi^k\right) &= 
	  \frac{1}{2}\int_{\Omega}\left[ \trhok \norm{{\vec{v}}^{k+1}}^2 - \trhok \norm{\vec{v}^{k}}^2 \right]  \, d\vec{x} \\
       &=\frac{1}{4}\int_{\Omega}\left[ \rho\left(\phi^{k+1}\right)\norm{{\vec{v}}^{k+1}}^2 - \rho\left(\phi^{k+1}\right)\norm{\vec{v}^{k}}^2 \right]  \, d\vec{x}\\
	 &+ \frac{1}{4}\int_{\Omega}\left[ \rho\left(\phi^{k}\right)\norm{{\vec{v}}^{k+1}}^2 - \rho\left(\phi^{k}\right)\norm{\vec{v}^{k}}^2 \right]  \, d\vec{x},
	\end{split}
	\label{eqn:evol_inner_prod}
	\end{equation}
	where use the definitions~\cref{eqn:averaging1} and~\cref{eqn:psi_ave_def} and the fact that $\rho$ is an affine function of $\phi$. Continuing the algebraic manipulations we obtain:
	\begin{align}
	\begin{split}
	\label{eqn:evol_inner_prod1}
	\left(\trhok \left(v^{k+1}_i - v^k_i\right), \tvi^k\right) &= 
	 \frac{1}{4}\int_{\Omega}\left[ \rho\left(\phi^{k+1}\right)\norm{{\vec{v}}^{k+1}}^2 - \rho\left(\phi^{k+1}\right)\norm{\vec{v}^{k}}^2 \right]  \, d\vec{x}\\
	&+ \frac{1}{4}\int_{\Omega}\left[ \rho\left(\phi^{k}\right)\norm{{\vec{v}}^{k+1}}^2 - \rho\left(\phi^{k}\right)\norm{\vec{v}^{k}}^2 \right]  \, d\vec{x}\\
	&+ \frac{1}{4}\int_{\Omega}\left[ \rho\left(\phi^{k}\right)\norm{{\vec{v}}^{k}}^2 - \rho\left(\phi^{k}\right)\norm{\vec{v}^{k}}^2 \right]  \, d\vec{x}\\
	&+ \frac{1}{4}\int_{\Omega}\left[ \rho\left(\phi^{k+1}\right)\norm{{\vec{v}}^{k+1}}^2 - \rho\left(\phi^{k+1}\right)\norm{\vec{v}^{k+1}}^2 \right]  \, d\vec{x},
	\end{split}\\
	\begin{split}
	\label{eqn:evol_inner_prod2}
	\implies \left(\trhok \left(v^{k+1}_i - v^k_i\right), \tvi^k\right) &=   
	\frac{1}{2}\int_{\Omega}\left[ \rho\left(\phi^{k+1}\right)\norm{{\vec{v}}^{k+1}}^2 - \rho\left(\phi^{k}\right)\norm{\vec{v}^{k}}^2 \right]  \, d\vec{x}\\
	&- \frac{1}{4}\int_{\Omega}\left[ \rho\left(\phi^{k+1}\right) - \rho\left(\phi^{k}\right) \right] \norm{\vec{v}^{k}}^2 \, d\vec{x}\\
	&- \frac{1}{4}\int_{\Omega}\left[ \rho\left(\phi^{k+1}\right) - \rho\left(\phi^{k}\right) \right] \norm{\vec{v}^{k+1}}^2 \, d\vec{x}.
	\end{split}
	\end{align}
Here we added $\frac{1}{4}\int_{\Omega}\left[ \rho\left(\phi^{k}\right)\norm{{\vec{v}}^{k}}^2 - \rho\left(\phi^{k}\right)\norm{\vec{v}^{k}}^2 \right]  \, d\vec{x}$ and $\frac{1}{4}\int_{\Omega}\left[ \rho\left(\phi^{k+1}\right)\norm{{\vec{v}}^{k+1}}^2 - \rho\left(\phi^{k+1}\right)\norm{\vec{v}^{k+1}}^2 \right]  \, d\vec{x}$, both of which are zero.

Now adding the second term,
\begin{align}
\begin{split}
\label{eqn:evol_inner_prod1_secadded}
\left(\trhok \left(v^{k+1}_i - v^k_i\right), \tvi^k\right)
&+ \frac{1}{2} \left(\left(\rho\left(\phi^{k+1}\right) - \rho\left(\phi^{k}\right)\right)\tvi^k, \tvi^k\right) 
= \\
&\quad\quad\quad\quad\quad\quad\quad\quad\frac{1}{2}\int_{\Omega}\left[ \rho\left(\phi^{k+1}\right)\norm{{\vec{v}}^{k+1}}^2 - \rho\left(\phi^{k}\right)\norm{\vec{v}^{k}}^2 \right]  \, d\vec{x}\\
&\quad\quad\quad\quad\quad\quad\quad- \frac{1}{4}\int_{\Omega}\left[ \rho\left(\phi^{k+1}\right) - \rho\left(\phi^{k}\right) \right] \norm{\vec{v}^{k}}^2 \, d\vec{x}\\
&\quad\quad\quad\quad\quad\quad\quad- \frac{1}{4}\int_{\Omega}\left[ \rho\left(\phi^{k+1}\right) - \rho\left(\phi^{k}\right) \right] \norm{\vec{v}^{k+1}}^2 \, d\vec{x}\\
&\quad\quad\quad\quad\quad\quad\quad+ \frac{1}{8}\int_{\Omega}\left[ \rho\left(\phi^{k+1}\right) - \rho\left(\phi^{k}\right) \right] \norm{\vec{v}^{k+1}}^2 \, d\vec{x}\\
&\quad\quad\quad\quad\quad\quad\quad+ \frac{1}{8}\int_{\Omega}\left[ \rho\left(\phi^{k+1}\right) - \rho\left(\phi^{k}\right) \right] \norm{\vec{v}^{k}}^2 \, d\vec{x}\\
&\quad\quad\quad\quad\quad\quad\quad+ \frac{1}{8}\int_{\Omega}\left[ \rho\left(\phi^{k+1}\right) - \rho\left(\phi^{k}\right) \right] 2 v_i^{k+1}v_i^{k} \, d\vec{x},\\
\end{split}
\end{align}
Here to expand $\tvi^{k} \tvi^{k}$, we utilized the identity $a^2 + b^2 + 2ab = (a + b)^2$, where $a = v_i^{k+1}$ and $b = v_i^{k}$; also recall that $\tvi^{k} = \left(v_i^{k+1} + v_i^{k}\right)/2$. After cancellations and simplification, 
\begin{align}
\begin{split}
\label{eqn:evol_inner_prod1_secadded2}
\left(\trhok \left(v^{k+1}_i - v^k_i\right), \tvi^k\right)
&+ \frac{1}{2} \left(\left(\rho\left(\phi^{k+1}\right) - \rho\left(\phi^{k}\right)\right)\tvi^k, \tvi^k\right) 
= \\
&\quad\quad\quad\quad\quad\quad\quad\quad\frac{1}{2}\int_{\Omega}\left[ \rho\left(\phi^{k+1}\right)\norm{{\vec{v}}^{k+1}}^2 - \rho\left(\phi^{k}\right)\norm{\vec{v}^{k}}^2 \right]  \, d\vec{x}\\
&\quad\quad\quad\quad\quad\quad\quad- \frac{1}{8}\int_{\Omega}\left[ \rho\left(\phi^{k+1}\right) - \rho\left(\phi^{k}\right) \right] \norm{\vec{v}^{k+1}}^2 \, d\vec{x}\\
&\quad\quad\quad\quad\quad\quad\quad- \frac{1}{8}\int_{\Omega}\left[ \rho\left(\phi^{k+1}\right) - \rho\left(\phi^{k}\right) \right] \norm{\vec{v}^{k}}^2 \, d\vec{x}\\
&\quad\quad\quad\quad\quad\quad\quad+ \frac{1}{8}\int_{\Omega}\left[ \rho\left(\phi^{k+1}\right) - \rho\left(\phi^{k}\right) \right] 2 v_i^{k+1}v_i^{k} \, d\vec{x},\\
\end{split}
\end{align}
Now, for the last three terms with the coefficient of $1/8$, we again utilize the identity $a^2 + b^2 - 2ab = (a - b)^2$ to complete the square, where $a = v_i^{k+1}$ and $b = v_i^{k}$. Performing this step we get,
\begin{align}
\begin{split}
\label{eqn:evol_inner_prod1_secadded3}
\left(\trhok \left(v^{k+1}_i - v^k_i\right), \tvi^k\right)
&+ \frac{1}{2} \left(\left(\rho\left(\phi^{k+1}\right) - \rho\left(\phi^{k}\right)\right)\tvi^k, \tvi^k\right) 
= \\
&\quad\quad\quad\quad\quad\quad\quad\quad\frac{1}{2}\int_{\Omega}\left[ \rho\left(\phi^{k+1}\right)\norm{{\vec{v}}^{k+1}}^2 - \rho\left(\phi^{k}\right)\norm{\vec{v}^{k}}^2 \right]  \, d\vec{x}\\
&\quad\quad\quad\quad\quad\quad\quad- \frac{1}{8}\int_{\Omega}\left[ \rho\left(\phi^{k+1}\right) - \rho\left(\phi^{k}\right) \right] \norm{\vec{v}^{k+1} - \vec{v}^{k} }^2 \, d\vec{x},\\
\end{split}
\end{align}
as desired.
\end{proof}

Recognizing that density is an affine transform of $\phi$: $\rho\left(\phi\right) = \alpha \phi + \beta$. Further, using Lipschitz continuity for $\phi$ and velocity we can obtain the following bound. 
\begin{lemma}{} {temporal_correction_estimate}
	The following estimate holds:
	\begin{align}
	\left| 
	\frac{1}{8}\int_{\Omega}\left[ \rho\left(\phi^{k+1}\right) - \rho\left(\phi^{k}\right) \right] \norm{\vec{v}^{k+1} - \vec{v}^{k} }^2 \, d\vec{x} \right| 
	\le \frac{1}{8}C_m \, \alpha L_{\phi,\text{max}} \, L_{\vec{v},\text{max}}^2 \delta t^3,
	\label{eqn:temp_estimate_corr_cl}
	\end{align}
	where $L_{\phi,\text{max}}$, $L_{\vec{v},\text{max}}$ are global maxima of Lipschitz constants and $C_m$ is the volume of the physical domain:
	\begin{gather*}
	C_m := \int_{\Omega} \mathrm{d}\vec{x},
	\,\,\,  L_{\phi,\text{max}} := \max_{\vec{x} \in \Omega} \left(L_{\phi}\left(\vec{x}\right)\right), \,\,\, L_{\vec{v},\text{max}} := \max_{\vec{x} \in \Omega} \left(L_{\vec{v}}\left(\vec{x}\right)\right), \,\,\,
	\left| \phi^{k+1}\left(\vec{x}\right) - \phi^k\left(\vec{x}\right) \right| \le L_{\phi}\left(\vec{x}\right) \delta t, \\
	\left| \rho\left(\phi^{k+1}\right) - \rho\left(\phi^{k}\right) \right| \le \alpha L_{\phi}\left(\vec{x}\right) \, \delta t, \,\,\, \text{and} \,\,\,
	\norm{\vec{v\left(\vec{x}\right)}^{k+1} - \vec{v\left(\vec{x}\right)}^{k} } \le L_{\vec{v}}\left(\vec{x}\right) \, \delta t,
	\end{gather*}
	and $\alpha$ is defined by \cref{eqn:alpha_beta_defn}.
\end{lemma}
\begin{lemma}{}{advection_estimate}
The variational advection term from the Cahn-Hilliard contribution in the momentum equation,~\cref{eqn:nav_stokes_var_semi_disc}, becomes:  
	\begin{equation}
	\begin{split}
	\frac{\delta t}{WeCn}\left(\tphi^k\tvi^k,\pd{\tmu^{k}}{x_i}\right) & = 
	-\frac{1}{2}\int_{\Omega}\left[ \rho\left(\phi^{k+1}\right)\norm{{\vec{v}}^{k+1}}^2 - \rho\left(\phi^{k}\right)\norm{\vec{v}^{k}}^2 \right]  \, d\vec{x}  \\
	&+ \frac{1}{8}\int_{\Omega}\left[ \rho\left(\phi^{k+1}\right) - \rho\left(\phi^{k}\right) \right] \norm{\vec{v}^{k+1} - \vec{v}^{k} }^2 \, d\vec{x}\\
	&- \frac{ \delta t}{Re} \norm{\sqrt{\tetak} \, \nabla \widetilde{\vec{v}}^k}_{L^2}^2  
	 - \frac{1}{Fr}\left(y,  \,  \rho\left(\phi^{k+1}\right) - \rho\left(\phi^{k}\right)  \right),
	\end{split}
	\label{eqn:cahn_hilliard_term_lemma}
	\end{equation}
	$\forall \;\; \tphi^k$, $\phi^{k+1}$, $\tmu^{k} \in  H^1(\Omega)$, and $\forall \;\; {\vec{v}}^k, {\vec{v}}^{k+1} \in \vec{H}_{0}^1(\Omega)$, where  ${\vec{v}}^k, {\vec{v}}^{k+1}, p^k, p^{k+1}, \phi^k, \phi^{k+1}, \mu^{k},\mu^{k+1}$ satisfy \cref{eqn:nav_stokes_var_semi_disc} -- \cref{eqn:phi_eqn_var_semi_disc},
	and
	\begin{equation}
	\begin{split}
	    \norm{\vec{v}}^2 &:= \sum_i \abs{v_i}^2, \\
	    \qquad \norm{\sqrt{\tetak} \, \nabla \widetilde{\vec{v}}^k}_{L^2}^2 &:= 
	\int_{\Omega} \sqrt{\tetak} \, \sum_i \sum_j \abs{\frac{\partial \tvi^k}{\partial x_j}}^2 \, d\vec{x} = 
	\int_{\Omega} \sqrt{\tetak} \, \norm{\nabla\vec{v}}^2_F \, d\vec{x}.
	\end{split}
	\end{equation}
\end{lemma}
\begin{proof} {}{}
The structure of this proof is similar to one presented in~\citet{Khanwale2020} with some changes.  However, for completeness we produce the  proof below. 
We start with momentum equation \cref{eqn:nav_stokes_var_semi_disc} added with \cref{eqn:consv_added} from \cref{rem:consv_add} using the test function $w_i=\delta t \, \tvi^k$:
	\begin{equation}
	\begin{split}
	&\left(\trhok \frac{v^{k+1}_i - v^k_i}{\delta t}, \delta t \, \tvi^k\right) + \left(\trhok \, \tvj^{k} \, \pd{\tvi^{k}}{x_j}, \delta t \, \tvi^k\right) + 
	\frac{1}{Pe}\left(\tJj^{k}\pd{\tvi^{k}}{x_j}, \delta t \, \tvi^k\right)\\
	&\frac{1}{2} \left(\frac{\rho\left(\phi^{k+1}\right) - \rho\left(\phi^{k}\right)}{\delta t}\tvi^k\,, \delta t \, \tvi^k\right) 
	+ \frac{1}{2}\left(\pd{\left(\trhok \tvj^k \right)}{x_j}\tvi^k\,, \delta t \, \tvi^k\right)
	+ \frac{1}{2}\left(\frac{1}{Pe} \pd{\tJj^k}{x_j}\tvi^k\,, \delta t \, \tvi^k\right)
	\\ + &\frac{Cn}{We}\left( \pd{}{x_j}\left({\pd{\tphi^{k}}{x_i} \, \pd{\tphi^{k}}{x_j}}\right), \delta t  \,\tvi^k\right) + 
	\frac{1}{We}\left(\pd{\tp^{k}}{x_i}, \delta t \, \tvi^k\right) - \frac{1}{Re}\left(\pd{}{x_j}\left(\tetak {\pd{\tvi^{k}}{x_j}}\right), \delta t \, \tvi^k\right)
	\\- &\frac{1}{Fr}\left(\trhok \hat{g_i}, \delta t \, \tvi^k\right) = 0. 
	\label{eqn:disc_nav_stokes_inner_prod}
	\end{split}
	\end{equation}
	The second and third terms together with fifth and sixth terms are in a trilinear form so from~\cref{eqn:trilinear_zero_vel} and~\cref{eqn:trilinear_zero_massflux} they go to zero and we have:
	\begin{align}
	\begin{split}
	&\frac{1}{2}\int_{\Omega}\left[ \rho\left(\phi^{k+1}\right) \norm{{\vec{v}}^{k+1}}^2 - \rho\left(\phi^{k}\right) \norm{\vec{v}^{k}}^2 \right]  \, d\vec{x}  
	- \frac{1}{8}\int_{\Omega}\left[ \rho\left(\phi^{k+1}\right) - \rho\left(\phi^{k}\right) \right] \norm{\vec{v}^{k+1} - \vec{v}^{k} }^2 \, d\vec{x}
	\\&
	+\frac{Cn}{We}\left( \pd{}{x_j}\left({\pd{\tphi^{k}}{x_i}\pd{\tphi^{k}}{x_j}}\right), \delta t \,\tvi^k\right) \\ &  + 
	\frac{1}{We}\left(\pd{\tp^{k}}{x_i}, \delta t \, \tvi^k\right) 
	- \frac{1}{Re} \left(\pd{}{x_j}\left(\tetak {\pd{\tvi^{k}}{x_j}}\right), \delta t \, \tvi^k\right)  -
	 \frac{1}{Fr}\left(\hat{g_i}, \delta t \, \trhok \tvi^k\right)= 0,
	\end{split}
	\label{eqn:disc_nav_stokes_inner_prod_no_advec}
	\end{align}
		where we made use of the fact that $\tvi^k = (v_i^{k+1}+v_i^k)/2$ and subsequently~\cref{lem:evolution_term_estimate}.
	We can now use solenoidality of the velocity field to get rid of the pressure term.  We can do this by integrating-by-parts on the pressure term:
	\begin{align}
	\begin{split}	
	& \quad \frac{1}{2}\int_{\Omega}\left[ \rho\left(\phi^{k+1}\right)\norm{{\vec{v}}^{k+1}}^2 - \rho\left(\phi^{k}\right)\norm{\vec{v}^{k}}^2 \right]  \, d\vec{x}
	- \frac{1}{8}\int_{\Omega}\left[ \rho\left(\phi^{k+1}\right) - \rho\left(\phi^{k}\right) \right] \norm{\vec{v}^{k+1} - \vec{v}^{k} }^2 \, d\vec{x}
	\\&
	\qquad  + \frac{Cn}{We}\left( \pd{}{x_j}\left({\pd{\tphi^{k}}{x_i}\pd{\tphi^{k}}{x_j}}\right), \delta t \,\tvi^k\right) \\ & \qquad  
	-\frac{\delta t}{We}\left(\tp^{k}, \pd{\tvi^k}{x_i}\right)  
	-\frac{1}{Re} \left(\pd{}{x_j}\left(\tetak {\pd{\tvi^{k}}{x_j}}\right), \delta t \, \tvi^k\right) 
	-\frac{1}{Fr}\left(\hat{g_i}, \delta t \, \trhok \tvi^k\right) = 0, \\
	\end{split} \\
	\displaybreak[0]
	\begin{split}
	\implies \quad 
	& \quad \frac{1}{2}\int_{\Omega}\left[ \rho\left(\phi^{k+1}\right)\norm{{\vec{v}}^{k+1}}^2 - \rho\left(\phi^{k}\right)\norm{\vec{v}^{k}}^2 \right]  \, d\vec{x}
	- \frac{1}{8}\int_{\Omega}\left[ \rho\left(\phi^{k+1}\right) - \rho\left(\phi^{k}\right) \right] \norm{\vec{v}^{k+1} - \vec{v}^{k} }^2 \, d\vec{x}
	\\& 
	\qquad  +\frac{Cn}{We}\left( \pd{}{x_j}\left({\pd{\tphi^{k}}{x_i}\pd{\tphi^{k}}{x_j}}\right), \delta t \, \tvi^k\right)  
	-\frac{1}{Re} \left(\pd{}{x_j}\left(\tetak{\pd{\tvi^{k}}{x_j}}\right), \delta t \, \tvi^k\right) 
	-\frac{1}{Fr}\left(\hat{g_i}, \delta t \, \trhok \tvi^k\right) = 0,\label{eqn:disc_nav_stokes_inner_prod_no_pres} 
	\end{split} \\
	\begin{split}
	\implies \quad 
	& \quad \frac{1}{2}\int_{\Omega}\left[ \rho\left(\phi^{k+1}\right)\norm{{\vec{v}}^{k+1}}^2 - \rho\left(\phi^{k}\right)\norm{\vec{v}^{k}}^2 \right]  \, d\vec{x} 
	- \frac{1}{8}\int_{\Omega}\left[ \rho\left(\phi^{k+1}\right) - \rho\left(\phi^{k}\right) \right] \norm{\vec{v}^{k+1} - \vec{v}^{k} }^2 \, d\vec{x}
	\\& 
	\qquad +\frac{Cn}{We}\left( \pd{}{x_j}\left({\pd{\tphi^{k}}{x_i}\pd{\tphi^{k}}{x_j}}\right), \delta t \,\tvi^k\right) 
	+ \frac{\delta t}{Re} \left( \sqrt{\tetak} \, {\pd{\tvi^{k}}{x_j}}, \sqrt{\tetak} \, {\pd{\tvi^{k}}{x_j}}\right) 
	 -\frac{1}{Fr}\left(\hat{g_i}, \delta t \, \trhok \tvi^k\right) = 0,\label{eqn:disc_nav_stokes_inner_prod_no_pres_weak}
	\end{split} \\
	\begin{split}
	\implies \quad 
	& \quad \frac{1}{2}\int_{\Omega}\left[ \rho\left(\phi^{k+1}\right)\norm{{\vec{v}}^{k+1}}^2 - \rho\left(\phi^{k}\right)\norm{\vec{v}^{k}}^2 \right]  \, d\vec{x}
	- \frac{1}{8}\int_{\Omega}\left[ \rho\left(\phi^{k+1}\right) - \rho\left(\phi^{k}\right) \right] \norm{\vec{v}^{k+1} - \vec{v}^{k} }^2 \, d\vec{x}
	\\& 
	\qquad + \frac{Cn}{We}\left( \pd{}{x_j}\left({\pd{\tphi^{k}}{x_i}\pd{\tphi^{k}}{x_j}}\right), \delta t \,\tvi^k\right) 
	+ \frac{\delta t}{Re} \norm{\sqrt{\tetak} \, \nabla \widetilde{\vec{v}}^k}^2_{L^2}
	 -\frac{1}{Fr}\left(\hat{g_i}, \delta t \, \trhok \tvi^k\right) = 0.
	\label{eqn:disc_nav_stokes_inner_prod_no_pres_weak_norm}
	\end{split}
	\end{align}
	Next, we invoke~\cref{lem:forcing_equivalent} and write~\cref{eqn:disc_nav_stokes_inner_prod_no_pres_weak_norm} as:
	\begin{align}
	\begin{split}
	&\frac{1}{2}\int_{\Omega}\left[ \rho\left(\phi^{k+1}\right)\norm{{\vec{v}}^{k+1}}^2 - \rho\left(\phi^{k}\right)\norm{\vec{v}^{k}}^2 \right]  \, d\vec{x}
	- \frac{1}{8}\int_{\Omega}\left[ \rho\left(\phi^{k+1}\right) - \rho\left(\phi^{k}\right) \right] \norm{\vec{v}^{k+1} - \vec{v}^{k} }^2 \, d\vec{x}
	\\& \quad\quad\quad\quad\quad\quad
	+ \frac{\delta t}{WeCn}\left(\tphi^k\tvi^k,\pd{\tmu^{k}}{x_i}\right) 
	 + \frac{\delta t}{Re} \norm{\sqrt{\tetak} \, \nabla \widetilde{\vec{v}}^k}^2_{L^2} 
	- \frac{1}{Fr}\left(\hat{g_i}, \delta t \, \trhok \, \tvi^k\right) = 0.
	\end{split}
	\label{eqn:disc_nav_stokes_eqv_forcing_ch_without_grav} 
	\end{align}
	Next we simplify the gravity term noting that
	\begin{align}
		-\frac{1}{Fr}\left(\hat{g_i}, \delta t \, \trhok \, \tvi^k\right) = -\frac{1}{Fr}\left(\pd{\left(-y\right)}{x_i}, \delta t \, \trhok \, \tvi^k\right) = 
		-\frac{1}{Fr}\left(y, \delta t \, \pd{ \left( \trhok \tvi^k \right)}{x_i}\right),
	\label{eqn:lemma2_estimate_without_gravSimple}
	\end{align}
	where $y = x_2$ and $\hat{\vec{g}} = (0, -1, 0)$. 
	Note that the boundary terms vanish
	in the process of integrating-by-parts due to the fact that ${\tvecv}^{k+1} \in \vec{H}_{0}^1(\Omega)$.
	Let $C_1 = \frac{\left(\rho_- - \rho_+\right)}{2\rho_+ Cn} \, m(\phi)$, then using 
	\cref{eqn:thermo_consistency_semi_disc}, 
	the fact that $\rho$ is affine in $\phi$,
    \cref{eqn:phi_eqn_var_semi_disc}, and 
	\cref{eqn:lemma2_estimate_without_gravSimple} we obtain:
	\begin{align}
	\begin{split}
	&-\frac{1}{Fr}\left(\hat{g_i}, \delta t \, \trhok \, \tvi^k\right)
	 = -\frac{1}{Fr}\left(y,  \,  \rho\left(\phi^{k+1}\right) - \rho\left(\phi^{k}\right)  \right)  
	-\frac{\delta t \, C_1}{Fr\,Pe}\left(y,  \, \pd{}{x_i} \left( \pd{\tmu^{k}}{x_i} \right)\right)\\
	&= -\frac{1}{Fr}\left(y,  \,  \rho\left(\phi^{k+1}\right) - \rho\left(\phi^{k}\right) \right)   
	+\frac{\delta t \, C_1}{Fr\,Pe}\left(\pd{y}{x_i} ,  \, \pd{\tmu^{k}}{x_i} \right)\\
	&= -\frac{1}{Fr}\left(y,  \,  \rho\left(\phi^{k+1}\right) - \rho\left(\phi^{k}\right)  \right)  
	- \frac{\delta t \, C_1}{Fr\,Pe}\left( \pd{}{x_i}\left(\pd{y}{x_i}\right) ,  \, \tmu^{k} \right) 
	 + \frac{\delta t \, C_1}{Fr\,Pe} \int_{\mathrm{d}\Omega}\tmu^{k}\left(\pd{y}{x_i}\right) \hat{n_i} \mathrm{d}\vec{x}\\
	&= -\frac{1}{Fr}\left(y,  \,  \rho\left(\phi^{k+1}\right) - \rho\left(\phi^{k}\right) \right)   
	  + \frac{\delta t \, C_1}{Fr\,Pe} \int_{\mathrm{d}\Omega}\tmu^{k}\hat{g_i} \hat{n_i} \mathrm{d}\vec{x} \\
	  &= -\frac{1}{Fr}\left(y,  \,  \rho\left(\phi^{k+1}\right) - \rho\left(\phi^{k}\right) \right),
	\end{split}
	\label{eqn:grav_stuff}
	\end{align} 
where $\hat{n_i}$ is outward normal to the boundary of the domain $\Omega$.  
	\begin{remark}
		In the last line of \eqref{eqn:grav_stuff} we asserted that 
		\begin{equation}
		\frac{\delta t C_1}{Fr\,Pe} \int_{\mathrm{d}\Omega}\tmu^{k}\hat{g_i} \hat{n_i} \mathrm{d}\vec{x} = 0.
		\end{equation}
		This is true as long as there is no three-phase contact line on any boundary on which $\hat{n_i} \hat{g_i}$ is non-zero. For the purpose of the
		analysis presented here, we will assume that this is true.
	\end{remark} 
	Combining \eqref{eqn:grav_stuff} with \cref{eqn:disc_nav_stokes_eqv_forcing_ch_without_grav} yields the desired result:
	\begin{align}
	\begin{split}
	\frac{1}{2}
	&\int_{\Omega}\left[ \rho\left(\phi^{k+1}\right)\norm{{\vec{v}}^{k+1}}^2 - \rho\left(\phi^{k}\right)\norm{\vec{v}^{k}}^2 \right]  \, d\vec{x}
	- \frac{1}{8}\int_{\Omega}\left[ \rho\left(\phi^{k+1}\right) - \rho\left(\phi^{k}\right) \right] \norm{\vec{v}^{k+1} - \vec{v}^{k} }^2 \, d\vec{x}  
	 \,\\
	&+ \frac{\delta t}{WeCn}\left(\tphi^k\tvi^k,\pd{\tmu^{k}}{x_i}\right) 
	+ \frac{\delta t}{Re} \norm{\sqrt{\tetak} \, \nabla \widetilde{\vec{v}}}^2_{L^2} 
	+ \, \frac{1}{Fr}\left(y,  \,  \rho\left(\phi^{k+1}\right) - \rho\left(\phi^{k}\right)  \right) = 0.
	\end{split}
	\label{eqn:disc_nav_stokes_eqv_forcing_ch} 
	\end{align}	
\end{proof}

\vspace{2cm}
We now have all the ingredients to prove energy-stability.  Our argument uses the fact that the energy functional~\cref{eqn:energy_functional} is decreasing as the discrete solution is evolving in time.  At the semi-discrete level, the successive decrease of the energy functional for each time step represents adherence to the second law of thermodynamics.  
We prove energy-stability in the following theorem. 
		
\begin{theorem}{energy-stability}{energy_stability}
		The time discretization of the Cahn-Hilliard Navier-Stokes (CHNS) equations as described by~\cref{eqn:nav_stokes_var_semi_disc} -- \cref{eqn:phi_eqn_var_semi_disc} satisfies the following energy law:
		\begin{equation}
		\begin{split}
		E_{tot}\left(\vec{v}^{k+1},\phi^{k+1} \right) - E_{tot}\left(\vec{v}^k,\phi^k\right) &= \frac{-\delta t}{Re} \norm{\sqrt{\tetak} \, \widetilde{\vec{v}}^k}^2_{L^2}
		 - \frac{\delta t}{Pe Cn^2We} \norm{\nabla \widetilde{\mu}^k}^2_{L^2} \\ &+ \frac{1}{24} \frac{1}{WeCn}\left( \d{^3\psi}{\phi^3}\Biggl|_{\phi=\lambda\left(\vec{x}\right)}, \left(\phi^{k+1} - \phi^k\right)^3\right)\\
			&+ \frac{1}{8}\int_{\Omega}\left[ \rho\left(\phi^{k+1}\right) - \rho\left(\phi^{k}\right) \right] \norm{\vec{v}^{k+1} - \vec{v}^{k} }^2 \, d\vec{x},
		\label{energy_law}
		\end{split}
		\end{equation}	
		for some $\lambda\left(\vec{x}\right)$ between $\phi^k\left(\vec{x}\right)$ and $\phi^{k+1}\left(\vec{x}\right)$. The time discretization is energy-stable in the following sense:
		\begin{equation}
		\label{eqn:energy_non_increasing}
		E_{tot}\left(\vec{v}^{k+1},\phi^{k+1} \right) \le E_{tot}\left(\vec{v}^k,\phi^k\right),
		\end{equation}
		provided that the following time-step restriction is observed:
		\begin{equation}
				0 \le \delta t \leq 
		\left(\frac{\frac{1}{Re}\left(\norm{\sqrt{\tetak} \, \nabla \widetilde{\vec{v}}^k}_{L^2}^2\right) + 
			\frac{1}{Pe Cn^2We}\norm{\nabla \widetilde{\mu}^k}_{L^2}^2}
		{\frac{C_m \, L_{\phi,\text{max}}^3 \, P_{\text{max}}}{WeCn}
	 + \frac{C_m \, \alpha L_{\phi,\text{max}} \, L_{\vec{v},\text{max}}}{8}}\right)^{\frac{1}{2}}. \label{eqn:timestep_condition_thr}
		\end{equation}	
\end{theorem}   
	
\begin{proof}
The proof uses $L^2$ estimates of the semi-discrete equations to estimate the energy change between two-time steps.  If the estimate is strictly negative, we have an energy-stable scheme.  We begin with~\cref{eqn:phi_eqn_var_semi_disc} with the test function $q=\delta t \, \tmu^k$: 
		\begin{equation}
		\begin{split}
		\left(\phi^{k+1} - \phi^k, \tmu^k\right) &= 
		  \left( \tvi^{k} \tphi^{k} , \delta t \, \pd{\tmu^k}{x_i} \right) - \frac{\delta t}{PeCn} \norm{\nabla \widetilde{\mu}^k}^2_{L^2}.	
		\label{eqn:phi_eqn_simple}
		\end{split}
		\end{equation}
Next, we take~\cref{eqn:mu_eqn_var_semi_disc} with test function $q=\phi^{k+1} - \phi^k$, where~\cref{eqn:psi_ave_def} defines $\tpsi'$:
		\begin{align}
		 \left(\tmu^{k} , \phi^{k+1} - \phi^k\right) = \left(\tpsi', \phi^{k+1} - \phi^k\right) + \frac{Cn^2}{2} \left(\norm{\nabla \phi^{k+1}}^2_{L^2} - \norm{\nabla \phi^{k}}^2_{L^2}\right),
		\label{eqn:mu_eq_unsimplified}
		\end{align}
where we also use the fact that $\tphi^{k} = (\phi^{k+1} + \phi^k)/2$. The first term on right-hand side of~\cref{eqn:mu_eq_unsimplified} simplifies further using~\cref{lem:free_en_taylor}:
		\begin{equation}
		\begin{split}
		\left(\tmu^{k} , \phi^{k+1} - \phi^k\right) &= \left(\psi(\phi^{k+1}) - \psi(\phi^{k}), 1\right) - \frac{1}{24} \left(\d{^3\psi}{\phi^3}\Biggl|_{\phi=\lambda\left(\vec{x}\right)},
		\left(\phi^{k+1} - \phi^k\right)^3\right)  \\
		&+ \frac{Cn^2}{2} \left(\norm{\nabla \phi^{k+1}}^2_{L^2} - \norm{\nabla \phi^{k}}^2_{L^2}\right).
		\label{eqn:mu_eq_simple}
		\end{split}
		\end{equation}
		 Now, combining~\cref{eqn:mu_eq_simple} and~\cref{eqn:phi_eqn_simple}, we have: 
		\begin{equation}
		\begin{split}
		 \left(\psi(\phi^{k+1}) - \psi(\phi^{k}), 1\right) &- 
		\frac{1}{24} \left(\d{^3\psi}{\phi^3}\Biggl|_{\phi=\lambda\left(\vec{x}\right)},
		\left(\phi^{k+1} - \phi^k\right)^3\right) +  
		\frac{Cn^2}{2} \left(\norm{\nabla \phi^{k+1}}^2_{L^2} - \norm{\nabla \phi^{k}}^2_{L^2}\right) \\ &= 
		\left(\tvi^{k} \tphi^{k}, \delta t \pd{\tmu^k}{x_i}\right) - \frac{\delta t}{PeCn} \norm{\nabla \widetilde{\vec{\mu}}^k}^2_{L^2}.
		\end{split}
		\label{eqn:ch_eq_simple}
		\end{equation}
		Next, we divide~\cref{eqn:ch_eq_simple} by $We Cn$ and from~\cref{lem:advection_estimate}, we can substitute the first term on the right-hand side by~\cref{eqn:cahn_hilliard_term_lemma}:
		\begin{equation}
		\begin{split}
		& \frac{1}{2}\int_{\Omega}\left[\rho\left(\phi^{k+1}\right)\norm{{\vec{v}}^{k+1}}^2 - \rho\left(\phi^{k}\right)\norm{\vec{v}^{k}}^2 \right]  \, d\vec{x}
		- \frac{1}{8}\int_{\Omega}\left[ \rho\left(\phi^{k+1}\right) - \rho\left(\phi^{k}\right) \right] \norm{\vec{v}^{k+1} - \vec{v}^{k} }^2 \, d\vec{x}\\
		&\qquad \quad
		+ \frac{\delta t}{Re} \left(\norm{ \sqrt{\tetak} \, \nabla \widetilde{\vec{v}}^k}^2_{L^2}\right) 
		+ \frac{1}{WeCn}\left(\psi(\phi^{k+1}) - \psi(\phi^{k}), 1\right) \\
		& \qquad \quad - \frac{1}{24 \, WeCn}
		 \left(\d{^3\psi}{\phi^3}\Biggl|_{\phi=\lambda\left(\vec{x}\right)},
		\left(\phi^{k+1} - \phi^k\right)^3\right) \\
		& \qquad \quad + \frac{Cn}{2We} \left(\norm{\nabla \phi^{k+1}}^2_{L^2} - \norm{\nabla \phi^{k}}^2_{L^2}\right) 
		+ 
		\frac{\delta t}{PeCn^2We} \norm{\nabla \tmu^{k}}^2_{L^2} 
		\\ & \qquad \quad + \frac{1}{Fr}\left(y,  \,  \rho\left(\phi^{k+1}\right) - \rho\left(\phi^{k}\right) \right)  =0.
        \end{split}
         \label{eqn:disc_nav_stokes_eqv_forcing_com}
		\end{equation}
		Simplifying and using the definition of the energy functional,~\cref{eqn:energy_functional}, 
		we obtain energy law \cref{energy_law}.

\smallskip

        \indent
        The problem with energy law \cref{energy_law} is that the final two terms on the right-hand side are sign indeterminate. 
		Therefore, in order for this energy to be non-increasing in time (i.e., \cref{eqn:energy_non_increasing}), we require the following:
		\begin{equation}
		\begin{split}
		\frac{\delta t}{Re} \norm{\sqrt{\tetak} \, \nabla \widetilde{\vec{v}}^k}_{L^2}^2 + 
		\frac{\delta t}{Pe Cn^2We} \norm{\nabla \tmu^{k}}_{L^2}^2 &\geq 
		\frac{1}{24 \, WeCn} \left(\d{^3\psi}{\phi^3}\Biggl|_{\phi=\lambda\left(\vec{x}\right)},\left(\phi^{k+1} - \phi^k\right)^3\right)
		\\&\;\;
		+ \frac{1}{8}\int_{\Omega}\left[ \rho\left(\phi^{k+1}\right) - \rho\left(\phi^{k}\right) \right] \norm{\vec{v}^{k+1} - \vec{v}^{k} }^2 \, d\vec{x}.
		\end{split}
		\label{eqn:stab_condition}
		\end{equation}
		Using the estimates from~\cref{lem:correction_estimate,lem:temporal_correction_estimate}, we can guarantee this inequality
		provided that:
		\begin{equation}
		\begin{split}
		\frac{\delta t}{Re} \norm{\sqrt{\tetak} \, \nabla \widetilde{\vec{v}}^k}_{L^2}^2 + 
		\frac{\delta t}{Pe Cn^2We} \norm{\nabla \tmu^{k}}_{L^2}^2 &\geq
		\frac{C_m L^3_{\text{max}} P_{\text{max}}}{WeCn}  \delta t^3
		+ \frac{C_m \, \alpha L_{\phi,\text{max}} \, L_{\vec{v},\text{max}}}{8}\delta t^3.
		\end{split}
		\end{equation}
		This condition becomes a condition on the maximum energy-stable time-step size, which is provided by \cref{eqn:timestep_condition_thr}.
\end{proof}
\begin{remark}
	Condition~\cref{eqn:timestep_condition_thr} is a weak condition (satisfied by most $\delta t$), as all the quantities in its statement are order one quantities.      
	While we cannot claim unconditional stability for the scheme, however, the scheme is energy-stable for large range of $\delta t$ values, which in practice allows us to take large time steps (for appropriate accuracy requirements).  A consequence is that we do not have to evaluate this condition at every timestep. This is reflected by numerical experiments we conduct.   
\end{remark}
\begin{remark}
	\textcolor{black}{As an example let us consider the case with a very weak flow (small velocities).  In this case viscous dissipation term is close to zero:
	\[
	\frac{1}{Re} \norm{\sqrt{\tetak} \, \nabla \widetilde{\vec{v}}^k}_{L^2}^2 \approx 0.
	\]
	The other term in the numerator is finite and order one:
	\[
	\frac{1}{Pe Cn^2We} \norm{\nabla \tmu^{k}}_{L^2}^2 \sim {\mathcal O}(1).
	\] 
	Therefore, we have a finite order one numerator. On the other hand, notice that both the terms in the denominator of~\cref{eqn:timestep_condition_thr} are estimates of 
	\[
	\frac{1}{24} \left| \left( \d{^3\psi}{\phi^3}\Biggl|_{\phi=\lambda\left(\vec{x}\right)}, \left(\phi^{k+1}\left(\vec{x}\right) - \phi^k\left(\vec{x}\right)\right)^3 \right) \right|
	\]
	and
	\[ \left| 
	\frac{1}{8}\int_{\Omega}\left[ \rho\left(\phi^{k+1}\right) - \rho\left(\phi^{k}\right) \right] \norm{\vec{v}^{k+1} - \vec{v}^{k} }^2 \, d\vec{x} \right|,
	\] 
	respectively.  	
	In the limit of velocity going to zero the interface wouldn't move much between two time points (as advection is close to zero) and $\phi^{k+1}\left(\vec{x}\right) - \phi^k\left(\vec{x}\right) \rightarrow 0 $, also for very small velocities $\norm{\vec{v}^{k+1} - \vec{v}^{k} }^2 \rightarrow 0$.  Therefore, both the terms in the denominator of~\cref{eqn:timestep_condition_thr} would approach 0, and with the finite order 1 numerator the right hand side of~\cref{eqn:timestep_condition_thr} would actually approach infinity justifying our assertion that we would have stability for most timesteps.\todo[color=cyan]{R1: \#1.2}}
\end{remark}
\begin{remark}
Our estimate for the time step size is modified to the one in~\citep{ Khanwale2020} for unequal densities, while extending the time-scheme to second order.
\end{remark}

\subsection{Solvability of the discrete-in-time, continuous-in-space CHNS system}
\label{subsec:semidiscrete_var_prob}
In this subsection we establish the solvability of \cref{eqn:nav_stokes_var_semi_disc} -- \cref{eqn:phi_eqn_var_semi_disc}.  The basic strategy for proving the existence results follows~\citep{Khanwale2020}.  We summarize the process as follows:

\begin{enumerate}
	\item Breakdown the problem into individual problems corresponding to each equation in \crefrange{eqn:nav_stokes_var_semi_disc}{eqn:phi_eqn_var_semi_disc}.  
	%
	\item Show that given $\widetilde{\mu}^k$, \cref{eqn:mu_eqn_var_semi_disc} uniquely determines $\tphi^{k}$. Subsequently, with knowledge of $\tphi^{k}$ from \cref{eqn:mu_eqn_var_semi_disc}, show that \cref{eqn:nav_stokes_var_semi_disc} -- \cref{eqn:cont_var_semi_disc} uniquely determines $\widetilde{\vec{v}}^k$ and $\widetilde{p}^{k}$. This establishes unique determination of $\tphi^{k}$, $\widetilde{\vec{v}}^k$, and $\widetilde{p}^{k}$ given $\widetilde{\mu}^k$. 
	\item From the above point the problem focus now can shift to~\cref{eqn:phi_eqn_var_semi_disc}, as a scalar equation for $\widetilde{\mu}^k$ which is a non-linear advection-diffusion type equation (due to fully-implicit discretisation).  
	\item Show that there exists a solution, $\widetilde{\mu}^k$, to  \cref{eqn:phi_eqn_var_semi_disc}, with $\tphi^{k}$ and $\widetilde{\vec{v}}^k$ understood
	to be functions of $\widetilde{\mu}^k$ via \crefrange{eqn:nav_stokes_var_semi_disc}{eqn:mu_eqn_var_semi_disc}.
\end{enumerate}

The bulk of the proof for existence focuses on showing that there exists a $\widetilde{\mu}^k$ which solves \cref{eqn:phi_eqn_var_semi_disc}.  This involves technical arguments that are leveraged from non-linear analysis.  The strategy is to use existence of solutions to pseudo-monotone operators from Brezis (see theorem 27.A \citet{Zeidler1985IIb}).  This requires proving that the operator in \cref{eqn:phi_eqn_var_semi_disc} with the setting of a real, reflexive Banach space~\footnote{Sobolev spaces considered in this manuscript are real and reflexive.} is \textit{\textbf{pseudo-monotone, bounded, and coercive}}.  The two main differences in the time-scheme in~\citep{Khanwale2020} and this manuscript are; 1) fully-coupled nature of solving  \crefrange{eqn:nav_stokes_var_semi_disc}{eqn:phi_eqn_var_semi_disc} and Crank-Nicolson averages of mixture density ($\rho$) and mixture viscosity ($\eta$) in \cref{eqn:nav_stokes_var_semi_disc}.  Both these changes do not affect the strategy discussed above in major ways and the proof remains largely unchanged with minor changes.  This is one of the advantages of our scheme here -- the aforementioned technique of proving existence (solvability) of the semi-discrete system is applicable here.  The proof that operator in \cref{eqn:phi_eqn_var_semi_disc} is \textit{\textbf{pseudo-monotone, bounded, and coercive}}  largely remains unchanged from~\citet{Khanwale2020}.  

It is important to note that the technique initially demonstrated by~\citet{Han2015} for linear operators and then modified for non-linear operators in~\citet{Khanwale2020} allows us to show solvability for a fully-nonlinear schemes.  However, proving uniqueness of the solution of \cref{eqn:phi_eqn_var_semi_disc} remains difficult, but for practical situations we do not see any problems. 

\subsection{Spatial discretization and the variational multiscale approach}
\label{subsec:space_scheme}
The unknowns:
\begin{equation}
\left( \phi, \, \mu,  \, \vec{v}, \, p \right),
\end{equation}
are all discretized in space using continuous Galerkin finite elements
with piecewise polynomial approximations. It is well-known from literature 
that approximating the velocity,
 $\vec{v}$, and the pressure, $p$, with the same polynomial order 
 leads to numerical instabilities as this violates the discrete inf-sup condition  or Ladyzhenskaya-Babuska-Brezzi condition (e.g., see \citet[page 31]{ Volker2016}).  A popular method to overcome this difficulty is to add
 stabilization terms to the evolution equations that transform the
 inf-sup stability condition to a coercivity statement \citep{article:TezMitRayShi92}. In particular, a large class of stabilization techniques derive from the so-called {\it variational multiscale} approach (VMS)~\citep{ hughes2018multiscale}, which generalizes the well-known 
 SUPG/PSPG~\citep{ article:BrooksHughes1982} method in the context of large-eddy simulation (LES)~\citep{ Hughes1995}.  
Additionally, VMS provides a natural leeway into modeling high-Reynolds number flows~\citep{ Bazilevs2007}.       
  
The VMS approach uses a direct-sum decomposition of the function spaces as follows.  If $\vec{v} \in \vec{V}$, $p \in Q$, and $\phi \in Q$ then we decompose these spaces as: 
\begin{eqnarray}
\vec{V} = \overline{\vec{V}} \oplus \vec{V}^\prime \qquad \text{and} \qquad Q = \overline{Q} \oplus Q^\prime,
\end{eqnarray}
where $\overline{\vec{V}}$ and $\overline{Q}$ are the cG(r) subspaces of $\vec{V}$ and $Q$, respectively, and the primed versions are the complements of the cG(r) subspaces in $\vec{V}$ and $Q$, respectively.
We decompose the velocity and pressure as follows: 
\begin{equation}
\vec{v} = \overline{\vec{v}} + \vec{v}^{\prime}, \quad
\phi = \overline{\phi} + \phi^{\prime}, \quad \text{and}
\quad p = \overline{p} + p^{\prime},
\end{equation} 
where the {\it coarse scale} components are
  $\overline{\vec{v}} \in \overline{\vec{V}}$ and  $\overline{p}, \overline{\phi}\in \overline{Q}$, and the {\it fine scale} components are
   $\vec{v}^{\prime} \in \vec{V}^\prime$ and $p^{\prime},\phi^{\prime} \in Q^\prime$.  We define a projection operator, $\mathscr{P}:\vec{V} \rightarrow \overline{\vec{V}}$, such that
 $\overline{\vec{v}} = \mathscr{P}\{\vec{v}\}$ and $\vec{v}^\prime = \vec{v} - \mathscr{P}\{\vec{v}\}$.  A similar operator can decompose $p$ and $\phi$.  

Substituting these decompositions in the original variational form  
\cref{defn:variational_form_sem_disc} yields:
\begin{align}	
\begin{split}	
		\text{Momentum Eqns:}& \quad \left(w_i,\rho(\phi)\pd{ \overline{v_i}}{t}\right) + \left(w_i,
		\pd{\left(\rho(\phi) v_i^{\prime}\right)}{t}\right) +  \left(w_i,\rho(\phi)\overline{v_j}\pd{\overline{v_i}}{x_j}\right) \\ & + \left(w_i,\rho(\phi)v_j^\prime\pd{\overline{v_i}}{x_j}\right) +
		\left(w_i,\pd{\left(\rho(\phi)\overline{v_j}v_i^\prime\right)}{x_j}\right) + \left(w_i,\pd{\left(\rho(\phi)v_j^\prime v_i^\prime\right)}{x_j}\right) \\ & +  
		\frac{1}{Pe}\left(w_i, J_j\pd{\overline{v_i}}{x_j}\right) + \frac{1}{Pe}\left(w_i, \pd{\left(J_j v_i^\prime\right)}{x_j}\right) + \frac{Cn}{We} \left(w_i,\pd{}{x_j}\left({\pd{\phi}{x_i}\pd{\phi}{x_j}}\right)\right)  \\
		& + \frac{1}{We}\left(w_i,\pd{\left(\overline{p} + p^{\prime}\right)}{x_i}\right) + \frac{1}{Re}\left(\pd{w_i}{x_k},\eta(\phi)\pd{\left(\overline{v_i} + v_i^{\prime}\right)}{x_k}\right) - \left(\frac{w_i,\rho(\phi)\hat{g_i}}{Fr}\right)   \\
		&+  \left(q,\pd{\overline{v_i}}{x_i}\right) + \left(q,\pd{v_i^\prime}{x_i}\right)= 0, \label{eqn:weak_VMS_ns}
		\end{split}\\
		%
		\text{Cahn-Hilliard Eqn:}& \quad \left(q, \pd{ \left(\overline{\phi} 
			+ \phi^{\prime} \right)}{t}\right) 
		- \left(\pd{q}{x_i}, \overline{v_i}\overline{\phi} \right)  
		- \left(\pd{q}{x_i},v_i^\prime \overline{\phi}\right) 
		- \left(\pd{q}{x_i},\overline{v_i} \phi^{\prime}\right) \\ &
		- \frac{1}{PeCn} \left(\pd{q}{x_i}, m(\phi)\frac{\partial \mu}{\partial x_i}\right) = 0, \label{eqn:phi_eqn_var_decom}\\
		\text{Chemical Potential:}& \quad -\left(q,\mu\right) + \left(q, \d{\psi}{\phi}\right) - Cn^2 \left(q, \pd{}{x_i}\left({\pd{\phi}{x_i}}\right)\right)   = 0, \label{eqn:mu_eqn_var_decom}
\end{align} 
%
where $\vec{w} \in \mathscr{P}\vec{H}^{r}(\Omega)$, $\vec{\overline{v}}\in L^2\left(0, T ;\; \mathscr{P}\vec{H}^{r,h}_0(\Omega)\right)$, $\overline{p},\;\overline{\phi},\;\mu \in L^2\left(0, T ;\; \mathscr{P}H^r(\Omega)\right), \vec{v}^\prime \in (\mathscr{I} - \mathscr{P})\vec{H}^{r}(\Omega)$,
$\phi^\prime,p^\prime \in (\mathscr{I} - \mathscr{P})H^r(\Omega)$, and 
$q \in \mathscr{P}H^r(\Omega)$.
Here $\mathscr{I}$ is the identity operator and $\mathscr{P}$ is the projection operator.  We use the residual-based approximation proposed in~\citet{ Bazilevs2007} for the fine-scale components, applied to a two-phase system~\cite{Khanwale2020}, to close the equations:  
\begin{equation}
\rho(\phi) v_i^\prime = -\tau_m \mathcal{R}_m(\rho,\overline{v_i},\overline{p}), \qquad 
p^\prime = -\rho(\phi)\tau_c \mathcal{R}_c(\overline{v_i}), \qquad \text{and} \qquad
\phi^\prime = - \tau_{\phi} \mathcal{R}_{\phi}(\overline{v_i}, \phi).
\end{equation}

We substitute the infinite-dimensional spaces with their discrete counterparts (superscript $h$) using conforming Galerkin finite elements where the
the trial and test functions are taken from the same spaces.
Note that we only solve for the coarse-scale components. 
The resulting discrete variational formulation can then be defined as follows.
\begin{definition} {}{eqn:weak_VMS_disc} 
 Find $\vec{\overline{v}}^h \in L^2\left(0, T ;\; \mathscr{P}\vec{H}^{r,h}_0(\Omega)\right)$ and $\overline{p}^h, \overline{\phi}^h, \mu^h\in L^2\left(0, T ;\; \mathscr{P}H^{r,h}(\Omega)\right)$ such that
	\begin{align}
	\begin{split}
	\text{Momentum:}&  \quad
	\left(w_i,\overline{\rho}^h\pd{\overline{v_i}^h}{t}\right) + \left(w_i,\overline{\rho}^h\overline{v_j}^h\pd{\overline{v_i}^h}{x_j}\right) - \left(w_i,\tau_m \mathcal{R}_m\left(\overline{v_j}^h,\overline{p}^h\right)\pd{\overline{v_i}^h}{x_j}\right) \\ 
	&+ \left(\pd{w_i}{x_j},\overline{v_j}^h\left(\tau_m 
	\mathcal{R}_m\left(\overline{v_i}^h,\overline{p}^h\right)\right)\right) -\left(\pd{w_i}{x_j},\frac{\tau_m^2}{\overline{\rho}^h} \, \mathcal{R}_m\left(\overline{v_j}^h,\overline{p}^h\right) \mathcal{R}_m\left(\overline{v_i}^h,\overline{p}^h\right)\right) \\
	&+
	\frac{1}{Pe}\left(w_i, J_j^h\pd{\overline{v_i}^h}{x_j}\right) + \frac{1}{Pe}\left(\pd{w_i}{x_j}, J_j^h \frac{\tau_m}{\overline{\rho}^h} \, \mathcal{R}_m\left(\overline{v_i}^h,\overline{p}^h\right)\right)  \\ 
	&-\frac{Cn}{We} \left(\pd{w_i}{x_j}, \, {\pd{\overline{\phi}^h}{x_i}\pd{\overline{\phi}^h}{x_j}} \right) - 
	\frac{1}{We}\left(\pd{w_i}{x_i},\overline{p}^h\right) \\
	&+ \frac{1}{We}\left(\pd{w_i}{x_i},\overline{\rho}^h \tau_c \mathcal{R}_c\left(\overline{v_i}^h\right) \right) 
	+ \frac{1}{Re}\left(\pd{w_i}{x_k},\overline{\eta}^h\pd{\overline{v_i}^h}{x_k}\right) 
	- \left(\frac{w_i,\overline{\rho}^h\hat{g_i}}{Fr}\right)  = 0,
	\label{eqn:nav_stokes_eqn_var_decom}
	\end{split} \\
	\text{Thermo:} & \quad J^h_i = \frac{\left(\rho_- - \rho_+ \right)}{2\;\rho_+ Cn} \, \pd{\mu^h}{x_i}, \label{eqn:weak_thermo_consistency_var_decom} \\
	\text{Solenoidality:}& \quad \left(q,\pd{\overline{v_i}^h}{x_i}\right) + \left(\pd{q}{x_i},\frac{\tau_m}{\overline{\rho}^h}\mathcal{R}_m\left(\overline{v_i}^h,\overline{p}^h\right)\right)= 0, \\ 
	\begin{split}
	\text{Cahn-Hilliard:}& \quad \left(q, \pd{\overline{\phi}^h}{t}\right)
	- \left(\pd{q}{x_i}, \overline{v_i}^h\overline{\phi}^h \right) + \left(\pd{q}{x_i},\frac{\tau_m}{\overline{\rho}^h}\mathcal{R}_m\left(\overline{v_i}^h,\overline{p}^h\right) \overline{\phi}^h\right) \\ &
	+ \left(\pd{q}{x_i},\overline{v_i}^h \tau_{\phi} \mathcal{R}_{\phi}\left(\overline{v_i}^h, \overline{\phi}^h\right) \right)
	- \frac{1}{PeCn} \left(\pd{q}{x_i}, \overline{m}^h\frac{\partial \mu^h}{\partial x_i}\right) = 0, \label{eqn:phi_eqn_var_decom_rb}
	\end{split}\\
	\text{Potential:}& \quad -\left(q,\mu^h\right) + \left(q, \d{\psi}{\overline{\phi}^h}\right) + Cn^2 \left(\pd{q}{x_i},{\pd{\overline{\phi}^h}{x_i}}\right)  = 0,
	\label{eqn:mu_eqn_var_decom_disc}
	\end{align}
	$\forall \vec{w} \in \mathscr{P}\vec{H}^{r,h}_0(\Omega)$ and
	 $\forall q \in \mathscr{P}H^{r,h}(\Omega)$, and time $t \in [0, T]$.
		\label{eqn:weak_VMS_disc}
\end{definition}

In the above expressions, we used the following notation:
\begin{equation}
\overline{\rho}^h := \rho\left(\overline{\phi}^h\right), \quad
\eta^h := \eta\left(\overline{\phi}^h\right), \quad
\text{and} \quad
\overline{m}^h:=m\left(\overline{\phi}^h\right),
\end{equation}
and the following parameter values:
\begin{equation}
\begin{split}
\tau_m &= \left( \frac{4}{\Delta t^2}  + \overline{v_i}^hG_{ij}\overline{v_j}^h + \frac{1}{\overline{\rho}^h \, Pe}\overline{v_i}^hG_{ij}\overline{J_j}^h 
+ C_{I} \left(\frac{\overline{\eta}^h}{\overline{\rho}^hRe}\right)^2 G_{ij}G_{ij}\right)^{-1/2},\\
\tau_c &=  \frac{1}{tr(G_{ij})\tau_m}, \quad
\tau_\phi = \left( \frac{4}{\Delta t^2}  + \overline{v_i}^hG_{ij}\overline{v_j}^h + C_{\phi}\left(\frac{1}{PeCn}\right)^2 G_{ij}G_{ij}\right)^{-1/2}.
\end{split}
\end{equation}
Here we set $C_{I}$ and $C_{\phi}$ for all our simulations to 6 and the residuals are given by
\begin{equation}
\begin{split}
\mathcal{R}_m\left(\overline{v_i}^h,\overline{p}^h\right) &= \overline{\rho}^h\pd{\overline{v_i}^h}{t} + \overline{\rho}^h\overline{v_j}^h\pd{\overline{v_i}^h\;}{x_j} + 
\frac{1}{Pe}J_j^h\pd{\overline{v_i}^h}{x_j} + \frac{Cn}{We} \pd{}{x_j}\left({\pd{\overline{\phi}^h}{x_i}\pd{\overline{\phi}^h}{x_j}}\right) \\ &+ 
\frac{1}{We}\pd{\overline{p}^h}{x_i} - \frac{1}{Re}\pd{}{x_k}\left({\overline{\eta}^h\pd{\overline{v_i}^h}{x_k}}\right) - \frac{\overline{\rho}^h\hat{g}}{Fr},  \\ 
\mathcal{R}_c\left(\overline{v_i}^h\right) &= \pd{\overline{v_i}^h}{x_i}, \quad
\mathcal{R}_\phi\left(\overline{v_i}^h,\overline{\phi}^h,\mu^h\right) = \pd{\overline{\phi}^h}{t} + \pd{\left(\overline{v_i}^h\;\overline{\phi}^h\right)}{x_i} - \frac{1}{PeCn} \pd{}{x_i}\left(\overline{m}^h{\pd{\mu^h}{x_i}}\right).
\end{split}
\end{equation}

The following assumptions were made in the above variational problem. 
\begin{enumerate}
	\item $\vec{v}^{\prime} = 0$ on the boundary $\partial\Omega$; similarly $\phi^{\prime} = 0$ on  $\partial\Omega$.
	\item $\left(\pd{w_i}{x_k},\overline{\eta}^h\pd{v_i^\prime}{x_k}\right) = 0$ from the orthogonality condition of the projector.  The projector utilizes the inner product that comes from the bilinear form of these viscous terms~\citep{Hughes2007,Bazilevs2007}.
	\item  We assume that $\left(\pd{q}{x_i},v_i^{\prime} \phi^{\prime}\right) = 0$ under the reasoning that fluctuations in $\phi^\prime$ are small compared to $v_i^\prime$. This significantly simplifies the formulation.
	\item We use the coarse-scale part of the VMS decomposition of $\phi$ to compute the mixture density and mixture viscosity. We use the pulled back
	$\phi^{*}$ (see \cref{rmk:phi_pullback}) for this calculation, which regularizes 
	$\phi$ by smoothing out overshoots and undershoots.  
	The pull back ensures $\phi^{*} \in H^{r}$, and it's projection on the mesh $\phi^{*,h} \in H^{r,h}$ as required for the cG formulation. 
\end{enumerate}
\begin{remark}
	The above formulation is written for a generic order ($r$) for the interpolating polynomials (basis functions). 
\end{remark}
Finally, the time derivative in the above set of expressions is still continuous.  In the fully discrete numerical method we substitute the time-derivatives in the momentum and phase field equations using the time scheme presented in~\cref{eqn:nav_stokes_var_semi_disc} -- \cref{eqn:phi_eqn_var_semi_disc}.

\subsection{Handling non-linearity}
\label{subsec:newton_iter}
The fully discrete version of \cref{eqn:nav_stokes_var_semi_disc} --  \cref{eqn:phi_eqn_var_semi_disc} represents a non-linear system of algebraic equations that discretize the CHNS system \cref{eqn:nav_stokes_eqn_var_decom} -- \cref{eqn:mu_eqn_var_decom_disc} and that must be solved in each time-step. 
Symbolically, we can write this nonlinear algebraic system as
\begin{equation}
F_i\left(U^k_1, U^k_2, \dots, U^k_n\right) = 0,
\end{equation} 
where $\vec{U}^{k}$ is a vector containing all of the degrees of freedom
at the discrete time $t^k$.
In order to solve this nonlinear algebraic system we make use of a Newton
method, which requires us to solve the following linear system in each 
Newton iteration:
\begin{align}
J_{ij}^{s,k} \; \delta U_j^{s,k} = -F_i\left(U_1^{s,k}, U_2^{s,k}, 
\dots, U_n^{s,k} \right), \quad 
\quad J_{ij}^{s,k} := \pd{}{U_j} F_i\left(U_1^{s,k}, U_2^{s,k}, \dots, U_n^{s,k}\right),
\label{eqn:newton_linear_system}
\end{align}
where  $U_j^{s, k}$ is the vector containing all the degrees of freedom at the $k^{\text{th}}$ time step and at the
$s^{\text{th}}$ Newton iteration. $\delta U_j^{s, k}$ is the ``perturbation"  (update) vector that will be used to update the current Newton iteration guess:
\begin{align}
U_j^{s+1, k} = U_j^{s, k} + \delta U_j^{s, k}.
\end{align}
 $J_{ij}^{s,k}$ is the Jacobian matrix, which we analytically compute by calculating the variations (partial derivatives) of the operators with respect to the degrees of freedom.  The calculation of $J_{ij}^{s,k}$ is more challenging for the fully-coupled approach compared to the block iterative technique in~\citet{ Khanwale2020}.  
The iterative procedure begins with an initial guess, which we simply take
as the solution from the previous time-step:
\begin{equation}
U_i^{0,k} = U_i^{k-1},
\end{equation}
and ends once we reach the desired tolerance:
\begin{equation}
\|\delta U_j^{s,k}\| \le \text{TOL}.
\end{equation}
Once this tolerance is reached we set $U_j^{k} = U_j^{s,k}$ and move on to the next time-step.


In this work we solve linear system \cref{eqn:newton_linear_system} 
in each Newton iteration on a massively parallel architecture.
In particular, we make use of the  {\sc petsc}  library, which provides efficient parallel implementations of the above ideas along with an extensive suite of preconditioners and solvers for the linear system~\citep{petsc-efficient, petsc-web-page, petsc-user-ref}.  The precise choice of linear solvers and preconditioners is different for different numerical experiments; and therefore,
we provide more details on these choices in the numerical experiment sections of this paper. 


\section{Octree based domain decomposition}
\label{sec:octree_mesh}
Octrees are widely used in the computational sciences
to represent dynamically-adapted hierarchical meshes ~\cite{ SundarSampathBiros08, BursteddeWilcoxGhattas11, Fernando2018_GR, Fernando:2017}; this
 is largely due to their conceptual simplicity and their 
 ability to scale across a large number of processors. 
 Adaptivity is crucial in the computational sciences, where in many cases it reduces the overall degrees of freedoms (problem size), making these simulations feasible on currently available computers. The use of adaptive discretizations can introduce additional challenges, especially in distributed computing, such as load-balancing, low-cost scalable mesh generation, and mesh partitioning. Thus, we use Dendro, a highly scalable parallel octree library, to generate full adaptive quasi-structured octree based meshes and partitions. In the following sections, we summarize how Dendro allows us to perform our numerical simulations. The reader can find a detailed account of the algorithms used in Dendro in~\cite{ Fernando:2017, Fernando2018_GR}.

\subsection{Octree construction and 2:1 balancing} 
Dendro refines an octant based on user-specified criteria proceeding in a top-down fashion. The user defines the refinement criteria by a function that takes the coordinates of the octant, and returns \texttt{true} or \texttt{false}. Since the refinement happens locally to the element, this step is embarrassingly parallel. In distributed-memory machines, the initial top-down tree construction enables an efficient partitioning of the domain across an arbitrary number of processes. All processes start at the root node (i.e., the cubic bounding box for the entire domain). We perform redundant computations on all processes to avoid communication during the refinement stage. Starting from the root node, all processes refine (similar to a sequential implementation) until the process produces at least $\mathcal{O}(p)$ octants requiring further refinement. The procedure ensures that upon partitioning across $p$ processors, each processor gets at least one octant. Then using a space-filling-curve (SFC) based partitioning, we partition the octants across $p$ partitions~\cite{ Fernando:2017}. Once the algorithm completes this partitioning, we can restrict the refinement criterion to a processor's partition, which we can re-distribute to ensure load-balancing. 
We enforce a condition in our distributed octrees that no two neighbouring octants differ in size by more than a factor of two (2:1 balancing). This ratio makes subsequent operations simpler without affecting the adaptive properties. Our balancing algorithm uses a variant of \textsc{TreeSort}~\cite{ Fernando:2017} with top-down and bottom-up traversal of octrees which is different from existing approaches~\cite{ bern1999parallel, BursteddeWilcoxGhattas11, SundarSampathAdavaniEtAl07}. 

\subsection{SFC-based octree partition} 
The refinement and subsequent two-to-one balancing of the octree procedures can result in a non-uniform distribution of elements across processes, leading to load imbalance. This imbalance is particularly challenging when meshing complex geometries. SFC induces a partial ordering operator in higher dimensional space, where \textsc{ TreeSort}~\cite{ Fernando:2017} performs a parallel sort operation on the octants. The SFC traverses the octants in the sorted order, which reduces the partitioning problem to partitioning a 1D curve. Finally, we use a Hilbert SFC curve based partitioning compared to traditionally used Morton (Z-curve), which produces superior partitions in large scale computations~\cite{ Fernando:2017}.

\subsection{Mesh generation}  

By meshing, we refer to the construction of the data structures required to perform numerical computations on topological octree data. \dendro~builds distributed data structures to perform finite difference (FD), finite volume (FV), and finite element (FE) computations. In this work, we use the FE data structures. One of the key steps of the mesh generation stage is to construct neighborhood information for octants, which uses two maps computed by~\dendro. The first map \oTo~determines the neighboring octants of a given octant, and the map \oTn~computes the nodes corresponding to a given octant. We generate the \oTo~map by performing parallel searches similar to approaches described in~\cite{ Fernando2018_GR}. Assuming we have $n$~octants per partition, the search operations and building required to create the \oTo~and \oTn~data structures have $\mathcal{O}(n\log(n))$ and $\mathcal{O}(n)$ complexity, respectively. 

\subsection{Handling hanging nodes}

 While the use of quasi-structured grids such as octree-grids makes parallel meshing scalable and efficient without sacrificing adaptivity, one challenge is to handle the resulting non-conformity efficiently. The resulting {\em hanging nodes} occur on faces/edges shared between unequal elements. These hanging nodes do not represent independent degrees of freedom. We do not store the hanging nodes in~\dendro~to minimize the memory footprint and to improve the overall efficiency. The polynomial order of the elements and the free (non-hanging) nodes on the face/edge determine the value of the function at the hanging nodes. Therefore, we introduce these extra degrees of freedom as temporary variables before elemental matrix assembly or matrix-vector multiplication (for matrix-free computations) and eliminate them following the elemental operation. This virtualization of the hanging nodes is straightforward as we limit the meshes to a two-to-one balance, which limits the number of overall cases we need to consider explicitly see~\cite{ Fernando2018_GR} for further details on the handling of hanging nodes in \dendro. 
 
\subsection{Re-meshing and interpolation}

An essential requirement is to adapt the spatial mesh as the fluid interface moves across the domain. Figure~\ref{fig:rt3d_mesh_At15} shows the adaptive mesh refinement following the deformation of the interface in a Rayleigh-Taylor instability.  In distributed-memory systems, this localized meshing requires a re-partition and re-balance of the load. Thus, after a few time steps, we re-mesh. This re-meshing step is similar to the initial mesh generation and refinement. Now, the process uses the current position of the interface as well as the original geometry. The two-to-one balance enforcement and meshing follow this mesh generation. We now transfer the velocity field from the old mesh to the new mesh using a simple interpolation process. That is, the grid transfer only happens between parent and child (for coarsening and refinement) as it otherwise remains unchanged.  Therefore, the transfer from the old mesh to the new one uses standard polynomial interpolation, followed by a simple re-partitioning based on the new mesh.



\section{Numerical experiments}
\label{sec:num_exp}


\subsection{2D simulations: manufactured solutions}
\label{subsec:manfactured_soln_result}

We use the method of manufactured solutions to assess the convergence properties of our method.  We select an input ``solution'' which is solenoidal, and substitute it in the full set of governing equations. We then use the residual as a body force on the right-hand side of~\crefrange{eqn:nav_stokes_var_semi_disc}{eqn:phi_eqn_var_semi_disc}.
We choose the following ``solution'' with appropriate body forcing terms: 
\begin{equation}
\begin{split}
\vec{v} &= \left( \pi \sin^2(\pi x_1)\sin(2 \pi x_2)\sin(t), \, -\pi\sin(2\pi x_1)\sin^2(\pi x_2)\sin(t), \, 0 \right), \\
p &= \cos(\pi x_1)\sin(\pi x_2)\sin(t), \quad 
\phi = \mu = \cos(\pi x_1)\cos(\pi x_2)\sin(t).
\end{split}
\label{eq:manufac_exact}
\end{equation}
Our numerical experiments use the following non-dimensional parameters: $Re = 10$, $We = 1$, $Cn = 1.0$, $Pe = 3.0$, and $Fr = 1.0$. The density ratio is set to $\rho_{-}/\rho_{+} = 1.0$.

For the first experiment we use a 2D uniform mesh with $512 \times 512$~
elements with quadratic polynomials. 
Panel~(a) of~\cref{fig:manufac_temporal_convergence} shows the temporal convergence of the $L^2$ errors (numerical solution compared with the manufactured solution) calculated at $t = \pi$ to allow for one complete time
period. The figure shows the evolution of the error versus time-step on a log-log scale. The errors are decreasing with a slope close to two for the phase-field function $\phi$ and velocity, thereby demonstrating second-order convergence. 

We next conduct a spatial convergence study. We fix the time step at $\delta t = 10^{-3}$, and vary the spatial mesh resolution. Panel~(b) of~\cref{fig:manufac_temporal_convergence} shows the spatial convergence of $L^2$ errors (numerical solution compared with the manufactured solution) at $t = \pi$. 
We observe second order convergence for both velocity and $\phi$.

Panel~(c) of~\cref{fig:manufac_temporal_convergence} shows mass conservation for an intermediate resolution simulation with $\delta t=10^{-3}$ and $300 \times 300$ elements. We plot mass drift:
\begin{equation}
\int_{\Omega} \phi\left(\vec{x},t\right) \, d\vec{x} 
- \int_{\Omega}\phi\left(\vec{x},t=0\right) \, d\vec{x},
\end{equation}
and expect this value to be close to zero as per the theoretical prediction of~\cref{prop:mass_conservation}. We observe excellent mass conservation with fluctuations of the order of $10^{-12}$, which is to be
expected in double precision arithmetic. Here we used a relative tolerance of $10^{-7}$ for the Newton iteration.  For the linear solves within each Newton iteration we used a relative tolerance of $10^{-7}$.  
 
\begin{figure}
	\centering
	\includegraphics[width=0.42\linewidth]{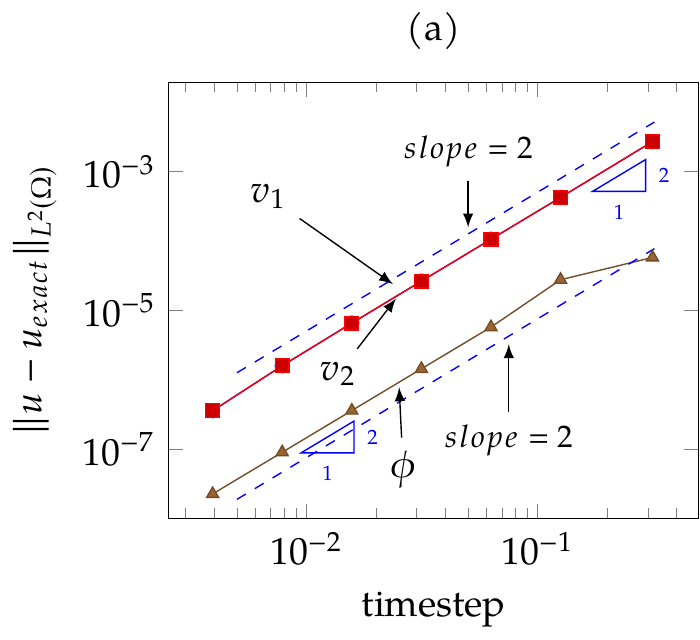}
	\includegraphics[width=0.42\linewidth]{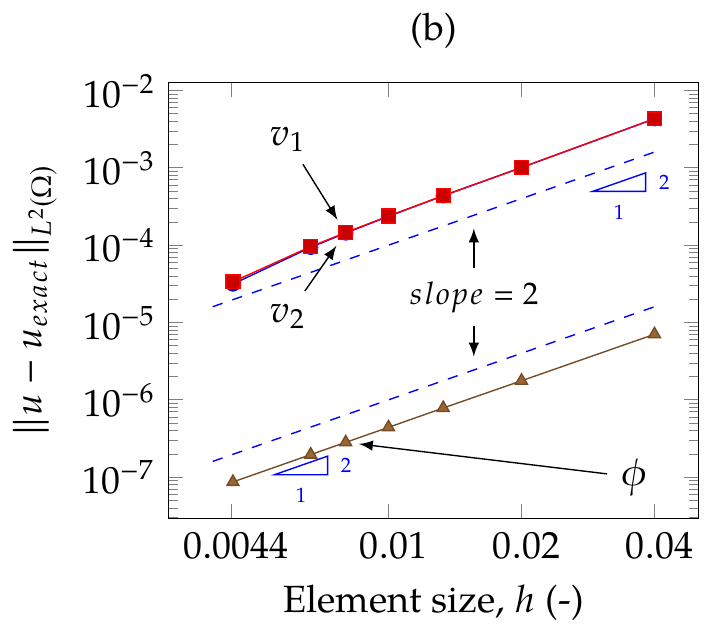}
	
	\includegraphics[width=0.45\linewidth]{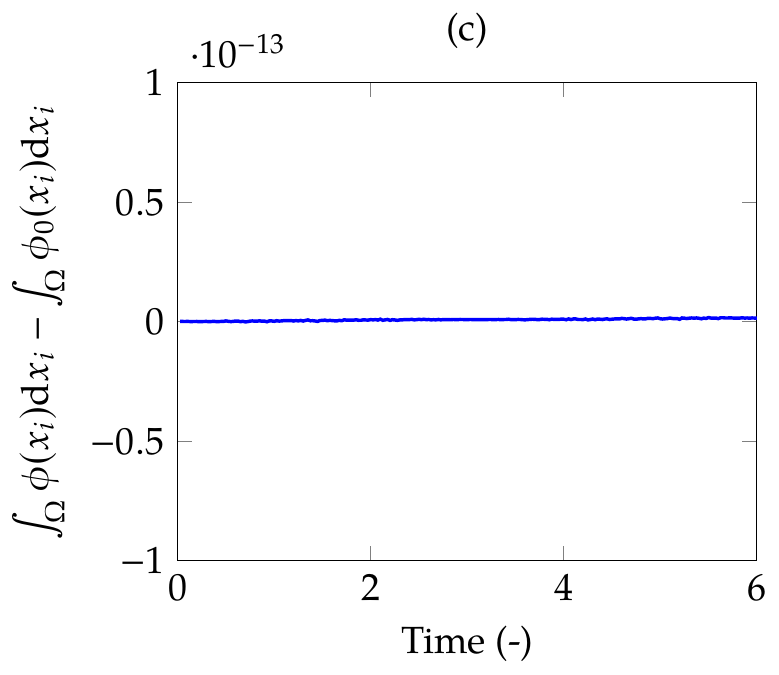}
	\caption{\textit{Manufactured Solution Examples} (\cref{subsec:manfactured_soln_result}). Shown in the Panels are
		(a) the temporal convergence of the numerical scheme for the case of manufactured solutions; (b) the spatial convergence of numerical scheme for the manufactured solutions with a time step of $10^{-3}$; and (c) the mass conservation for the case of manufactured solutions using $300 \times 300$ elements with time step of $10^{-2}$.
	} 
	\label{fig:manufac_temporal_convergence}
\end{figure}

\subsection{2D simulations: single rising bubble}
\label{subsec:single_rising_drop_2D}

To validate the framework, we consider benchmark cases for a single rising bubble in a quiescent water channel \citep{ Hysing2009, Aland2012, Yuan2017}.  
We set the Froude number ($Fr = u^2/(gD)$) to 1, which fixes the non-dimensional velocity scale to $u = \sqrt{gD}$, where $g$ is the gravitational acceleration, and $D$ is the diameter of the bubble.  This scaling gives a Reynolds number of $\rho_c g^{1/2}D^{3/2}/\mu_c$, where $\rho_c$ and $\mu_c$ are the specific density and specific viscosity of the continuous fluid, respectively. 
The Archimedes number,  $Ar = \rho_c g^{1/2}D^{3/2}/\mu_c$, 
scales the diffusion term in the momentum equation. 
The Weber number here becomes $We = \rho_c g D^2/\sigma$.  We use the density of the continuous fluid to non-dimensionalize: $\rho_{+} = 1$.  The density and viscosity ratios are $\rho_{+}/\rho_{-}$ and $\nu_{+}/\nu_{-}$, respectively. We present results for two standard benchmark cases.  

\Cref{tab:physParam_bubble_rise_2D_benchmarks} shows the parameters and the corresponding non-dimensional numbers for the two cases simulated in this work.  The bubble is centered at $(1,1)$, and since our scaling length scale is the bubble diameter, the bubble diameter for our simulations is 1.  The domain is 
$[0,2]\times[0,4]$.  
Following the benchmark studies in the literature, we choose the top and bottom wall to have no slip boundary conditions and the side walls to have boundary conditions: $v_1=0$ ($x$-velocity) and $\frac{\partial v_2}{\partial x}=0$ ($y$-velocity). We use the biCGstab (bcgs) linear solver from the PETSc suite along with the Additive Schwarz (ASM) preconditioner for the linear solves in the Newton iterations (see~\cref{subsec:newton_iter}).  We use a time step of $\num{2.5e-3}$ for both the test cases. The convergence criterion for both test cases uses a relative tolerance of $10^{-6}$ for Newton iteration and a relative tolerance of $10^{-7}$ for the linear solves within each Newton iteration.  

\begin{table}[H]
	\centering\normalsize\setlength\tabcolsep{5pt}
	\begin{tabular}{@{}|c|c|c|c|c|c|c|c|c|c|c|c|@{}}
		\toprule
		Test Case  & $\rho_{c}$  & $\rho_{b}$  & $\mu_{c}$  & $\mu_{b}$ & $\rho_{+}/\rho_{-}$ & $\nu_{+}/\nu_{-}$  & $g$ & $\sigma$ & $Ar$ & $We$ & $Fr$ \\
		\midrule
		\midrule
		{$1$}  & {1000}  & {100}  & {10}  &{1.0} & {10}     &{10} & {0.98}    & {24.5}  &  {35}     & {10} & {1.0}    \\
		{$2$}  & {1000}  & {1.0}  & {10}  &{0.1} & {1000}     &{100} & {0.98}    & {1.96}  &  {35}     & {125} & {1.0}    \\
		\bottomrule
	\end{tabular}
	\caption{Physical parameters and corresponding non-dimensional numbers for the 2D single rising drop  benchmarks considered
	in \cref{subsec:single_rising_drop_2D}.}
	\label{tab:physParam_bubble_rise_2D_benchmarks}                            
\end{table}

\subsubsection{Test case 1}
\label{subsubsec:single_rising_drop_2D_t1}
This test case considers the effect of higher surface tension, and consequently less deformation of the bubble as it rises. 
We compare the bubble shape in~\cref{fig:test_case1} with benchmark quantities presented in three previous studies~\citep{ Hysing2009, Aland2012, Yuan2017}.  We take $Cn=\num{5e-3}$ for this case.  Panel~(a) of~\cref{fig:test_case1} shows a shape comparison against benchmark studies in the literature, and we see an excellent agreement in the shape of the bubble.  Panel~(b) of~\cref{fig:test_case1} shows a comparison of centroid locations with respect to time against benchmark studies in the literature; again, we see an excellent agreement. We can see from the magnified inset in panel (b) of~\cref{fig:test_case1} that as we keep increasing the mesh resolution, the plot approaches the benchmark studies. We see an almost exact overlap between the benchmark and cases with $h = 2/400$ and $h = 2/600$, where $h$ is the size of the element, demonstrating spatial convergence. 

We next check whether the numerical method follows the theoretical energy stability proved in~\cref{thrm:energy_stability}.  We present the evolution of the energy functional defined in \cref{eqn:energy_functional} for test case 1. Panel~(c) of~\cref{fig:test_case1}~shows that the energy is decreasing in accordance with the energy stability condition for all three spatial resolutions of $h = 2.0/200$, $h = 2.0/400$, and $h = 2.0/600$. 

Finally, we check the mass conservation. Panel~(d) shows the total mass of the system minus the initial mass. At all reported spatial resolutions the change in the total mass is of the order of $10^{-8}$, even after 1600 time steps.  The numerical method delivers excellent mass conservation for long time horizons. 

This test case is a good example of a physical system evolving towards a steady state solution (terminal velocity and shape of the bubble).  This test case has also been used for benchmarking numerical schemes in prior literature.  Therefore, we use this example to contrast the time-step advantage of the fully-coupled, non-linear scheme.  Notice from~\cref{tab:timestep_comparison} that the time-step used with fully-coupled non-linear schemes (current work and \citet{Guo2017}) is an order of magnitude larger than the time-step used in linearized block schemes~\citep{Shen2010a,Shen2010b,Shen2015,Chen2016,Zhu2019}.  As mentioned in the introduction, there are many practical applications in bio-microfluidics where such a property is critical. 

\begin{table}[H]
	\centering\normalsize\setlength\tabcolsep{5pt}
	\begin{tabular}{@{}|c|c|c|c|@{}}
		\toprule
		Sr. no  & literature  & type of discretisation & time-step used  \\
		\midrule
		\midrule
		{$1$}  & Current work & fully-coupled non-linear & $\num{2.5e-3}$ \\
		{$2$}  & \citet{Zhu2019}  & linearized block  & $\num{3e-4}$  \\
		{$3$}  & \citet{Guo2017}  & fully-coupled non-linear  & $\num{1e-3}$  \\
		{$4$}  & \citet{Chen2016}  & linearized block  & $\num{2e-4}$  \\
		{$5$}  & \citet{Shen2015}  & linearized block  & $\num{1e-4}$  \\
		{$6$}  & \citet{Shen2010a,Shen2010b}  & linearized block  & $\num{1e-4}$  \\
		\bottomrule
	\end{tabular}
	\caption{Comparison of time-step used for in the current paper and literature for test case 1 benchmark considered
		in \cref{subsec:single_rising_drop_2D}.}
	\label{tab:timestep_comparison}                            
\end{table}

\begin{figure}[H]
	\centering
	\begin{tikzpicture}
	\begin{axis}[width=0.5\linewidth,scaled y ticks=true,xlabel={$\mathrm{x}$},ylabel={$\mathrm{y}$},legend style={nodes={scale=0.65, transform shape}}, ymin=0.0, ymax=4.0, ytick distance=1.0,  xtick={0.0, 1.0, 2}, title={(a)},
	legend style={nodes={scale=0.95, transform shape}, row sep=2.5pt},
	legend entries={present study, \citet{Hysing2009}, \citet{Aland2012}},
	legend pos= north west,
	legend image post style={scale=1.0},
	unit vector ratio*=1 1 1,
	xmin=0, xmax=2, 
	legend image post style={scale=3.0},
	]
	\addplot [only marks,mark size = 0.5pt,color=blue, each nth point=5, filter discard warning=false, unbounded coords=discard] table [x={x},y={y},col sep=comma] {bubbleShape.csv};
	\addplot [only marks,mark size = 0.5pt,color=black,each nth point=3, filter discard warning=false, unbounded coords=discard] table [x={x},y={y},col sep=comma] {Hysing_Re35We10.csv};
	\addplot [only marks,mark size = 0.5pt,color=red,each nth point=1, filter discard warning=false, unbounded coords=discard] table [x={x},y={y},col sep=comma] {Aland_Voight_Re35We10_shape.csv};
	\end{axis}
	\end{tikzpicture}
	\includegraphics[]{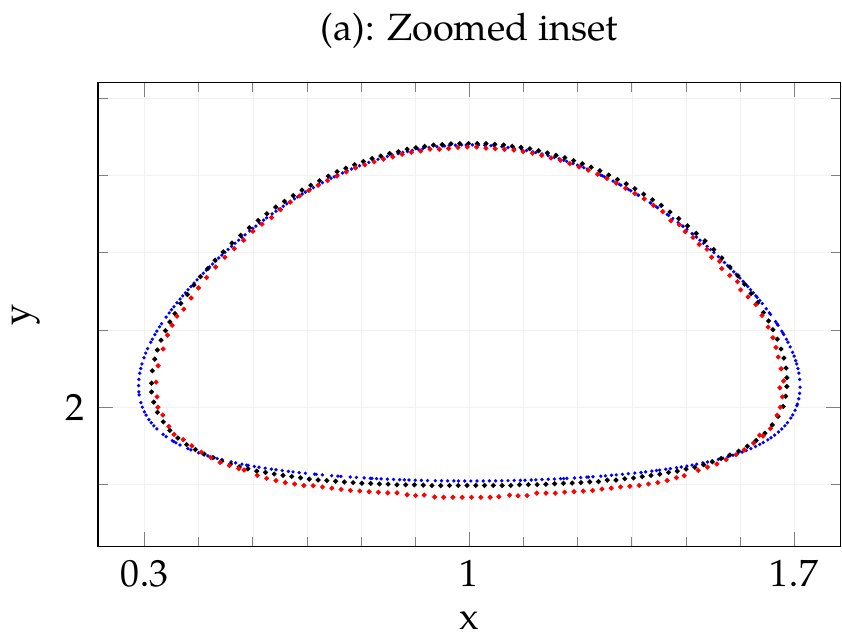}
	
	\begin{tikzpicture}[spy using outlines={rectangle, magnification=3, size=1.5cm, connect spies}]
	\begin{axis}[width=0.45\linewidth,scaled y ticks=true,xlabel={Time (-)},ylabel={Centroid},legend style={nodes={scale=0.65, transform shape}}, xmin=0, xmax=4.2, ymin=0.9, ymax=2.5, ytick distance=0.5,  xtick={0.0, 1.0, 2, 3, 4}, title={(b)},
	legend style={nodes={scale=0.95, transform shape}, row sep=2.5pt},
	legend entries={\citet{Hysing2009}, \citet{Aland2012}, \citet{Yuan2017}, $h = 2.0/200$, $h = 2.0/400$, $h = 2.0/600$},
	legend pos= north west,
	legend image post style={scale=1.0}
	]
	\addplot [line width=0.15mm, color=black] table [x={time},y={centroid},col sep=comma] {Hysing_centroid_Re35We10.csv};
	\addplot [line width=0.15mm, color=red] table [x={time},y={centroid},col sep=comma] {Aland_Voight_centroid_Re35We10.csv};
	\addplot [line width=0.15mm, color=ForestGreen] table [x={time},y={centroid},col sep=comma] {Yuan_centroid_Re35We10.csv};
	\addplot+[mark size = 0.5pt]table [x={time},y={centroid},col sep=comma, each nth point=1, filter discard warning=false, unbounded coords=discard] {centerOfMass_200_400.csv};	
	\addplot+[mark size = 0.5pt]table [x={time},y={centroid},col sep=comma, each nth point=1, filter discard warning=false, unbounded coords=discard] {centerOfMass_400_800.csv};
	\addplot+[mark size = 0.5pt]table [x={time},y={centroid},col sep=comma, each nth point=1, filter discard warning=false, unbounded coords=discard] {centerOfMass_600_1200.csv};
	\coordinate (a) at (axis cs:3.6,2.0);
	\end{axis}
	\spy [black] on (a) in node  at (5,1.2);
	\end{tikzpicture}
	\hskip 5pt
	\includegraphics{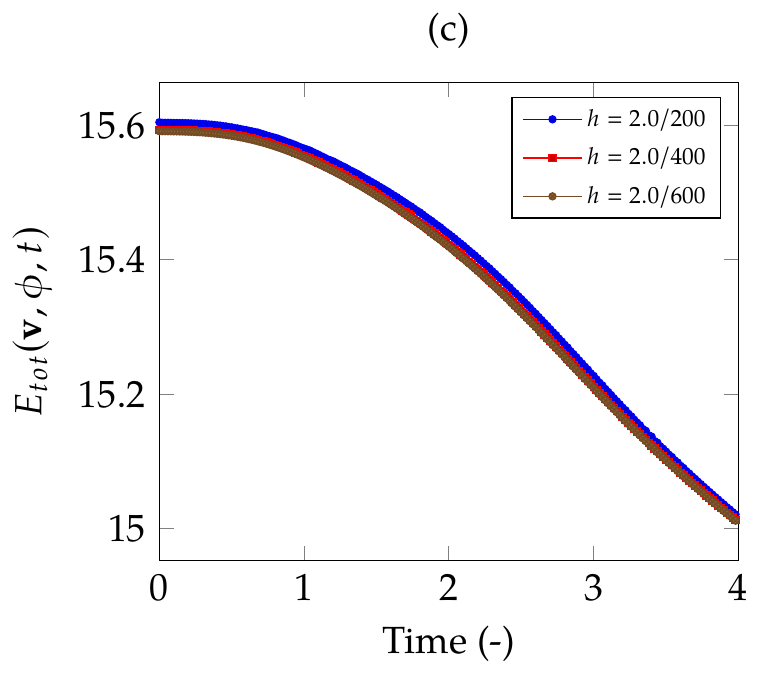}
	
	\includegraphics{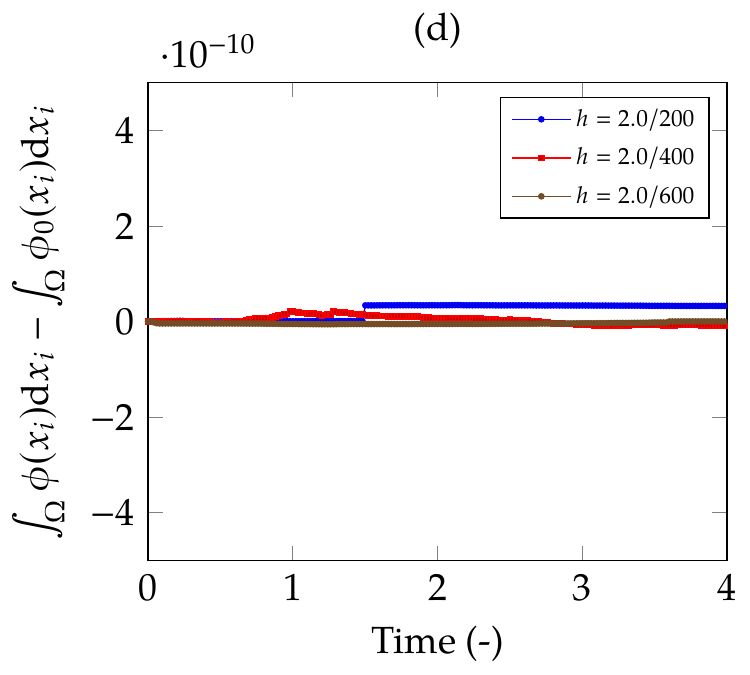}
	\caption{ \textit{2D Single Rising Drop Test Case 1}
		(\cref{subsubsec:single_rising_drop_2D_t1}). Shown in the panels are (a) comparisons of the computed bubble shape against results
		from the literature at non-dimensional time $T = 4.2$; (b) comparisons of the rise of the bubble centroid against results from the literature; (c) decay of the energy functional illustrating \cref{thrm:energy_stability}; and (d) total mass conservation (integral of total $\phi$).}
	\label{fig:test_case1}
\end{figure}

\subsubsection{Test case 2}
\label{subsubsec:single_rising_drop_2D_t2}
This test case considers a lower surface tension resulting in high deformations of the bubble as it rises. 
As before, we compare the bubble shape in~\cref{fig:test_case2} with benchmark quantities presented in~\citep{ Hysing2009, Aland2012, Yuan2017}.  Panel~(a) shows the shape comparison with benchmark studies in the literature. We see an excellent agreement in the shape of the bubble. All simulations (our results and benchmarks) exhibit a skirted bubble shape. We see an excellent match in the overall bubble shape with minor differences in the dynamics of its tails. Specifically, we see that the tails of the bubble in our case pinch-off to form satellite bubbles\footnote{Such instabilities require a low $Cn$ number, as only a  thin interface can capture the dynamics of the thin tails of the bubble}.  We performed this simulation with a $Cn=0.0025$ and three different spatial resolutions.  We can see in panel~(a) of~\cref{fig:test_case2} that our simulation captures this filament pinch-off in the tails.  The works of~\citet{ Aland2012, Yuan2017} did not observe these thin tails and pinch-offs, while~\citet{ Hysing2009} described pinch-off of the tails and satellite bubbles.  

Panel~(b) of~\cref{fig:test_case2} compares the centroid location evolution with time. Again we see an excellent agreement with all three previous benchmark studies.  We can see from the magnified inset in this panel that as we increase the mesh resolution the plot approaches the benchmark studies and we see an almost exact overlap between the benchmark and cases with $h = 2/1000$ and $h = 2/2000$ demonstrating spatial convergence. Next, we report the evolution of the energy functional defined in~\cref{eqn:energy_functional} for test case 2.  Panel~(c) of~\cref{fig:test_case2}~shows the decay of the total energy functional in accordance with the energy stability condition for all three spatial resolutions of $h = 2.0/800$, $h = 2.0/1000$, and $h = 2.0/1200$. Finally, panel~(d) of~\cref{fig:test_case2} shows the total mass of the system in comparison with the total initial mass of the system. We can see that for all spatial resolutions, the change in the total mass against the initial total mass is of the order of \num{1e-8}, which illustrates that the numerical method satisfies mass conservation over long time horizons.

\begin{figure}[H]
	\centering
	\begin{tikzpicture}
	\begin{axis}[width=0.5\linewidth,scaled y ticks=true,xlabel={$\mathrm{x}$},ylabel={$\mathrm{y}$},legend style={nodes={scale=0.65, transform shape}}, ymin=0, ymax=4, ytick distance=1.0,  xtick={0.0, 1.0, 2}, title={(a)},
	legend style={nodes={scale=0.95, transform shape}, row sep=2.5pt},
	legend entries={present study, \citet{Hysing2009}, \citet{Aland2012}, \citet{Yuan2017}},
	legend pos= north west,
	legend image post style={scale=5.0},
	unit vector ratio*=1 1 1,
	xmin=0, xmax=2, 
	legend image post style={scale=3.0}
	]
	\addplot [only marks,mark size = 0.1pt,color=blue, each nth point=6, filter discard warning=false, unbounded coords=discard] table [x={x},y={y},col sep=comma] {bubbleShape_Re35We125.csv};
	\addplot [only marks,mark size = 0.1pt,color=black,each nth point=1, filter discard warning=false, unbounded coords=discard] table [x={x},y={y},col sep=comma] {Hysing_Re35_We125_bubble_shape.csv};
	\addplot [only marks,mark size = 0.1pt,color=red,each nth point=1, filter discard warning=false, unbounded coords=discard] table [x={x},y={y},col sep=comma] {Aland_Voight_Re35_We125_bubble_shape.csv};
	\addplot [only marks,mark size = 0.1pt,color=ForestGreen,each nth point=1, filter discard warning=false, unbounded coords=discard] table [x={x},y={y},col sep=comma] {Yuan_Re35_We125_bubble_shape.csv};
	\end{axis}
	\end{tikzpicture}
	\includegraphics{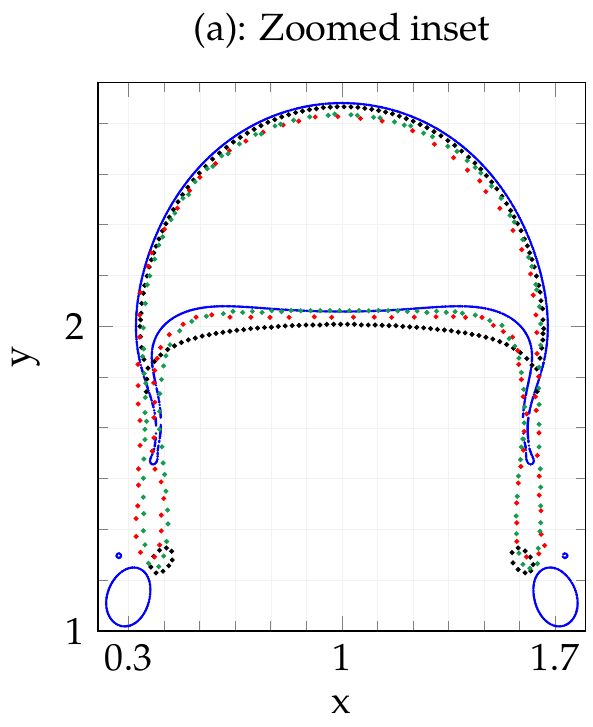}
	
	\begin{tikzpicture}[spy using outlines={rectangle, magnification=3, size=1.5cm, connect spies}]
	\begin{axis}[width=0.45\linewidth,scaled y ticks=true,xlabel={Time (-)},ylabel={Centroid},legend style={nodes={scale=0.65, transform shape}}, xmin=0, xmax=4.2, ymin=0.9, ymax=2.8, ytick distance=0.5,  xtick={0.0, 1.0, 2, 3, 4}, title={(b)},
	legend style={nodes={scale=0.95, transform shape}, row sep=2.5pt},
	legend entries={\citet{Hysing2009}, \citet{Aland2012}, \citet{Yuan2017}, $h = 2.0/600$, $h = 2.0/800$, $h = 2.0/1000$},
	legend pos= north west,
	legend image post style={scale=1.0}
	]
	\addplot [line width=0.1mm, color=black] table [x={time},y={centroid},col sep=comma] {Hysing_centroid_Re35We125.csv};
	\addplot [line width=0.1mm, color=red] table [x={time},y={centroid},col sep=comma] {aland_voight_centroid_Re35We125.csv};
	\addplot [line width=0.1mm, color=ForestGreen] table [x={time},y={centroid},col sep=comma] {Yuan_centroid_Re35We125.csv};
	\addplot+[mark size = 0.75pt]table [x={time},y={centroid},col sep=comma, each nth point=1, filter discard warning=false, unbounded coords=discard] {centerOfMass_600_1200_Re35We125.csv};	
	\addplot+[mark size = 0.75pt]table [x={time},y={centroid},col sep=comma, each nth point=1, filter discard warning=false, unbounded coords=discard] {centerOfMass_800_1600_Re35We125.csv};
	\addplot+[mark size = 0.75pt]table [x={time},y={centroid},col sep=comma, each nth point=1, filter discard warning=false, unbounded coords=discard] {centerOfMass_1000_2000_Re35We125.csv};
	\coordinate (a) at (axis cs:3.6,2.05);
	\end{axis}
	\spy [black] on (a) in node  at (5,1.2);
	\end{tikzpicture}
	\hskip 5pt
	\includegraphics{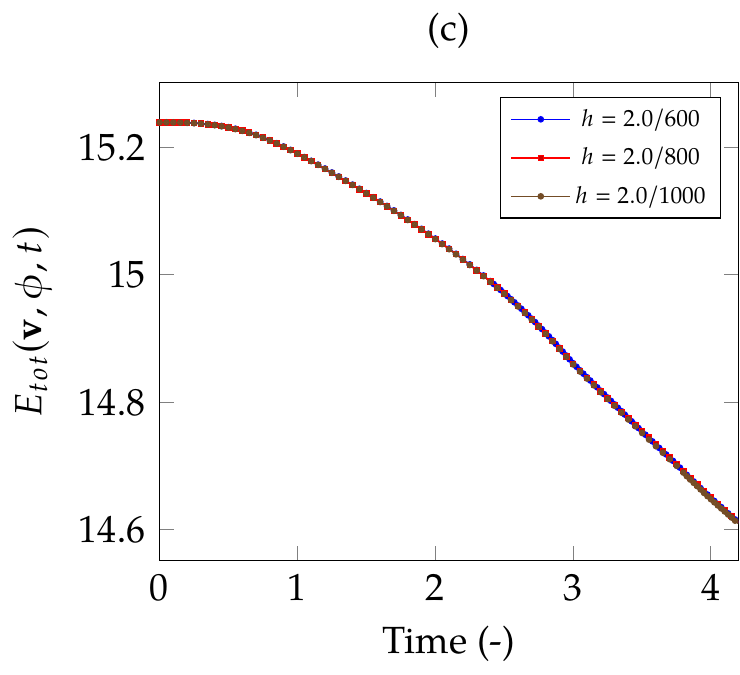}
	
	\includegraphics{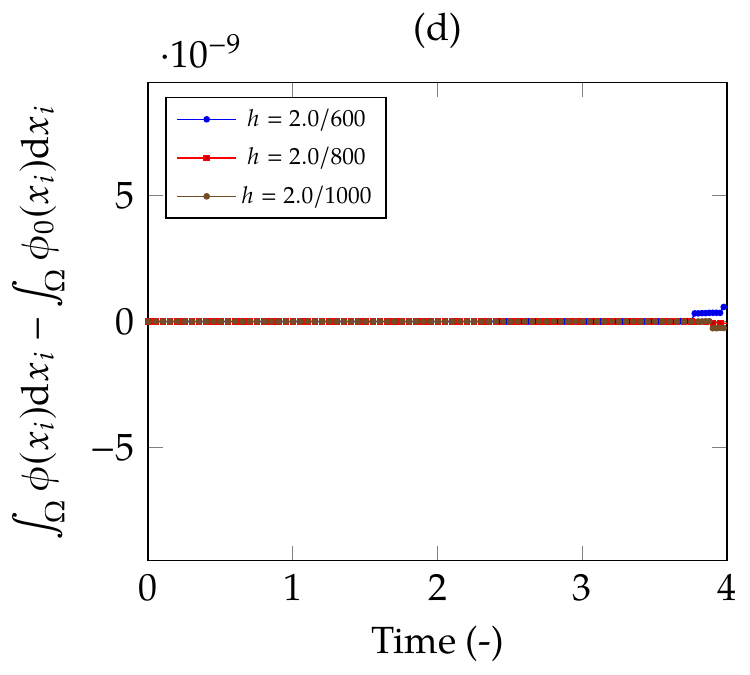}
		\caption{\textit{2D Single Rising Drop Test Case 2} 
	(\cref{subsubsec:single_rising_drop_2D_t2}). Shown in the panels are (a)~comparisons of the computed bubble shape against results from the literature at non-dimensional time $T = 4.2$; (b)~comparisons of the rise of the bubble centroid against results from the literature; (c)~decay of the energy functional illustrating theorem~\ref{thrm:energy_stability}; and (d)~total mass conservation (integral of total $\phi$).}
	\label{fig:test_case2}
\end{figure}

\subsection{2d simulations: Rayleigh-Taylor instability}
\label{subsec:rayleigh_taylor_2D}
We now demonstrate the performance of the numerical framework with large deformation in the interface and chaotic regimes (high Reynolds numbers).  While the bubble rise case in the previous sub-section is an interplay between surface tension and buoyancy, buoyancy dominates the evolution of the Rayleigh-Taylor instability. Here the choice of non-dimensional numbers ensures that the surface tension effect is small (high Weber number). In contrast, other studies switch off the surface tension forcing terms in the momentum equations (see, e.g.,~\citep{ Xie2015, Tryggvason1990, Li1996, Guermond2000}).  

The setup is as follows: the heavier fluid is on top of lighter fluid and the interface is perturbed. The heavier fluid on top penetrates the lighter fluid and buckles, which generates instabilities.  This interface motion is challenging to track due to large changes in its topology. Additionally, the Rayleigh-Taylor instabilities generally encompass turbulent conditions that require resolving finer scales to capture the complete dynamics.  We non-dimensionalize the problem by selecting the width of the channel as the characteristic length scale and the density of the lighter fluid as the characteristic specific density.  Just as in the bubble rise case we use buoyancy-based scaling, setting the Froude number ($Fr = u^2/(gD)$) to 1, which fixes the non-dimensional velocity scale to be $u = \sqrt{gD}$, where $g$ is the gravitational acceleration, and $D$ is the width of the channel.  Using this velocity to calculate the Reynolds number, we get $Re = \rho_L g^{1/2}D^{3/2}/\mu_L$, where $\rho_L$ and $\mu_L$ are the specific density and specific viscosity of the light fluid, respectively. We set the Reynolds number to $Re = 1000$.  These choices lead to a Weber number of $We = \rho_c g D^2/\sigma$.  To compare our results with previous studies, we simulate with the same initial conditions as presented in~\citet{ Tryggvason1988, Guermond2000, Ding2007, Xie2015}.  The $We$ number is selected to be 1000, so that the effect of surface tension is small on the evolution of the interface.  

The Atwood number ($At$) is often used to parametrize the dependence on density ratio, with $At = \left(\rho_{+} - \rho_{-}\right)/\left(\rho_{+} + \rho_{-}\right)$.  For the density ratios of 0.33, and 0.1, the Atwood numbers are $At = 0.5$, and $At = 0.82$, respectively.
We chose specific density of the heavy fluid to non-dimensionalize, therefore $\rho_{+} = 1.0$, and $\rho_{-} = 0.33$ for $At = 0.5$, while $\rho_{-} = 0.1$ for $At = 0.82$.  $\nu_{+}/\nu_{-}$ the viscosity ratio is set to 1. We use a no-slip boundary condition for velocity on all the walls along with no flux boundary conditions for $\phi$ and $\mu$.  The no-flux boundary condition for $\phi$ and $\mu$ inherently assumes a 90 degree wetting angle for both the fluids.  

\begin{remark}
		Weak surface tension reduces vortex roll-up in the simulations of immiscible systems, especially at lower $At$. Experimental results from~\citet{ Waddell2001} show different vortex roll-up amounts for miscible and immiscible systems. This difference is analogous to the difference between zero surface tension simulations and finite surface tension simulations. This effect is irrelevant to compare front locations against the literature (short time horizons). Nevertheless, it is crucial to accurately track the long time dynamics (as smaller filaments are more stable in the non-zero surface tension case). 
\end{remark}     

We run numerical experiments for $Cn = 0.005, 0.0025, 0.00125$ with a uniform mesh of $400\times3200$, $400\times3200$, and $800\times6400$, respectively, for two different Atwood numbers: $At = 0.82, 0.5$.  The time step size for all the experiments is \num{1.25e-4}. A carefully tuned algebraic multi-grid linear solver with successive over-relaxation is setup for the linear solves in the Newton iterations (see~\cref{subsec:newton_iter}).  We detail the command-line arguments used in~\ref{sec:app_linear_solve}.  For the convergence criteria for both 2D Rayleigh-Taylor test cases we use a relative tolerance of $10^{-6}$ for Newton iteration and a relative tolerance of $10^{-7}$ for  the linear solves.

\Cref{fig:rt2d} shows the snapshots of the interface shape along with the corresponding vorticity generated as it evolves in time for $At = 0.82$ and $Cn = 0.00125$. We observe the usual evolution of Rayleigh Taylor instability where the heavier fluid penetrates the light one, causing the lighter fluid to rise near the wall.  The penetrating plume of the heavier fluid sheds small filaments at a non-dimensional time of $t^\prime = 1.358$.  The penetrating plume is symmetric at early times, with symmetry breaking occurring at longer times. The instability further proceeds to a periodic flapping. At longer times, the instability transitions to a chaotic mixing stage.\footnote{These high-resolution simulations of Rayleigh Taylor instability over long time horizons could serve as benchmarks. This data will be made publicly available.
} 

\begin{figure}[H]
	\centering	
	\begin{tabular}{p{0.13\textwidth}p{0.13\textwidth}p{0.13\textwidth}p{0.13\textwidth}p{0.13\textwidth}}
		\subfigure [$t^\prime = 0.0$] {
			\includegraphics[width=\linewidth]{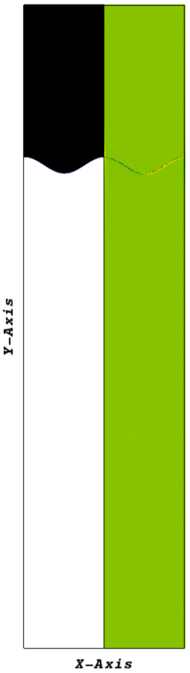}
			\label{subfig:rt_snap_1}
		} &
		\subfigure [$t^\prime = 0.996$] {
			\includegraphics[width=\linewidth]{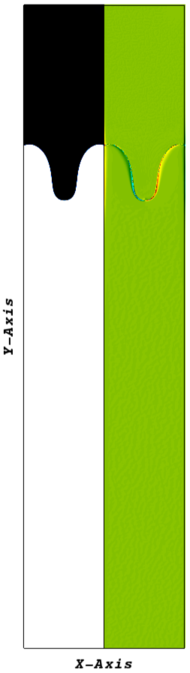}
			\label{subfig:rt_snap_2}
		} & 
		\subfigure [$t^\prime = 1.7205$] {
			\includegraphics[width=\linewidth]{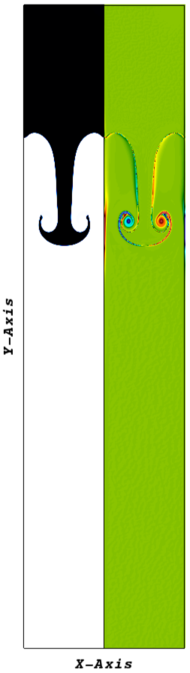}
			\label{subfig:rt_snap_3}
		} &
		
		\subfigure [$t^\prime = 2.3544$] {
			\includegraphics[width=\linewidth]{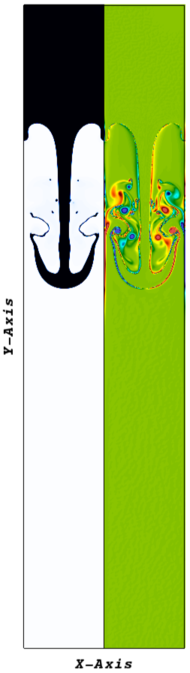}
			\label{subfig:rt_snap_4}
		} &
		\subfigure [$t^\prime = 2.9882$] {
			\includegraphics[width=\linewidth]{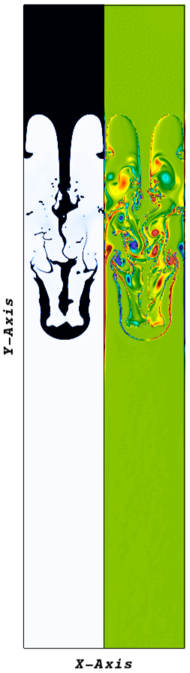}
			\label{subfig:rt_snap_5}
		} \\ 
		\subfigure [$t^\prime = 3.5316$] {
			\includegraphics[width=\linewidth]{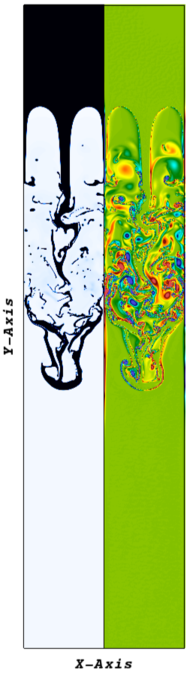}
			\label{subfig:rt_snap_6}
		} &
		
		\subfigure [$t^\prime = 4.0749$] {
			\includegraphics[width=\linewidth]{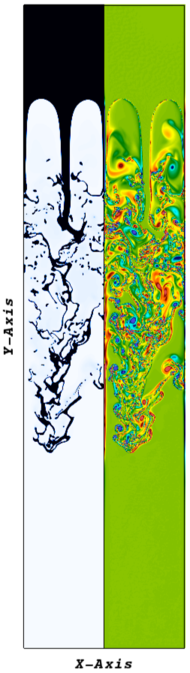}
			\label{subfig:rt_snap_7}
		} &
		\subfigure [$t^\prime = 4.89$] {
			\includegraphics[width=\linewidth]{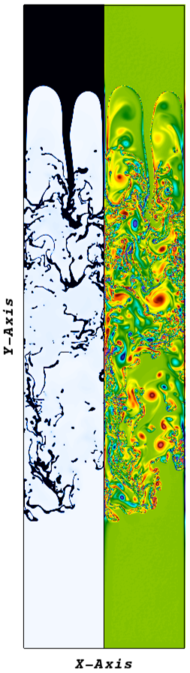}
			\label{subfig:rt_snap_8}
		} & 
		\subfigure [$t^\prime = 5.433$] {
			\includegraphics[width=\linewidth]{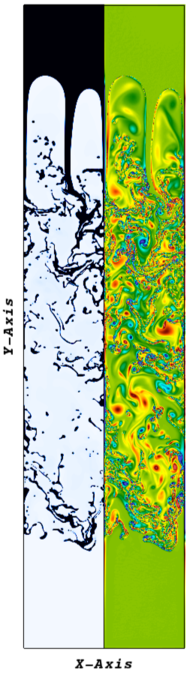}
			\label{subfig:rt_snap_9}
		} &
		\subfigure [$t^\prime = 5.9765$] {
			\includegraphics[width=\linewidth]{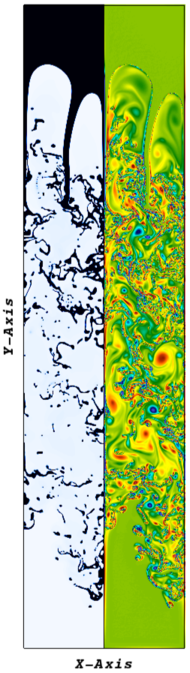}
			\label{subfig:rt_snap_10}
		}  
	\end{tabular}
	\caption{\textit{Rayleigh-Taylor instability in 2D} (\cref{subsec:rayleigh_taylor_2D}): Dynamics of the interface as a function of time for $At = 0.82$ (density ratio of 0.1).  In each panel the left plot illustrates the interface, and the right plot shows corresponding vorticity. Here normalised time $t' = t \sqrt{At}$}
	\label{fig:rt2d}
\end{figure}

\subsubsection{Influence of $Cn$ on long time dynamics}

We observe from~\cref{fig:rt2d} that the development of the instability (at longer times) depends on the resolution of shedding filaments. Therefore, the long-time dynamics depend on the interface thickness that the $Cn$ number controls.  To analyze the influence of $Cn$ number on the development of the Rayleigh-Taylor instability, we perform numerical experiments with three $Cn$ numbers. First, we compare our results with those in the literature to validate the framework.  For the cases of $At = 0.5$ and $At =0.82$ several previous studies presented the location of the top and bottom fronts as a function of time.  Panel~(a)~\cref{fig:RT_2D_comparison} compares the bottom and top front locations with previous studies~\citep{ Xie2015, Tryggvason1990, Li1996, Guermond2000}.  Our current results match the previous benchmarks for all three $Cn$ numbers.  Panels~(b) and~(c) from~\cref{fig:RT_2D_comparison} show the energy decay in line with~\cref{thrm:energy_stability} for $At = 0.5,0.82$. We observe energy stability for all the three $Cn$ numbers. 

At longer times, defining the top/bottom front becomes difficult due to filament breakup.  Therefore, we plot the center of mass of the heavy fluid as a function of time. This location is a good integral metric to track coarse-scale dynamics. Panels~(a) and~(b) of~\Cref{fig:RT_2D_centroid} show the evolution of center of mass of the heavy fluid for $At = 0.5$ and $At=0.82$, respectively. We observe a convergence of dynamics as we decrease the $Cn$ number. There are some deviations even for small $Cn$ numbers at longer time durations. The chaotic filament breakup and concomitant fluid features due to the relative motion of the interface cause these deviations. We explore them by visualizing the development of coherent vortices using the $Q$-criterion~\citep{ Hunt1988}. \Cref{subfig:rt_q_snap_1} shows that $Cn = 0.005$ under-resolves the filaments shed as the instability develops. This lack of sufficient resolution causes under-resolution of fine-scale vortices which depend on the shear instability generated by the finer filaments.  We observe these finer filaments as we decrease the $Cn$ number; \cref{subfig:rt_q_snap_2} shows the interface and the corresponding $Q-$criterion for $Cn = 0.0025$ at the same time point as $Cn = 0.005$. We resolve these finer filaments better in this case.  Upon further reduction of $Cn$  to $0.00125$, \cref{subfig:rt_q_snap_3} captures much finer flow structures compared to $Cn = 0.0025$.  We observe that even if integral metrics like front locations and center of mass match for two $Cn$ numbers ($0.005$ and $0.0025$ in our case, see~\cref{fig:RT_2D_centroid}), the fine-scale flow structures can be quite different due to their dependence on the resolution of the finer filaments. This fine structure resolution has a profound influence on the higher-order statistics of Rayleigh-Taylor instability.  We defer a detailed analysis of higher-order statistics to future work.

\begin{figure}[H]
	\centering
	\begin{tikzpicture}[]
	\begin{axis}[width=0.5\linewidth,scaled y ticks=true,xlabel={Normalised time $t' = t \sqrt{At}$(-)},ylabel={Position(-)},legend style={nodes={scale=0.65, transform shape}}, xmin=0, xmax=3.25, ymin=-3.5, ymax=2.0, ytick distance=1.0,  xtick={0.0, 0.5, 1.0, 1.5, 2, 2.5, 3.0, 3.5, 4.0}, title={(a)},
	legend style={nodes={scale=0.95, transform shape}, row sep=2.0pt},
	legend entries={
		\citet{Ding2007}~($At = 0.5$), \citet{Guermond2000}~($At = 0.5$), \citet{Tryggvason1988}~($At = 0.5$), 
		present study bottom front ($At = 0.5$~$Cn = 0.005$), 
		present study top front ($At = 0.5$~$Cn = 0.005$),
		present study bottom front ($At = 0.5$~$Cn = 0.0025$), 
		present study top front ($At = 0.5$~$Cn = 0.0025$),
		present study bottom front ($At = 0.5$~$Cn = 0.00125$), 
		present study top front ($At = 0.5$~$Cn = 0.00125$), 
		\citet{Xie2015}~($At = 0.82$), 
		present study bottom front ($At = 0.82$~$Cn = 0.005$), 
		present study top front ($At = 0.82$~$Cn = 0.005$),		
		present study bottom front ($At = 0.82$~$Cn = 0.0025$), 
		present study top front ($At = 0.82$~$Cn = 0.0025$),
		present study bottom front ($At = 0.82$~$Cn = 0.00125$), 
		present study top front ($At = 0.82$~$Cn = 0.00125$)},
	legend pos= outer north east,
	legend image post style={scale=1.0},
	legend style={font=\footnotesize}
	]
	\addplot +[only marks,mark size = 1.5pt] table [x={time},y={position},col sep=comma] {Ding_combined.csv};
	\addplot +[only marks,mark size = 1.5pt] table [x={time},y={position},col sep=comma] {Guermond_combined.csv};
	\addplot +[only marks,mark size = 1.5pt] table [x={time},y={position},col sep=comma] {Tryggvason_combined.csv};	
	\addplot [line width=0.25mm, color = red, dotted]table [x={time},y={bottomFront},col sep=comma, each nth point=3, filter discard warning=false, unbounded coords=discard] {extentsRTinstability_At0dot5_Cn0dot005.csv};
	\addplot [line width=0.25mm, color = red, dotted]table [x={time},y={topFront},col sep=comma, each nth point=3, filter discard warning=false, unbounded coords=discard] {extentsRTinstability_At0dot5_Cn0dot005.csv};
	\addplot [line width=0.25mm, color = red, dashed]table [x={time},y={bottomFront},col sep=comma, each nth point=3, filter discard warning=false, unbounded coords=discard] {extentsRTinstability_At0dot5_Cn0dot0025.csv};
	\addplot [line width=0.25mm, color = red, dashed]table [x={time},y={topFront},col sep=comma, each nth point=3, filter discard warning=false, unbounded coords=discard] {extentsRTinstability_At0dot5_Cn0dot0025.csv};
	\addplot [line width=0.25mm, color = red]table [x={time},y={bottomFront},col sep=comma, each nth point=3, filter discard warning=false, unbounded coords=discard] {extentsRTinstability_At0dot5_Cn0dot00125.csv};
	\addplot [line width=0.25mm, color = red]table [x={time},y={topFront},col sep=comma, each nth point=3, filter discard warning=false, unbounded coords=discard] {extentsRTinstability_At0dot5_Cn0dot00125.csv};
	\addplot [only marks,mark size = 1.5pt, mark=triangle*, color=black, each nth point=3, filter discard warning=false, unbounded coords=discard] table [x={time},y={position},col sep=comma] {xie_combined_At0dot82.csv};
	\addplot [line width=0.25mm, color = blue, dotted]table [x={time},y={bottomFront},col sep=comma, each nth point=3, filter discard warning=false, unbounded coords=discard] {extentsRTinstability_At0dot82_Cn0dot005.csv};
	\addplot [line width=0.25mm, color = blue, dotted]table [x={time},y={topFront},col sep=comma, each nth point=3, filter discard warning=false, unbounded coords=discard] {extentsRTinstability_At0dot82_Cn0dot005.csv};
	\addplot [line width=0.25mm, color = blue, dashed]table [x={time},y={bottomFront},col sep=comma, each nth point=3, filter discard warning=false, unbounded coords=discard] {extentsRTinstability_At0dot82_Cn0dot0025.csv};
	\addplot [line width=0.25mm, color = blue, dashed]table [x={time},y={topFront},col sep=comma, each nth point=3, filter discard warning=false, unbounded coords=discard] {extentsRTinstability_At0dot82_Cn0dot0025.csv};
	\addplot [line width=0.25mm, color = blue, solid]table [x={time},y={bottomFront},col sep=comma, each nth point=3, filter discard warning=false, unbounded coords=discard] {extentsRTinstability_At0dot82_Cn0dot00125.csv};
	\addplot [line width=0.25mm, color = blue, solid]table [x={time},y={topFront},col sep=comma, each nth point=3, filter discard warning=false, unbounded coords=discard] {extentsRTinstability_At0dot82_Cn0dot00125.csv};
	\end{axis}
	\end{tikzpicture}
	
	\includegraphics{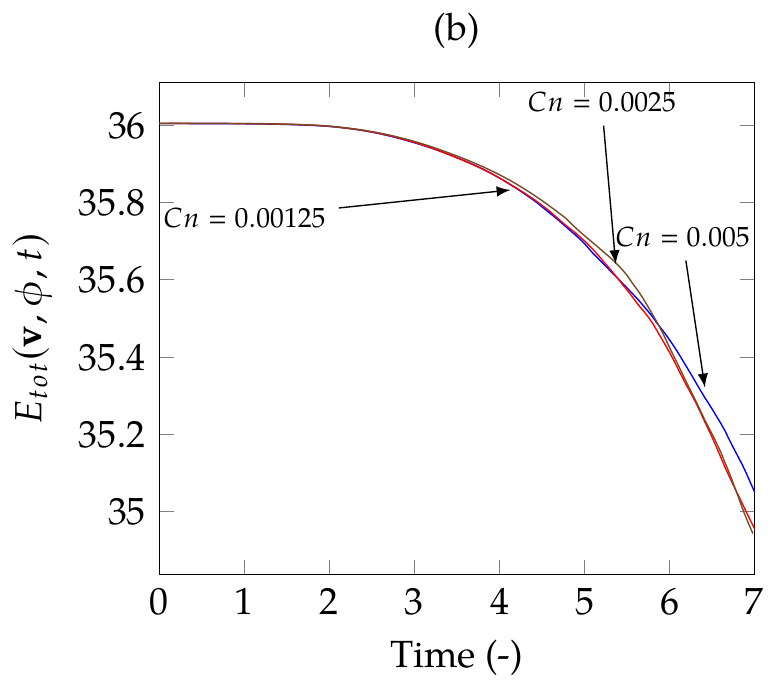}
	\includegraphics{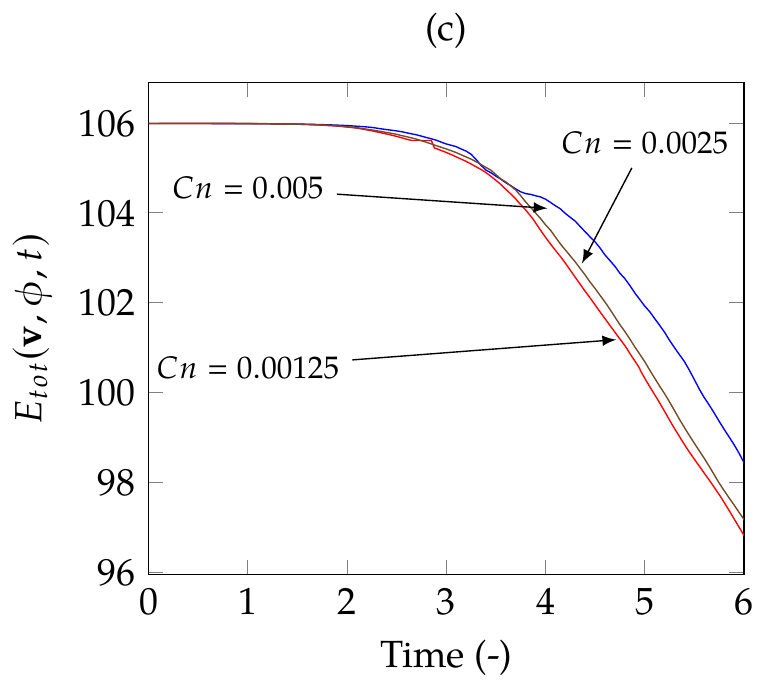}
	\caption{\textit{Rayleigh-Taylor instability 2D} (\cref{subsec:rayleigh_taylor_2D}): (a) Comparison of positions of top and bottom front of the interface with literature; (b) decay of the energy functional illustrating \cref{thrm:energy_stability} for $At = 0.5$; (c) decay of the energy functional illustrating \cref{thrm:energy_stability} for $At = 0.82$}
	\label{fig:RT_2D_comparison}
\end{figure}
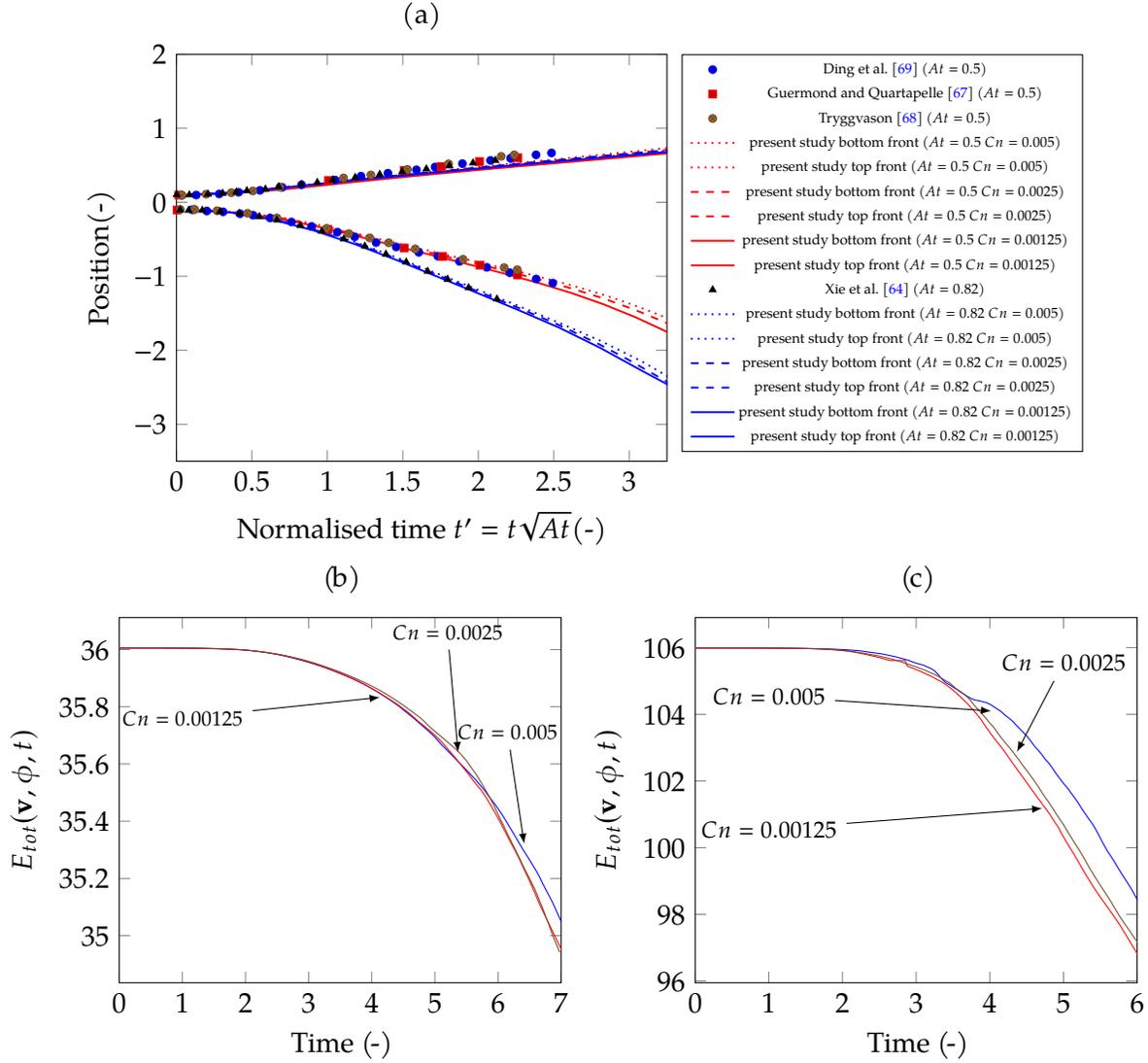

\begin{figure}[H]
	\centering	
	\includegraphics{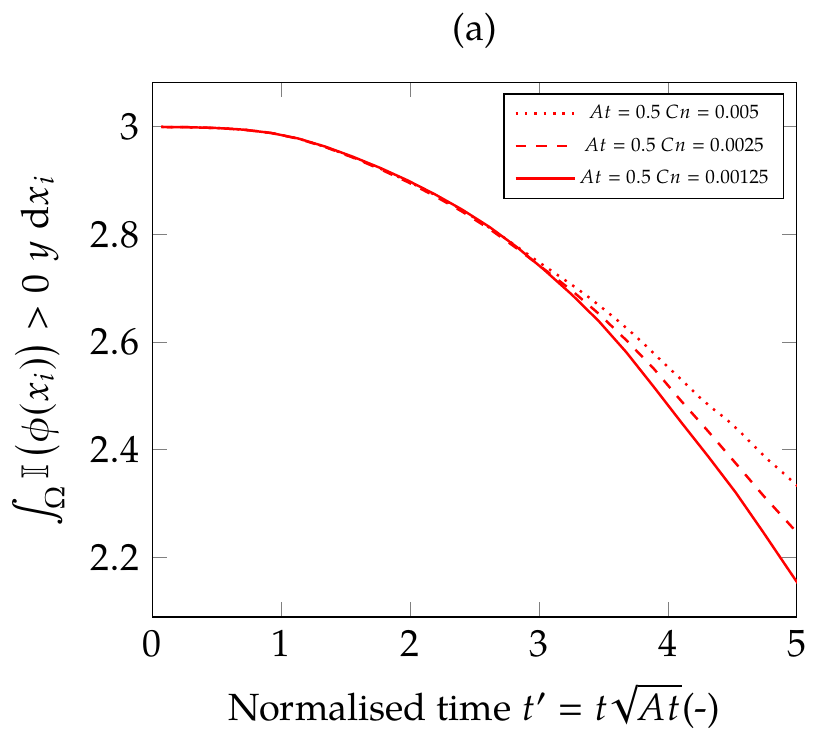}
	\includegraphics{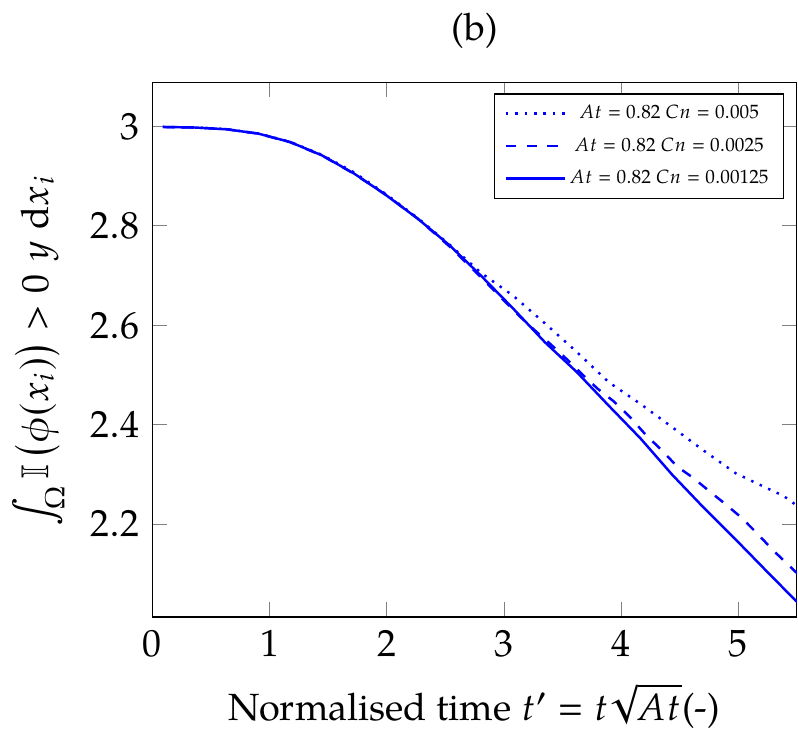}
	\caption{\textit{Rayleigh-Taylor instability 2D} (\cref{subsec:rayleigh_taylor_2D}): Comparison of centroids of heavy fluid for different $Cn$ numbers; (a) $At = 0.5$; (b) $At = 0.82$.}
	\label{fig:RT_2D_centroid}
\end{figure}

\begin{figure}[H]
	\centering	
		\subfigure [$Cn = 0.005$] {
			\includegraphics[width=0.22\linewidth]{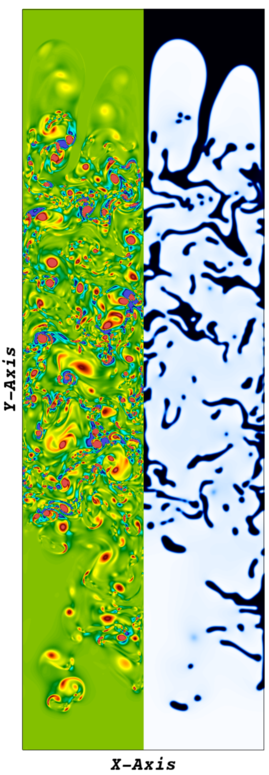}
			\label{subfig:rt_q_snap_1}
		} 
		\subfigure [$Cn = 0.0025$] {
			\includegraphics[width=0.22\linewidth]{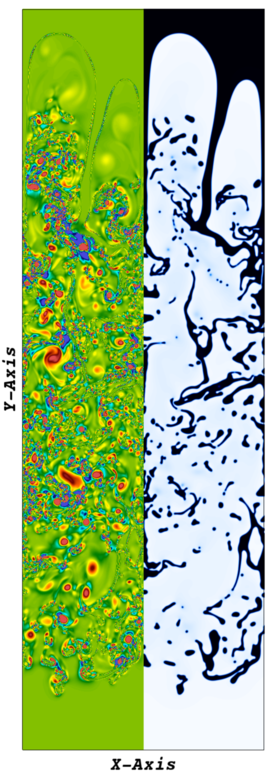}
			\label{subfig:rt_q_snap_2}
		} 
		\subfigure [$Cn = 0.00125$] {
			\includegraphics[width=0.22\linewidth]{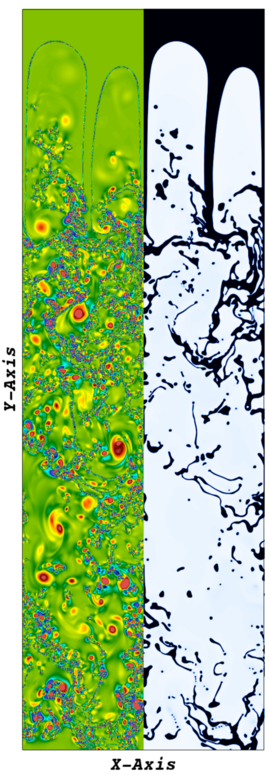}
			\label{subfig:rt_q_snap_3}
		}  
	\caption{\textit{Rayleigh-Taylor instability in 2D} (\cref{subsec:rayleigh_taylor_2D}): $Q-$criterion of Rayleigh Taylor instability for $At = 0.82$ (density ratio of 0.1) at $t^\prime = t\sqrt{At} = 5.6143$.  In each panel the left plot illustrates the $Q-$criterion, and the right plot shows corresponding interface location. These plots are zoomed insets of the domain near the interfacial instabilities.}
	\label{fig:rt2d_q}
\end{figure}

Finally, we report on the numerical mass conservation properties of the proposed scheme for the 2D Rayleigh-Taylor experiments. Panel (a) of \cref{fig:RT_mass_consv} shows the change in the total mass with respect to the initial total mass. We observe that it is of the order of $10^{-4}$. Therefore, we see excellent mass conservation even with a high amount of deformation of the interface over very large time horizons (over 30,000 time steps), especially in the presence of fine filaments that are clearly under-resolved for computationally tractable $Cn$ numbers. We see that there is some deterioration in mass conservation for the smallest $Cn$ number, again 
due to the under-resolution of thin filaments.  We see similar behavior of mass conservation for $At = 0.82$ for all three $Cn$ numbers; these results are shown in Panel (b) of \cref{fig:RT_mass_consv}.

 \begin{figure}[H]
	\centering
	\includegraphics{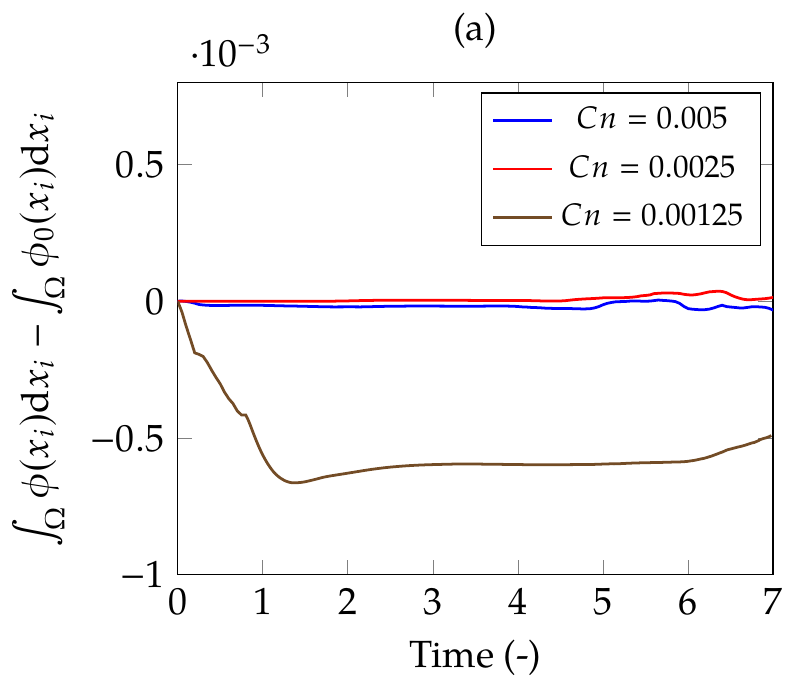}
	\includegraphics{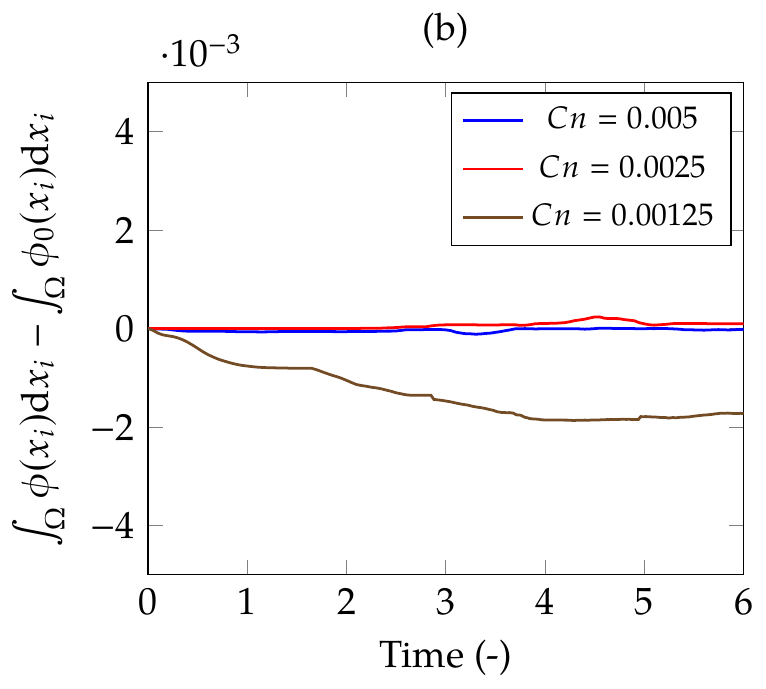}
	\caption{\textit{Rayleigh-Taylor instability 2D} (\cref{subsec:rayleigh_taylor_2D}): (a) total mass conservation (integral of total $\phi$) for $At = 0.5$; (b) total mass conservation (integral of total $\phi$) for $At = 0.82$.}
	\label{fig:RT_mass_consv}
\end{figure}

\subsection{3D simulations: Rayleigh-Taylor instability}
\label{subsec:rayleigh_taylor}
Next we deploy our framework in 3D and simulate the Rayleigh-Taylor instability in 3D using adaptive octree meshes.  For the 3D simulations we choose 
the following initial condition for $\phi$ to describe the interface:
\begin{align}
\phi(\vec{x}) &= \tanh\left(\sqrt{2} \left[ \frac{x_2 - h_0 - g\left(\vec{x}\right)}{Cn}\right]\right),\\ 
g(\vec{x}) &= 0.05 \left[\cos\left(2 \pi x_1 \right) + \cos\left(2 \pi x_3 \right)\right].
\label{eq:initialConditionRT}
\end{align}
Here $h_0$ is the location in the vertical direction for the interface, which in this case is chosen to be twice the characteristic length from the bottom of the channel. Typical simulations in the literature choose a rectangular domain that only captures one wavelength of the initial condition (e.g.,~\citep{ Tryggvason1990}).  To illustrate the advantage of the adaptive octree framework, we choose to include four wavelengths in the initial condition, resulting in a larger domain.  \Cref{fig:rayleighTaylorSetup} shows the initial condition, along with the schematics of the computational domain. We use a $Cn=0.0075$ and $At = 0.15$.  For this lower $At$ number simulation, the effect of non-zero surface tension is important. 
The non-dimensionalization follows the same logic as the 2D cases, with the Reynolds number set to 1000, the Weber number ($We = \rho_c g D^2/\sigma$) set to 1000, and viscosity ratio, $\nu_{+}/\nu_{-}$, set to 1.

\begin{figure}[H]
	\centering
	\includegraphics[width=0.4\linewidth]{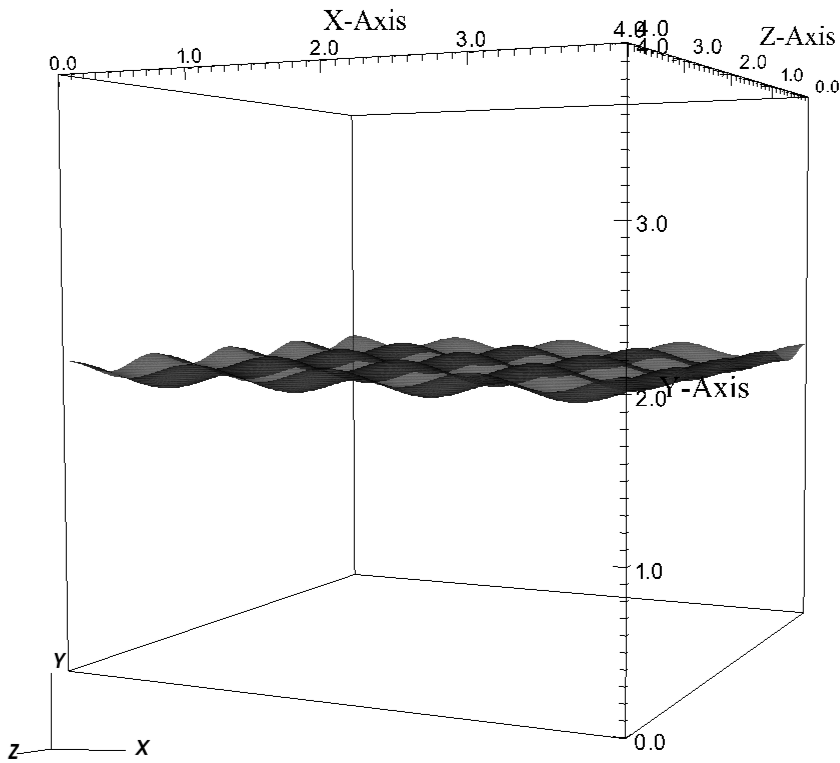}
	\caption{3D Rayleigh-Taylor instability (\cref{subsec:rayleigh_taylor}): shown here is  the computational domain with the iso-surface $\phi = 0$ showing the initial condition of the interface.}
	\label{fig:rayleighTaylorSetup}
\end{figure}

Due to the energy stability of the proposed numerical method we are able to take a reasonably large time-step size of $\delta t = 0.0025$.  We refine near the interface to a level corresponding to element length of $4/2^{8}$, ensuring about three elements for resolving the diffuse interface, while the refinement away from the interface is $4/2^4$.  Similar to the 2D cases, the boundary conditions are no-slip for velocity and no-flux for $\phi$ and $\mu$ on all the walls.  
In \ref{sec:app_linear_solve} we provide a detailed description of the preconditioners and linear solvers along with the command-line arguments that are used. The convergence criterion for all the 3D Rayleigh-Taylor simulations use a relative tolerance of $10^{-6}$ for Newton iteration and a relative tolerance of $10^{-6}$ for the linear solves within each Newton iteration. 


\Cref{fig:rt3d_mesh_At15} shows the evolution of the interface along with the solution-adapted mesh.  We color the mesh with the order parameter value (blue for the heavy fluid and white for the light fluid) to show the  evolution of the system.  It is seen that as the interface evolves it deforms and expands, causing the mesh density to gradually increase.  This gradual growth helps with the efficiency of the simulation, since a uniform mesh for this case would be computationally prohibitive.  
The efficient and scalable implementation of the approach allows us to run this large scale simulation on \Stampede~with 256 KNL nodes.  We present the scalability of the approach later in the scaling section.  

\Cref{fig:rt3d_interface_At15} shows that the initial sinusoidal perturbation develops into penetrating plumes of heavy fluid pushing down while the lighter fluid buckles and forms bubbles. The simulation maintains symmetry until the mixing becomes chaotic at long times; this is similar to the results in the 2D case.  As different parts of the interface move in opposite directions, Kelvin-Helmholtz instabilities cause the plumes to roll up, which in turn causes causes mushroom-like structures to develop (see~\cref{subfig:rt3d_int_At15_snap_4}).  
Although these two types of spikes (upwards/downwards) begin to develop in a checkerboard pattern that preserves symmetry, their dynamics are different due to the velocity differential that the two fronts face.  

The downward spikes undergo further deformation and we see the emergence of four long filaments from the mushroom structure (see~\cref{subfig:rt3d_int_At15_snap_7}) caused by the shear generated between the fluids. \citet{Liang2016} and~\citet{Jain2020}  also report four secondary filaments in their simulations for the same $At$ number, although their simulations were for zero surface tension (i.e., dynamics similar to miscible systems).  

On the other hand, the mushroom structures from the upward spikes develop into long and thin circular films.  The upward spikes develop circular films adjacent to the wall "bubbles" (i.e., structures near the wall that the heavier fluid generates as it is displaced by the lighter one) interact with these bubbles to merge and form larger structures (see~\cref{subfig:rt3d_int_At15_snap_8}). While the central plumes have little-to-no interaction with the wall, the bubbles continue to rise and ultimately collide with the top wall.  

Another important difference between the upward and downward spikes is their rate of growth.  \Crefrange{subfig:rt3d_int_At15_snap_5}{subfig:rt3d_int_At15_snap_8} show that the fronts of the upward spikes move slower than the fronts of the downward spikes due to the density differential.  \Cref{subfig:rt3d_int_At15_snap_9} shows that the mushroom structures collide with top and bottom walls, which then leads to further breakup that
creates the conditions for chaotic mixing. To the best of our knowledge, this is the first analysis in the literature of the dynamics for multiple wave single-mode instabilities. 

\begin{figure}[]
	\centering	
	\begin{tabular}{p{0.32\textwidth}p{0.32\textwidth}p{0.32\textwidth}}
		\subfigure [$t = 0.0$] {
			\includegraphics[width=\linewidth]{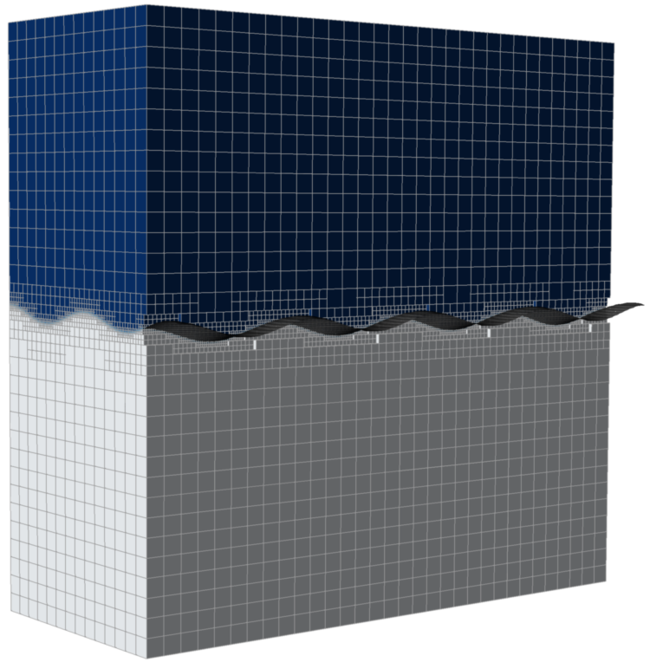}
			\label{subfig:rt3d_mesh_At15_snap_1}
		} &
		\subfigure [$t = 3.75$] {
			\includegraphics[width=\linewidth]{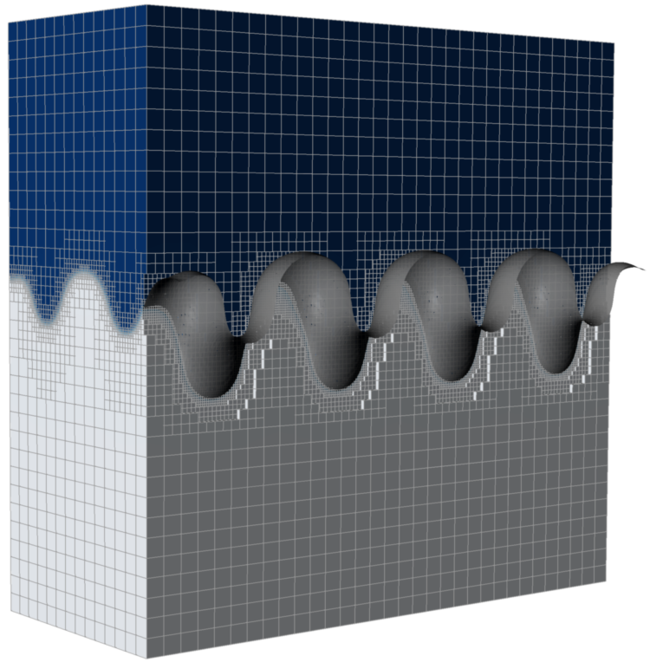}
			\label{subfig:rt3d_mesh_At15_snap_2}
		} & 
		\subfigure [$t = 5.0$] {
			\includegraphics[width=\linewidth]{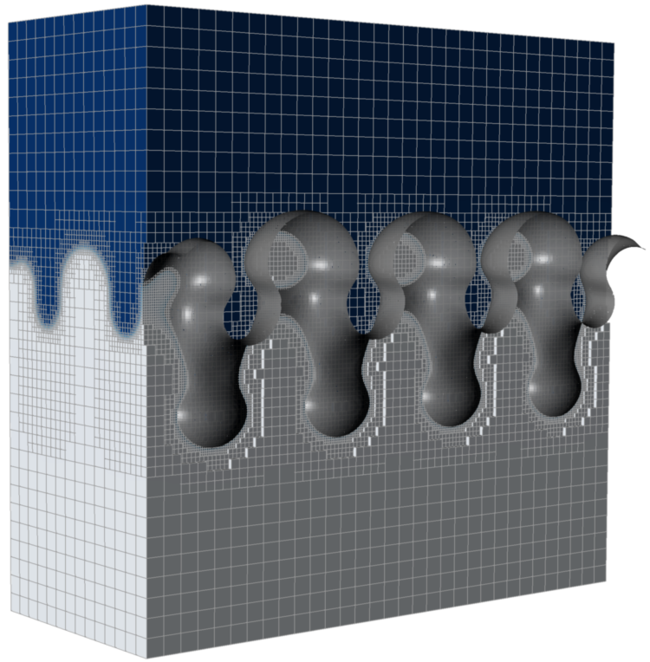}
			\label{subfig:rt3d_mesh_At15_snap_3}
		} \\
		
		\subfigure [$t = 6.75$] {
			\includegraphics[width=\linewidth]{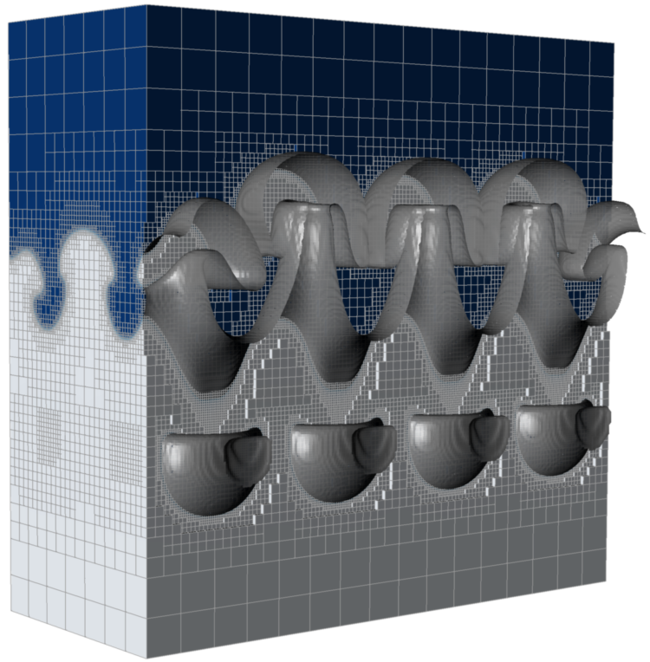}
			\label{subfig:rt3d_mesh_At15_snap_4}
		} &
		\subfigure [$t = 7.625$] {
			\includegraphics[width=\linewidth]{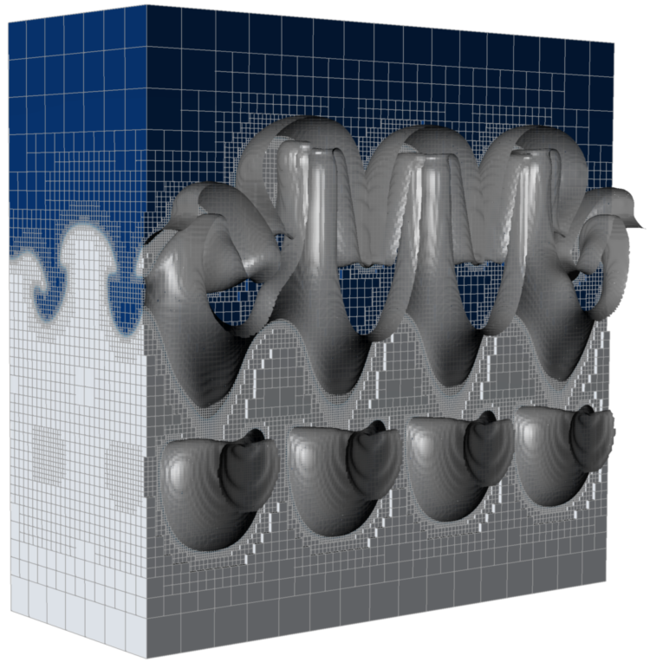}
			\label{subfig:rt3d_mesh_At15_snap_5}
		} & 
		\subfigure [$t = 8.5$] {
			\includegraphics[width=\linewidth]{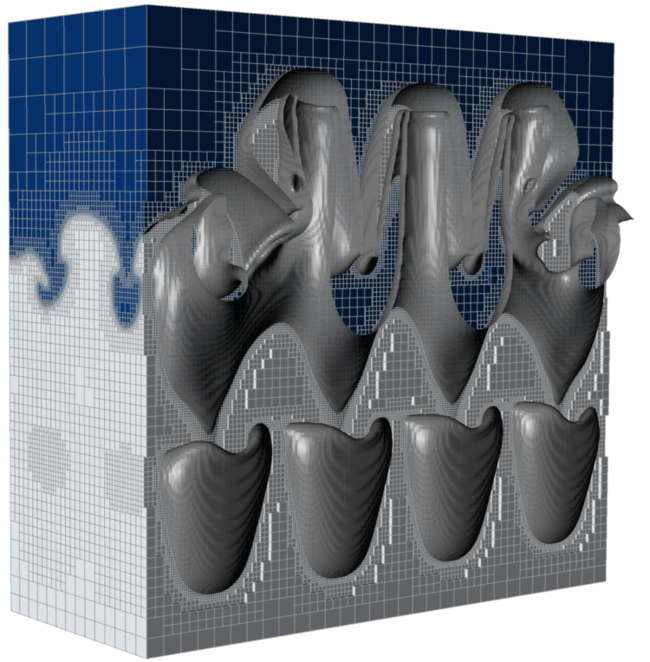}
			\label{subfig:rt3d_mesh_At15_snap_6}
		} \\
		
		\subfigure [$t = 9.375$] {
			\includegraphics[width=\linewidth]{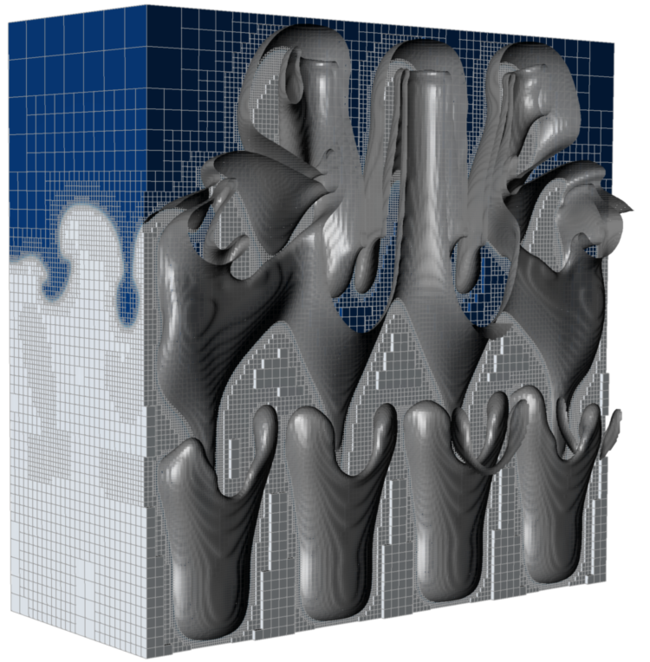}
			\label{subfig:rt3d_mesh_At15_snap_7}
		} &
		\subfigure [$t = 10$] {
			\includegraphics[width=\linewidth]{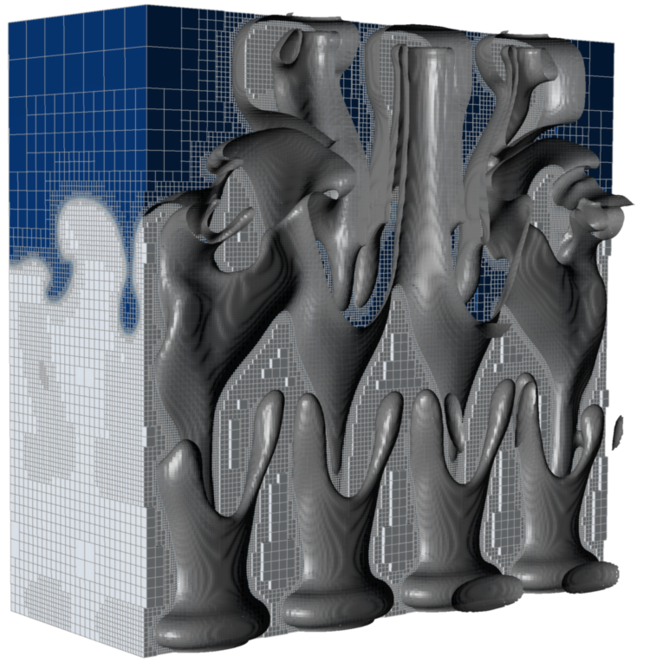}
			\label{subfig:rt3d_mesh_At15_snap_8}
		} & 
		\subfigure [$t = 11.25$] {
			\includegraphics[width=\linewidth]{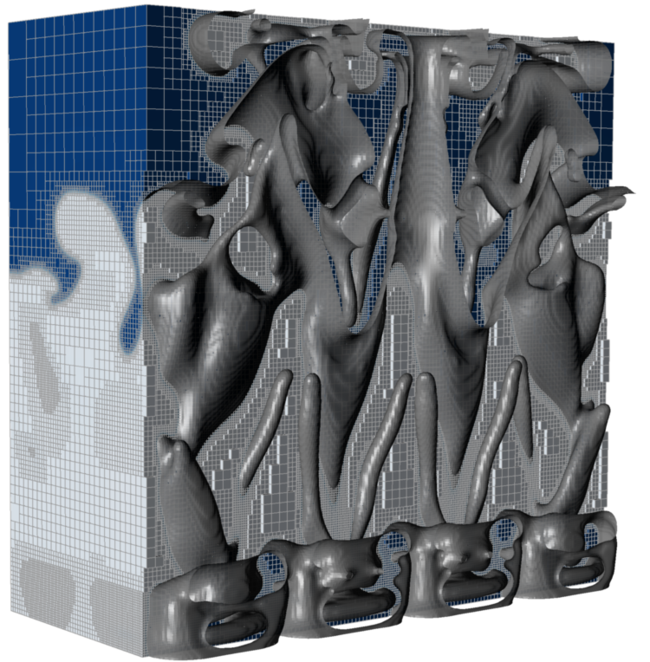}
			\label{subfig:rt3d_mesh_At15_snap_9}
		} \\
	\end{tabular}
	\caption{\textit{Rayleigh-Taylor instability in 3D} (\cref{subsec:rayleigh_taylor}): Snapshots of the mesh at various time-points in the simulation for Rayleigh-Taylor instability for $At = 0.15$.  The figures show half of the mesh of the actual domain to illustrate the refinement around the interface of two fluids represented by the gray iso-surface of $\phi = 0$. The phase field $\phi$ values color the mesh, where blue represents heavy fluid and white represents light fluid.  Here $t $(-) is the non-dimensional time.}
	\label{fig:rt3d_mesh_At15}
\end{figure}

\begin{figure}[]
	\centering	
	\begin{tabular}{p{0.32\textwidth}p{0.32\textwidth}p{0.32\textwidth}}
		\subfigure [$t = 0.0$] {
			\includegraphics[width=\linewidth]{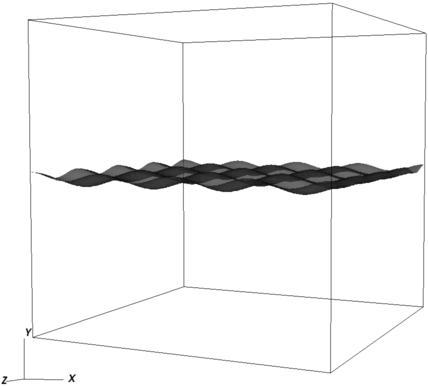}
			\label{subfig:rt3d_int_At15_snap_1}
		} &
		\subfigure [$t = 3.75$] {
			\includegraphics[width=\linewidth]{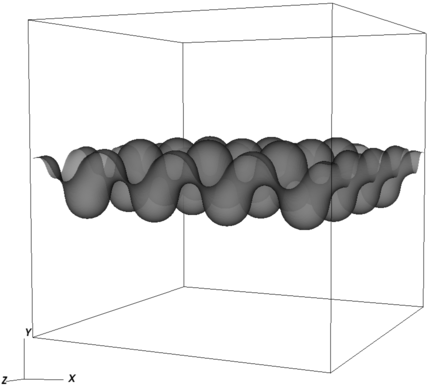}
			\label{subfig:rt3d_int_At15_snap_2}
		} & 
		\subfigure [$t = 5.0$] {
			\includegraphics[width=\linewidth]{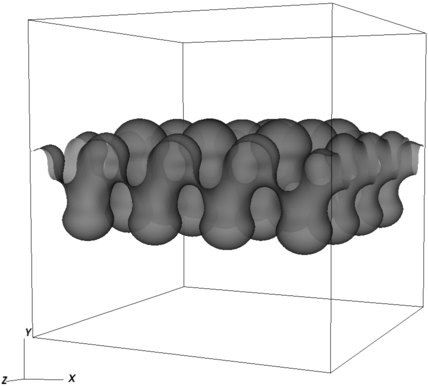}
			\label{subfig:rt3d_int_At15_snap_3}
		} \\
		
		\subfigure [$t = 6.75$] {
			\includegraphics[width=\linewidth]{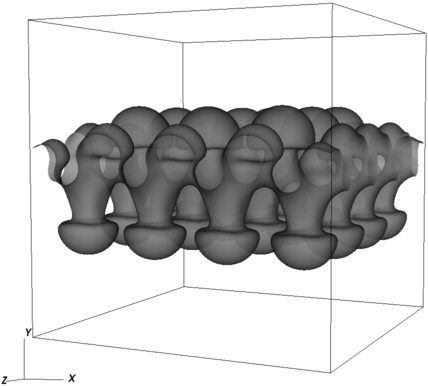}
			\label{subfig:rt3d_int_At15_snap_4}
		} &
		\subfigure [$t = 7.625$] {
			\includegraphics[width=\linewidth]{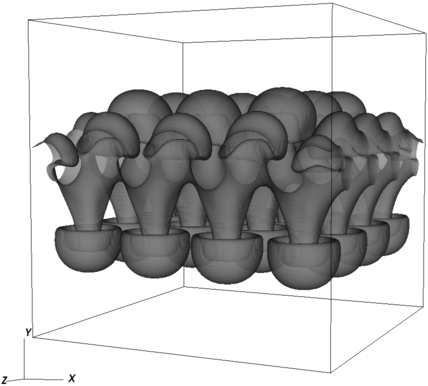}
			\label{subfig:rt3d_int_At15_snap_5}
		} & 
		\subfigure [$t = 8.5$] {
			\includegraphics[width=\linewidth]{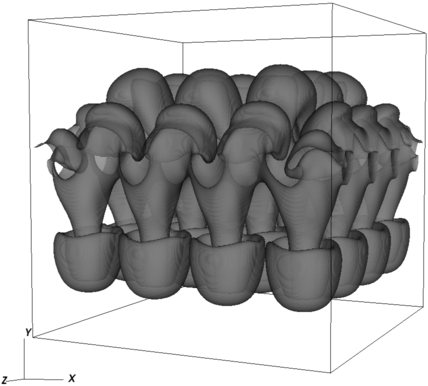}
			\label{subfig:rt3d_int_At15_snap_6}
		} \\
		
		\subfigure [$t = 9.375$] {
			\includegraphics[width=\linewidth]{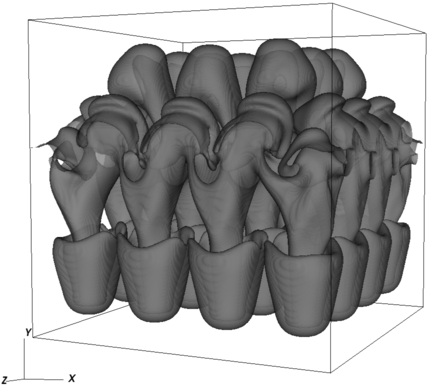}
			\label{subfig:rt3d_int_At15_snap_7}
		} &
		\subfigure [$t = 10$] {
			\includegraphics[width=\linewidth]{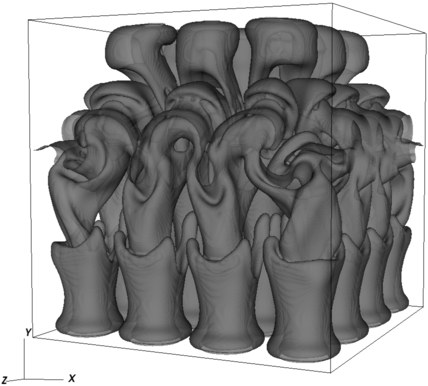}
			\label{subfig:rt3d_int_At15_snap_8}
		} & 
		\subfigure [$t = 11.25$] {
			\includegraphics[width=\linewidth]{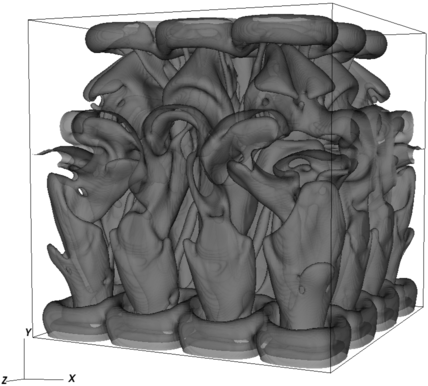}
			\label{subfig:rt3d_int_At15_snap_9}
		} 
	\end{tabular}
	\caption{\textit{Rayleigh-Taylor instability in 3D} (\cref{subsec:rayleigh_taylor}): Snapshots of the zero isosurface of $\phi$ which represents the interface at various time-points in the simulation for Rayleigh-Taylor instability for $At = 0.15$.  Here $t $(-) is the non-dimensional time.}
	\label{fig:rt3d_interface_At15}
\end{figure}


\subsection{3D simulations: lid-driven cavity}
\label{subsec:ldc_3d}

Our next example also exhibits considerable deformation of the interface.  The setup consists of a regularized lid-driven cavity, such that there are no corner singularities due to the imposed velocities.  Half of the cubical domain contains fluid~1, the other half contains fluid ~2. The fluids have finite surface tension.  \Cref{fig:ldcSetup}(a) shows the domain. Initially, the interface between the two fluids is flat. We use the length of the domain as the non-dimensionalizing length scale.  The key non-dimensional numbers are  $Re = 100$ (laminar), $We = 100$, $Cn = 0.005$, and $Pe = 13333$. The viscosity and density ratios are 1. We adaptively refine the mesh with the finest element size near the interface being $1/(2^8)$, the coarsest element size away from the interface being $1/(2^4)$, and wall element size is $1/(2^6)$.  Boundary conditions are no-slip for velocity on all the walls except the top wall, and no flux conditions for $\phi$ and $\mu$. For the top wall boundary conditions the $y$ and $z$-velocities are set to zero, while the   
$x$-velocity is given by following function:
\begin{equation}
v_1 = 2^4 x_1 \left(1 - x_1\right) x_3 \left(1-x_3\right).
\end{equation}
The velocity goes to zero on all the corners, thus avoiding corner singularities.  It is worthwhile to note that while several authors (\citet{ Chakravarthy1996, Chella1996, Park2016}) have studied this problem in 2D; to our best knowledge it has not been explored in 3D.  

The dynamics of the system with $Re = 100$ is in the laminar regime. In single-phase systems, a vortex develops near the top wall as the cavity flow evolves. This vortex forms for the two-phase case, as~\cref{fig:ldcSetup}(b) displays.  The vortex in~\cref{fig:ldcSetup}(b) rolls up the flat interface up the sidewall and causes a thin wetting layer to form on the top wall. \Cref{fig:ldc_interface} shows how this layer progressively moves until it rolls back into the center of the box. The interface rapidly rolls up as it moves towards the domain center; it moves away from walls since the velocity goes to zero near the wall. \citet{Chakravarthy1996, Chella1996}, and \citet{Park2016} described this behavior in 2D simulations. The stability of the thin layer at the top wall depends on the surface tension strength ($We$ number).  We simulate a moderate $We$ number which allows mixing and deformation but the film at the top is stable; for larger $We$ numbers the film at the top starts breaking up into bubbles.  We use a relative tolerance of $10^{-6}$ for Newton iteration.  For the linear solves within each Newton iteration we use a relative tolerance of $10^{-6}$.

The previous two examples illustrate the capability of the framework to capture fairly complex interfacial motion in 3D due to its robust and efficient implementation, which we discuss next.
\begin{figure}[H]
	\centering
	\begin{tabular}{cc}
	(a) \, \includegraphics[width=0.4\linewidth]{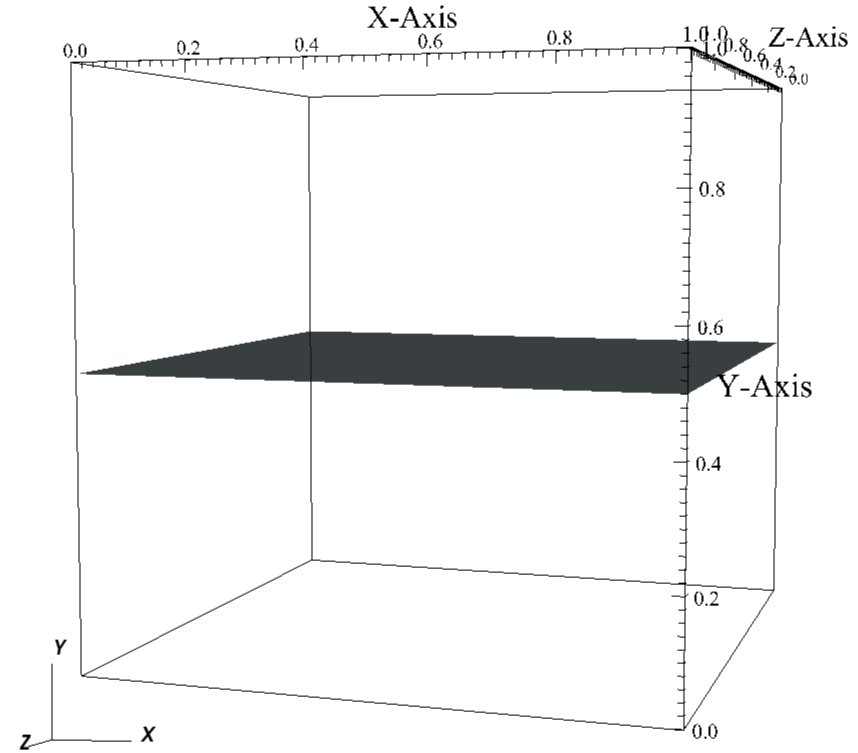} &
    (b) \, \includegraphics[width=0.4\linewidth]{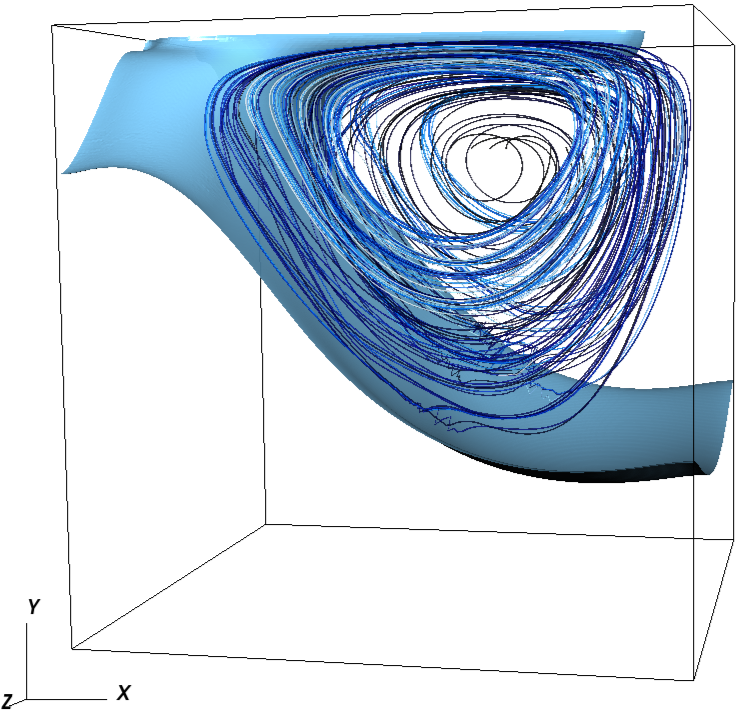}
    \end{tabular}
	\caption{The lid-driven cavity problem (\cref{subsec:ldc_3d}): (a) schematic of the computational domain with the iso-surface $\phi = 0$ showing the initial condition of the interface;
	(b) streamlines of velocity overlayed on the interface at $t = 16.5$.}
	\label{fig:ldcSetup}
\end{figure}


\begin{figure}[]
	\centering	
	\begin{tabular}{p{0.32\textwidth}p{0.32\textwidth}p{0.32\textwidth}}
		\subfigure [$t = 0.0$] {
			\includegraphics[width=\linewidth]{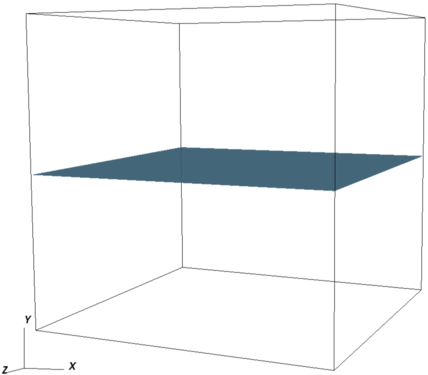}
			\label{subfig:ldc_int_snap_1}
		} &
		\subfigure [$t = 6.25$] {
			\includegraphics[width=\linewidth]{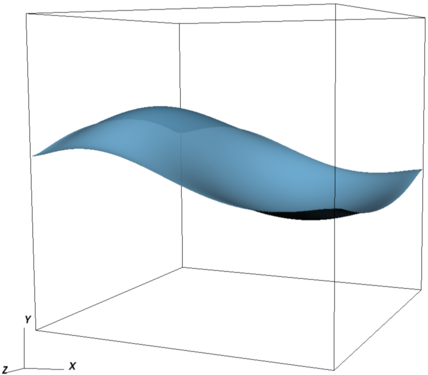}
			\label{subfig:ldc_int_snap_2}
		} & 
		\subfigure [$t = 8.75$] {
			\includegraphics[width=\linewidth]{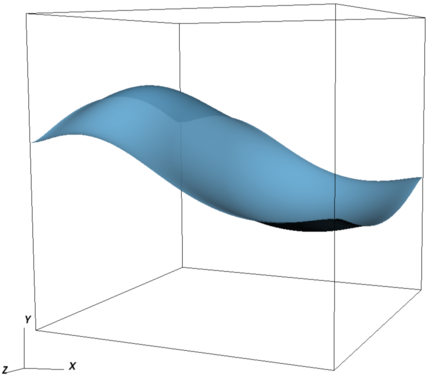}
			\label{subfig:ldc_int_snap_3}
		} \\
		
		\subfigure [$t = 11.25$] {
			\includegraphics[width=\linewidth]{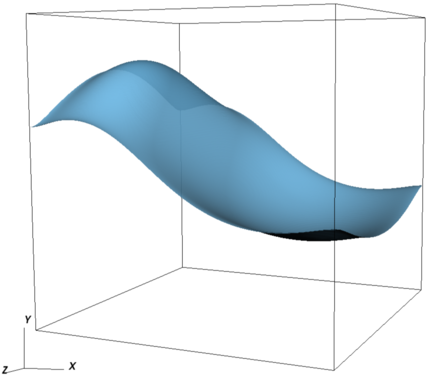}
			\label{subfig:ldc_int_snap_4}
		} &
		\subfigure [$t = 13.75$] {
			\includegraphics[width=\linewidth]{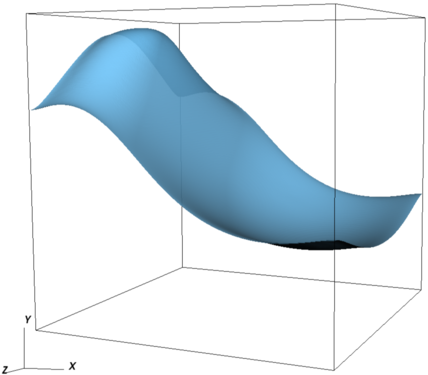}
			\label{subfig:ldc_int_snap_5}
		} & 
		\subfigure [$t = 16.25$] {
			\includegraphics[width=\linewidth]{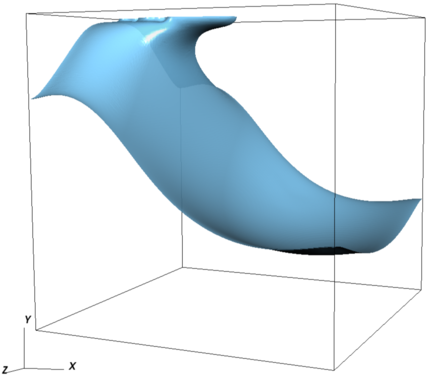}
			\label{subfig:ldc_int_snap_6}
		} \\
		
		\subfigure [$t = 18.75$] {
			\includegraphics[width=\linewidth]{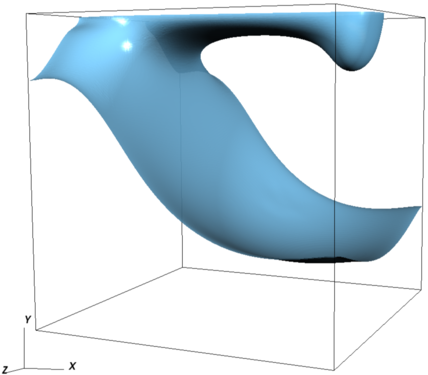}
			\label{subfig:ldc_int_snap_7}
		} &
		\subfigure [$t = 22.5$] {
			\includegraphics[width=\linewidth]{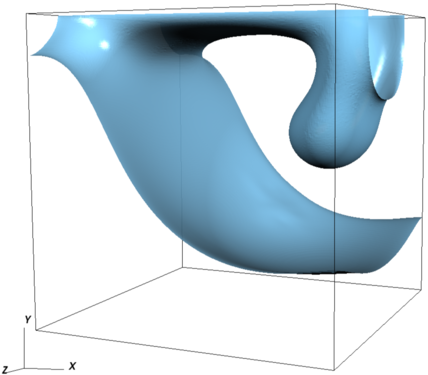}
			\label{subfig:ldc_int_snap_8}
		} & 
		\subfigure [$t = 25.25$] {
			\includegraphics[width=\linewidth]{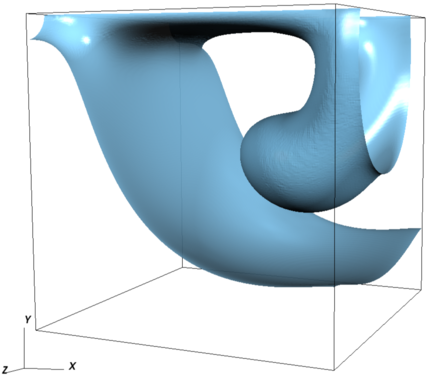}
			\label{subfig:ldc_int_snap_9}
		} 
	\end{tabular}
	\caption{\textit{Mixing of two fluids in a cubic driven cavity} (\cref{subsec:ldc_3d}): Snapshots of the zero iso-surface of $\phi$ which represents the interface at various time-points in the simulation for lid-driven cavity problem.  Here $t $(-) is the non-dimensional time.}
	\label{fig:ldc_interface}
\end{figure}
\section{Scaling of the numerical implementation}
\label{sec:scaling}
\subsection{Strong scaling}

We perform a detailed timing analysis to demonstrate the parallel scalability of the framework.  We perform all scaling tests on TACC's ~\Stampede~ Knights Landing processors ranging from~1 node to~256 nodes with~68 processors per node.  We study in detail the lid-driven cavity case in~\cref{subsec:ldc_3d} for the scaling analysis.  We run each scaling experiment for ten time steps such that the initialization and setup do not dominate the timing. Three different levels of refinement characterize each mesh: 1. background mesh refinement ($L_{\mathrm{bkg}}$); 2. wall refinement ($L_{\mathrm{wall}}$); and 3. interface refinement ($L_{\mathrm{interface}}$). As the interface evolves, the mesh is subsequently refined near the interface and coarsened away from it.  \Cref{tab:refinement} shows the level of refinement and the approximate number of total elements for each of the three different meshes used for the scaling study.
\begin{table}[H]
\centering
\normalsize
\resizebox{0.3\linewidth}{!}{%
\begin{tabular}{|c|ccc|c|}
\hline
\hline
& $L_{\mathrm{bkg}}$ & $L_{\mathrm{wall}}$ & $L_{\mathrm{interface}}$ & $N_{\mathrm{elem}}$ \\ \hline
\textsc{M1} & 3 & 5 & 7 & 280 K\\
\textsc{M2} & 4 & 6 & 8 & 1 M\\
\textsc{M3} & 5 & 7 & 9 & 19 M\\
\hline
\end{tabular}%
}
\caption{Level of refinements involved in different meshes used for scaling studies}
\label{tab:refinement}
\end{table}

Panel~(a) of~\cref{fig:strong_scaling} shows the result of strong scaling for three different meshes, and panel~(b) shows the corresponding relative speedup. \Cref{fig:strong_scaling} demonstrates that we achieve a better than ideal scaling for these cases. To analyze the reasons behind this, we analyze component-by-component the framework timings and observe that the two dominant ones in the total run time are the Jacobian assembly ($J$ in~\cref{sec:numerical_tecniques}) and the preconditioner setup (PC setup in~\petsc), with PC setup dominating the timing (see~\cref{fig:scaling-fraction}).  Therefore, to understand the scaling behavior, we plot the strong scaling curves and corresponding relative speed up for these two components in~\cref{fig: strongScalingDetailed}.   Panels~(a) and~(b) of~\cref{fig: strongScalingDetailed} detail the scaling for Jacobian assembly, whereas, panels~(c) and~(d) show the scaling for PC setup.  As expected, the Jacobian assembly scales almost ideally for all the cases considered from panels~(a) and~(b) of~\cref{fig: strongScalingDetailed}.  However, the PC setup has better than ideal scaling, see panels~(c) and~(d).  As the PC setup dominates the total solution time, the overall scaling mimics this behavior.  


\tikzexternalenable
\begin{figure}[H]
	\centering
	\includegraphics{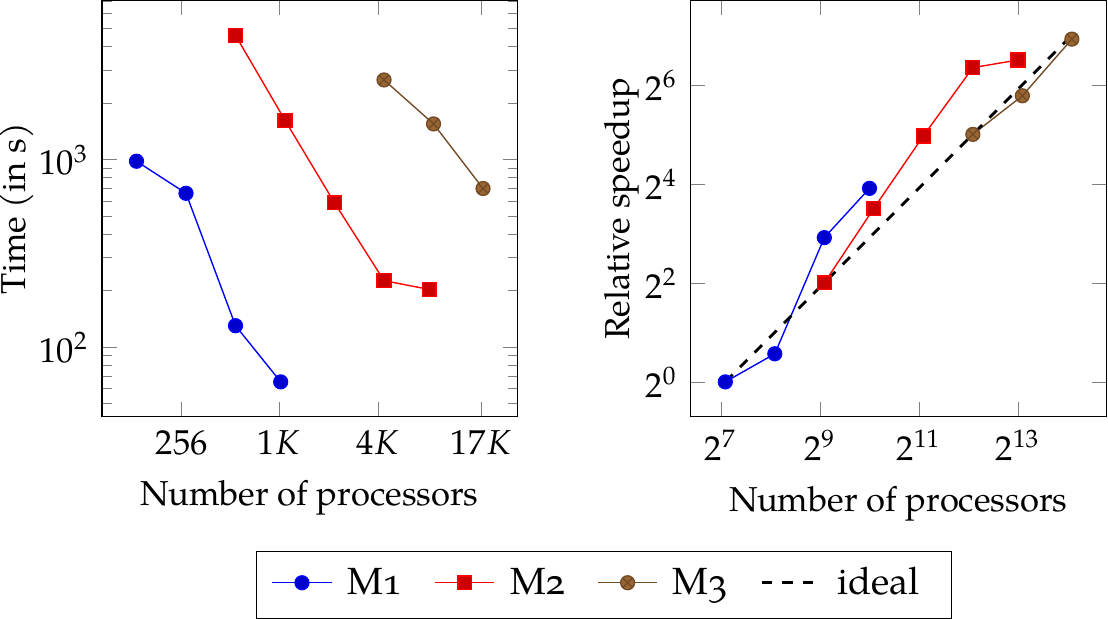}
	\caption{\textit{Strong scaling:} Shown in the panels are (a) the strong scaling behavior and (b) the relative speedup for total solve time for our solver on \Stampede~ Knights Landing processor. Three different mesh are considered: \textsc{M1} with 280 K elements, \textsc{M2} with 1 M elements and \textsc{M3} with 19 M elements. Excellent scaling behavior is observed up to
	 $\mathcal{O}(17K)$ processors.}
	\label{fig:strong_scaling}
\end{figure}

\begin{figure}[H]
	\centering
	\includegraphics{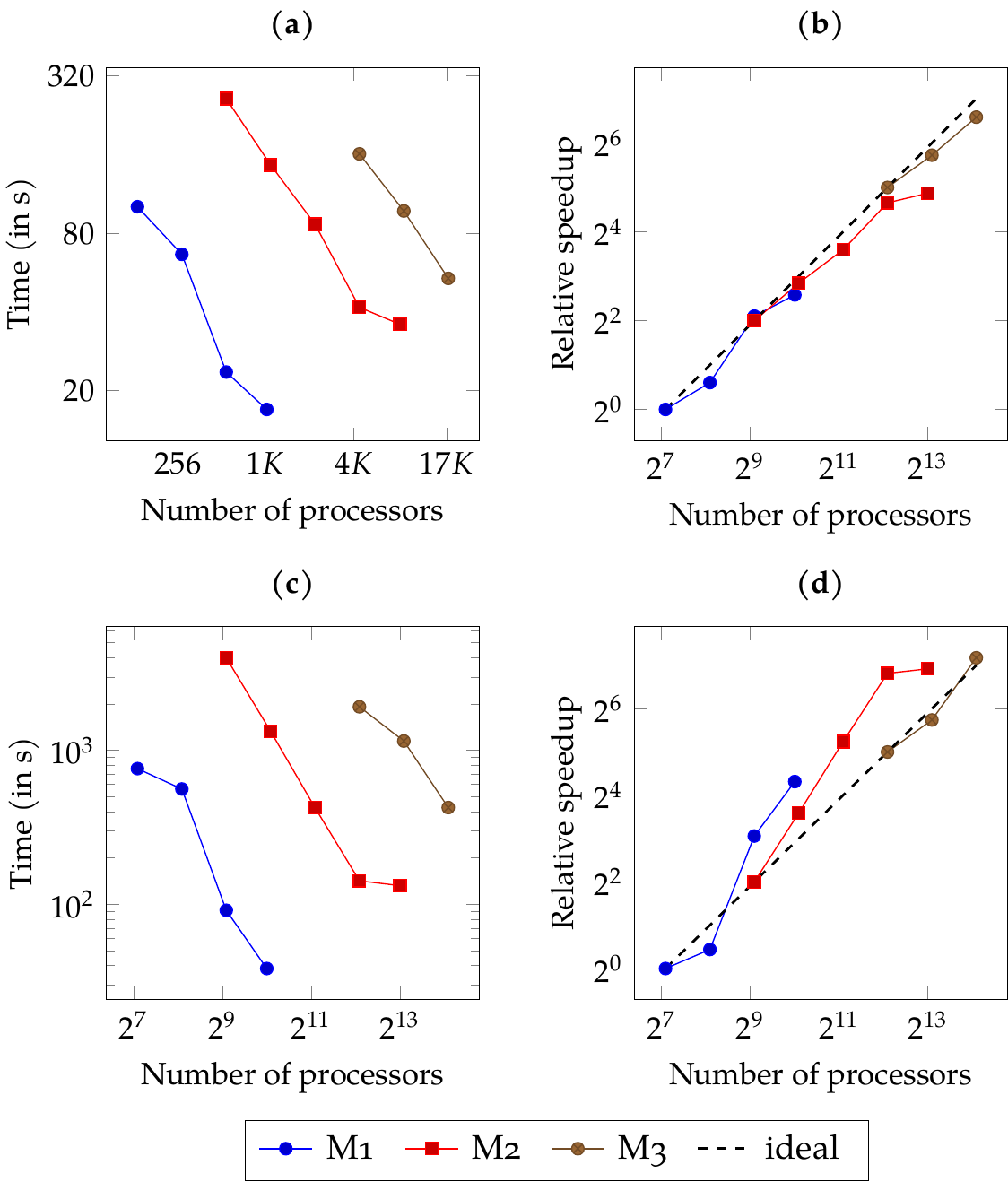}
	\caption{\textit{Strong scaling of components:} Figure showing the segregated strong scaling behaviour of two major component of our solver on \Stampede~ Knights Landing processor: (a) shows Jacobian assembly strong scaling; (b) shows relative speedup for Jacobian assembly;  (c) shows preconditioner setup; (d) shows relative speedup for Preconditioner setup. }
	\label{fig: strongScalingDetailed}
\end{figure}
\tikzexternaldisable
The deep memory hierarchy inherent in modern-day clusters may be the cause for the better than ideal scaling behavior.  We use an additive Schwarz-based block preconditioning with each block using LU  factorization (-\texttt{sub\_pc\_type} in~\petsc) (see~\ref{sec:app_linear_solve}) to compute the factors of the block matrix. The factorized block preconditioned matrix loses its sparsity, and the resultant matrix is less sparse ($b \times b$) matrix, where $b$ is the block size of the matrix on a processor. This denser matrix can no longer fit in L1, L2, or L3 cache for a big enough problem size. As the number of processors increases, the resulting problem size in a processor diminishes. Thus, these denser blocks begin to fit in cache, and therefore, we achieve a better than ideal speedup for PC setup and subsequent solve time. 

Figure~\ref{fig:scaling-fraction} shows the relative fraction of time spent in different meshes. 
%
We observe two specific trends. First, as the number of processors increases such that number of elements remains constant, the relative cost of PC setup decreases. This is because of the fact that LU factorization is performed on a smaller block matrix. Also, the cost of communication while performing matrix assembly and vector assembly increases with increase in number of processor, whereas LU factorization is performed on a block matrix and thus requires no communication. Secondly, increasing the number of elements such that the number of processor is fixed has a substantial effect on the relative cost of PC setup. LU factorization is an $\mathcal{O}(N^3)$ operation, whereas the number of FLOPS and communication involved in other operations like matrix assembly is atmost $\mathcal{O}(N^2)$. This explains the increase in the PC setup cost.
Future work will seek to develop schemes that will efficiently use cheaper preconditioners, whilst maintaining the energy stability and mass conservation of these schemes.


\begin{figure}[H]
\centering
\resizebox{0.50\textwidth}{!}{%
	\includegraphics{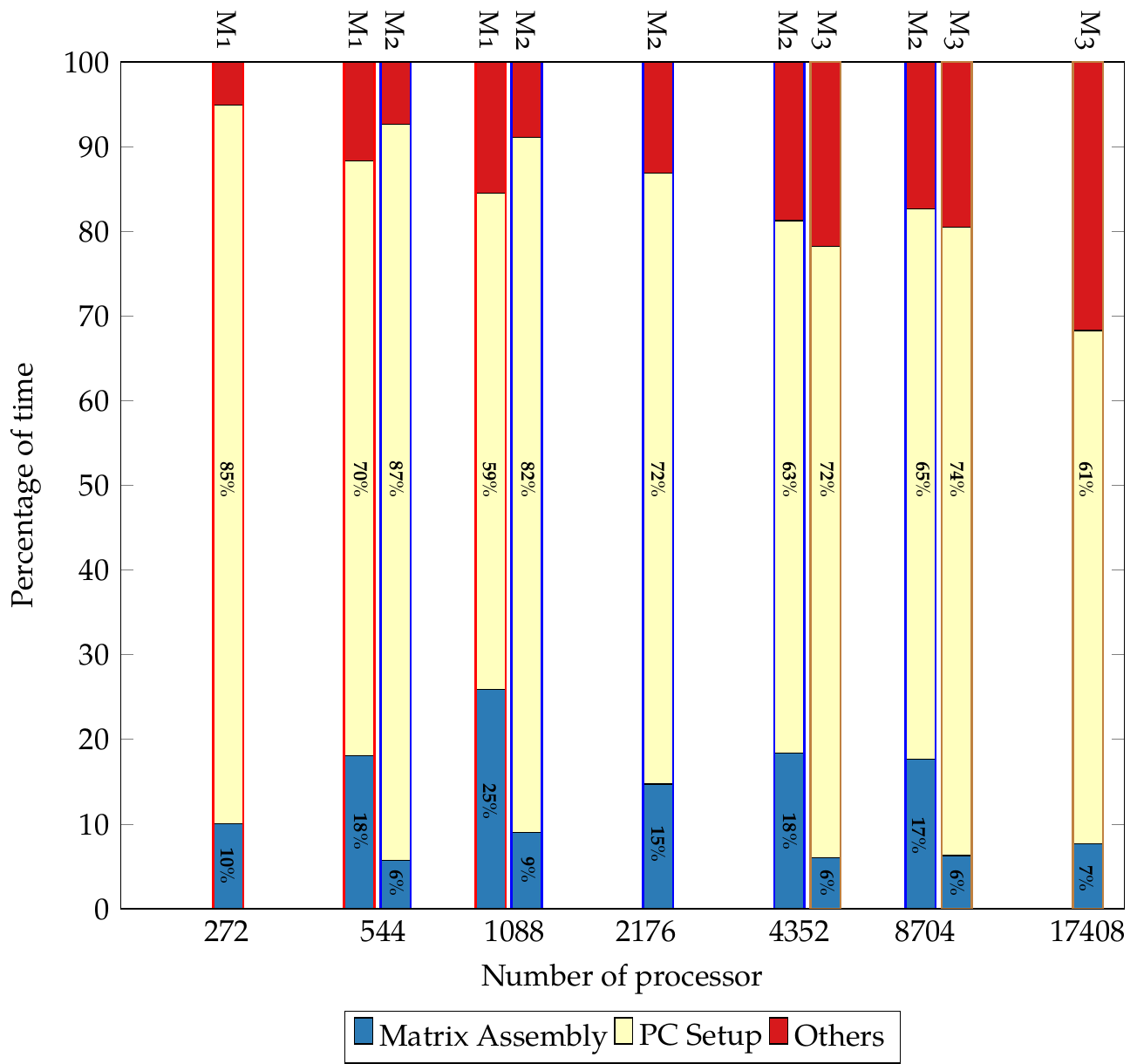}
}
\caption{Percentage of time taken for individual components for three different meshes used for strong scaling case. The meshes are labelled at the top.}
\label{fig:scaling-fraction}
\end{figure}
\begin{remark}
We use a Krylov space solver (BCGS) with ASM/LU preconditioning. Therefore, algebraic multi-grid solvers with \petsc's AMG can substantially improve the solve time.  Although we had some success with this setup (e.g., the 2D Rayleigh Taylor in~\cref{subsec:rayleigh_taylor_2D} cases use this setup), the 3D examples use the ASM/LU combination.  
\end{remark}
 
\subsection{Weak scaling}
Performing exact weak scaling \footnote{Weak scaling studies the variation in simulation time as we increase the number of processors with a fixed problem size per processor.} is non--trivial due to the adaptive nature of the mesh. Therefore, we consider  simulations with an approximately equal number of elements per processor to deduce the weak scaling nature of our solver.  \Cref{fig:weakScaling} shows the weak scaling results for the different problem sizes. Overall, similar to the strong scaling, we see an excellent weak scaling efficiency for different problem sizes. For each of the cases, we observe a weak scaling efficiency of greater than 0.5 for a 64 fold increase in the problem size. We also analyze the weak scaling for the dominant components of the overall time, i.e. Jacobian assembly and PC setup.  \Cref{fig:weakScalingDetailed} segregates the time for these components. The weak scaling efficiency ($> 0.5$) for each of these individual components is excellent as we vary the problem size by two orders of magnitude.
\begin{figure}[H]
	\centering
	\includegraphics{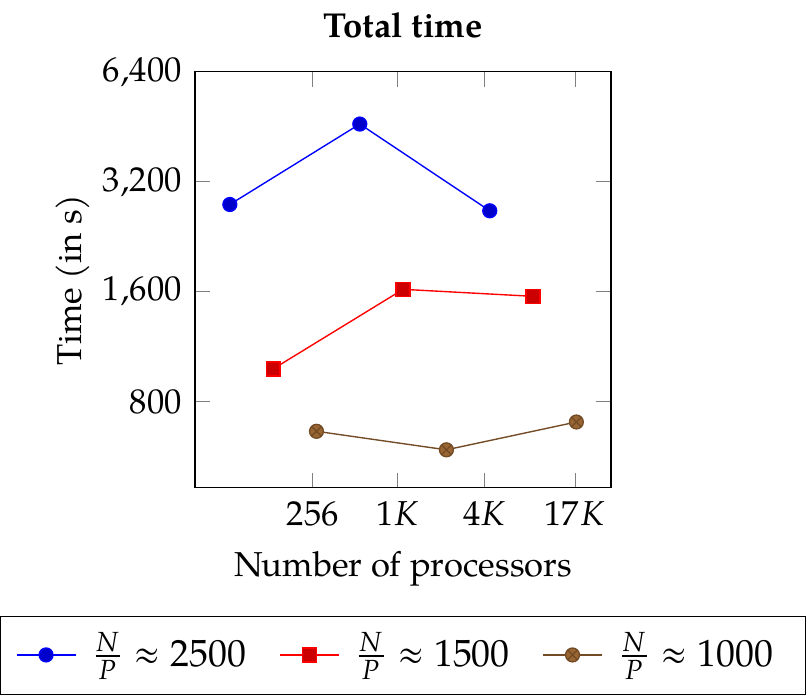}
	\caption{\textit{Weak scaling:} Figure showing weak scaling behaviour of total solve time of our solver on \Stampede~ Knights Landing processor for different number of elements per processor.}
	\label{fig:weakScaling}
\end{figure}

\begin{figure}[H]
	\centering
	\includegraphics{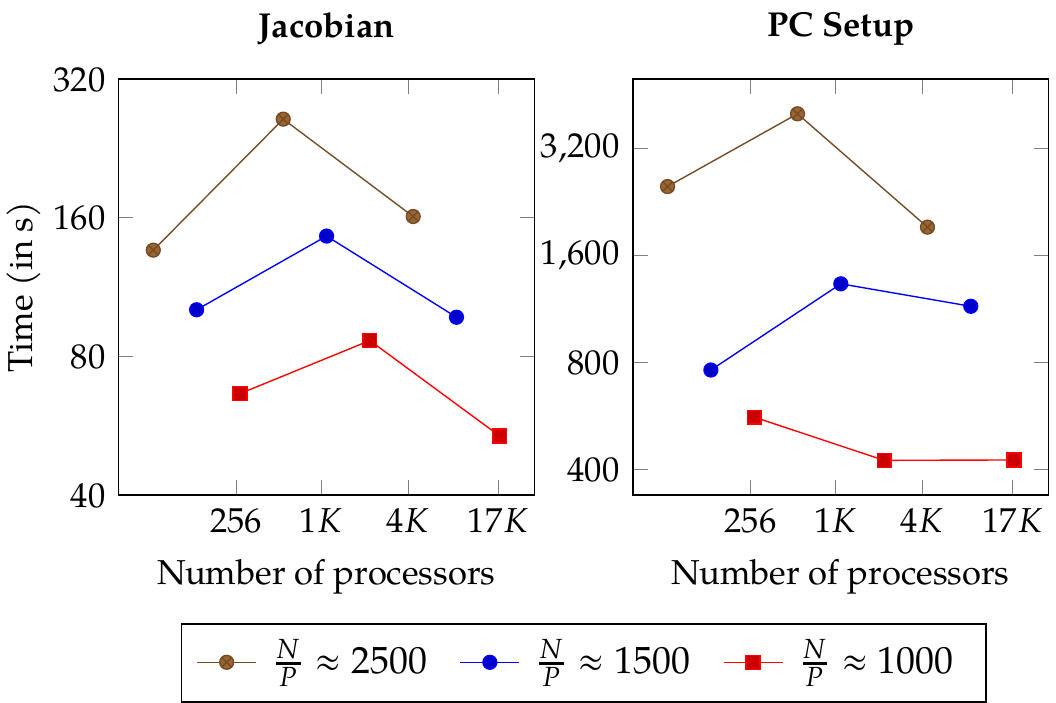}
	\caption{\textit{Weak scaling:} Figure showing the segregated weak scaling behavior of two major components of our solver on \Stampede~ Knights Landing processor for different number of elements per processor: Jacobian assembly and  preconditioner setup.}
	\label{fig:weakScalingDetailed}
\end{figure}
The preconditioner complexity and cost challenges (as seen from~\cref{fig:scaling-fraction}) are also present in the block iterative method of~\citet{ Khanwale2020},  where one performs these operation multiple times within one timestep (once every block iteration). This combined cost makes the block iterative method much more expensive compared to the fully coupled solver. Therefore, the current coupled framework improves the overall time to solve by reducing the number of evaluations required for Jacobian assembly and PC setup. In challenging application problems, preconditioning the linear problem (sub-iteration within each Newton solve)  requires substantial computational effort; future work will seek improve the speed and efficiencies of the preconditioning.

\section{Conclusions and future work}

In this work we developed a numerical framework for solving the Cahn-Hilliard Navier-Stokes (CHNS) model of two-phase flows.  We used a continuous Galerkin spatial discretization with linear finite elements and a second-order time-marching scheme. The newly developed 
scheme is a variant of the energy-stable energy 
block-iterative scheme used in~\citet{ Khanwale2020}, but in this work we improved both on the formal order of accuracy of the scheme (now second-order accurate) and the efficiency of the Newton iteration. We rigorously proved that the new time-marching scheme maintains energy-stability.  We used a variational multiscale approach to stabilize the pressure.  The resulting method was implemented using a scalable adaptive meshing framework to perform large 3D simulations of complex multiphase flows. Solution-adapted meshing was
accomplished using octrees in the~\dendro package. We presented a comprehensive set of numerical experiments in both 2D and 3D, which were used to validate and test our numerical framework.  We also used these numerical experiments to validate our theoretical estimates for the energy-stability and mass conservation properties of the method.  We ran the resulting computational code for the CHNS model on a massively parallel architecture to simulate multi-wave single-mode Rayleigh-Taylor instability. We were able to push the framework to solve up to 2 billion degrees of freedom.  The same 3D framework was also used to simulate the mixing of two fluids in a cubic lid-driven cavity.  We performed a detailed scaling analysis of the 3D framework and show excellent strong and weak scaling up to  $\mathcal{O}(17K)$ MPI processes.  
   
For future work we plan to perform a detailed analysis of the performance of higher-order basis functions.  As was observed in the scaling study~(\cref{sec:scaling}), the preconditioner setup and Jacobian assembly are dominant parts of the overall solve time; we will aim to reduce the computational costs of these steps. We will also seek to develop projection-based methods which utilize all benefits of VMS approach, while decoupling the pressure row into a separate linear problem.  Very efficient multigrid solvers can then be used for both the blocks to make them more efficient and scalable.  These improved solvers will help with adapting the framework to a matrix-free approach to solve the linear problems which are very efficient at large scales as shown by~\citet{Ishii2019}.

\section{Acknowledgements}
The authors acknowledge XSEDE grant number TG-CTS110007 for computing time on TACC Stampede2.  The authors also acknowledge computing allocation through a DD award on TACC Frontera, ALCF computing resources and Iowa State University computing resources. JAR was supported in part by NSF Grants DMS--1620128 and DMS--2012699. BG, KS, MAK were funded in part by NSF grant 1935255, 1855902. HS was funded in part by NSF grant 1912930. VMC acknowledges the CSIRO Professorial Chair (VMC) in Computational Geoscience at Curtin University, Deep Earth Imaging Enterprise Future Science Platforms of the Commonwealth Scientific Industrial Research Organisation, CSIRO, of Australia, the European Union's Horizon 2020 research and innovation programme under the Marie Sklodowska-Curie grant agreement No 777778 (MATHROCKS), the Curtin Corrosion Centre and the Curtin Institute for Computation.

\bibliography{total_references,hari}

\begin{thebibliography}{77}
\expandafter\ifx\csname natexlab\endcsname\relax\def\natexlab#1{#1}\fi
\providecommand{\url}[1]{\texttt{#1}}
\providecommand{\href}[2]{#2}
\providecommand{\path}[1]{#1}
\providecommand{\DOIprefix}{doi:}
\providecommand{\ArXivprefix}{arXiv:}
\providecommand{\URLprefix}{URL: }
\providecommand{\Pubmedprefix}{pmid:}
\providecommand{\doi}[1]{\href{http://dx.doi.org/#1}{\path{#1}}}
\providecommand{\Pubmed}[1]{\href{pmid:#1}{\path{#1}}}
\providecommand{\bibinfo}[2]{#2}
\ifx\xfnm\relax \def\xfnm[#1]{\unskip,\space#1}\fi
\bibitem[{Anderson et~al.(1998)Anderson, McFadden, and Wheeler}]{Anderson1998}
\bibinfo{author}{D.~M. Anderson}, \bibinfo{author}{G.~B. McFadden},
  \bibinfo{author}{A.~A. Wheeler}, \bibinfo{journal}{Annual Review of Fluid
  Mechanics} \bibinfo{volume}{30} (\bibinfo{year}{1998})
  \bibinfo{pages}{139--165}. \DOIprefix\doi{10.1146/annurev.fluid.30.1.139}.
\bibitem[{Teh et~al.(2008)Teh, Lin, Hung, and Lee}]{Teh2008}
\bibinfo{author}{S.-y. Teh}, \bibinfo{author}{R.~Lin}, \bibinfo{author}{L.-h.
  Hung}, \bibinfo{author}{A.~P. Lee}, \bibinfo{journal}{Lab on a Chip}
  \bibinfo{volume}{8} (\bibinfo{year}{2008}) \bibinfo{pages}{198}.
  \DOIprefix\doi{10.1039/b715524g}.
\bibitem[{Dangla et~al.(2013)Dangla, Kayi, and Baroud}]{Dangla2013}
\bibinfo{author}{R.~Dangla}, \bibinfo{author}{S.~C. Kayi},
  \bibinfo{author}{C.~N. Baroud}, \bibinfo{journal}{Proceedings of the National
  Academy of Sciences} \bibinfo{volume}{110} (\bibinfo{year}{2013})
  \bibinfo{pages}{853--858}. \DOIprefix\doi{10.1073/pnas.1209186110}.
\bibitem[{Stoecklein et~al.(2014)Stoecklein, Wu, Owsley, Xie, {Di Carlo}, and
  Ganapathysubramanian}]{Stoecklein2014}
\bibinfo{author}{D.~Stoecklein}, \bibinfo{author}{C.-Y. Wu},
  \bibinfo{author}{K.~Owsley}, \bibinfo{author}{Y.~Xie},
  \bibinfo{author}{D.~{Di Carlo}}, \bibinfo{author}{B.~Ganapathysubramanian},
  \bibinfo{journal}{Lab on a Chip} \bibinfo{volume}{14} (\bibinfo{year}{2014})
  \bibinfo{pages}{4197--4204}. \DOIprefix\doi{10.1039/C4LC00653D}.
\bibitem[{Stoecklein et~al.(2016)Stoecklein, Wu, Kim, {Di Carlo}, and
  Ganapathysubramanian}]{Stoecklein2016a}
\bibinfo{author}{D.~Stoecklein}, \bibinfo{author}{C.~Y. Wu},
  \bibinfo{author}{D.~Kim}, \bibinfo{author}{D.~{Di Carlo}},
  \bibinfo{author}{B.~Ganapathysubramanian}, \bibinfo{journal}{Physics of
  Fluids} \bibinfo{volume}{28} (\bibinfo{year}{2016}).
  \DOIprefix\doi{10.1063/1.4939512}. \href{http://arxiv.org/abs/1506.01111}{\tt
  arXiv:1506.01111}.
\bibitem[{Khanwale et~al.(2020)Khanwale, Lofquist, Sundar, Rossmanith, and
  Ganapathysubramanian}]{Khanwale2020}
\bibinfo{author}{M.~A. Khanwale}, \bibinfo{author}{A.~D. Lofquist},
  \bibinfo{author}{H.~Sundar}, \bibinfo{author}{J.~A. Rossmanith},
  \bibinfo{author}{B.~Ganapathysubramanian}, \bibinfo{journal}{Journal of
  Computational Physics} \bibinfo{volume}{419} (\bibinfo{year}{2020})
  \bibinfo{pages}{109674}.
  \DOIprefix\doi{https://doi.org/10.1016/j.jcp.2020.109674}.
\bibitem[{Osher and Fedkiw(2003)}]{Osher2006}
\bibinfo{author}{S.~Osher}, \bibinfo{author}{R.~Fedkiw}, \bibinfo{title}{Level
  Set Methods and Dynamic Implicit Surfaces}, volume \bibinfo{volume}{153} of
  \textit{\bibinfo{series}{Applied Mathematical Sciences}},
  \bibinfo{publisher}{Springer}, \bibinfo{year}{2003}.
\bibitem[{Yan et~al.(2018)Yan, Lin, Bazilevs, and Wagner}]{Yan2018}
\bibinfo{author}{J.~Yan}, \bibinfo{author}{S.~Lin},
  \bibinfo{author}{Y.~Bazilevs}, \bibinfo{author}{G.~J. Wagner},
  \bibinfo{journal}{Computers \& Fluids}  (\bibinfo{year}{2018}).
\bibitem[{Jacqmin(1996)}]{Jacqmin1996}
\bibinfo{author}{D.~Jacqmin}, in: \bibinfo{booktitle}{34th AIAA Aerospace
  Sciences Meeting {\&} Exhibit}, volume \bibinfo{volume}{AIAA96},
  \bibinfo{organization}{American Institute of Aeronautics and Astronautics},
  \bibinfo{address}{Reno, Nevada, USA}, p. \bibinfo{pages}{0858}.
\bibitem[{Jacqmin(2000)}]{Jacqmin2000}
\bibinfo{author}{D.~Jacqmin}, \bibinfo{journal}{Journal of Fluid Mechanics}
  \bibinfo{volume}{402} (\bibinfo{year}{2000}) \bibinfo{pages}{57--88}.
  \DOIprefix\doi{10.1017/S0022112099006874}.
\bibitem[{Guo et~al.(2017)Guo, Lin, Lowengrub, and Wise}]{Guo2017}
\bibinfo{author}{Z.~Guo}, \bibinfo{author}{P.~Lin},
  \bibinfo{author}{J.~Lowengrub}, \bibinfo{author}{S.~M. Wise},
  \bibinfo{journal}{Computer Methods in Applied Mechanics and Engineering}
  \bibinfo{volume}{326} (\bibinfo{year}{2017}) \bibinfo{pages}{144--174}.
  \DOIprefix\doi{10.1016/j.cma.2017.08.011}.
\bibitem[{Shokrpour~Roudbari et~al.(2018)Shokrpour~Roudbari,
  {\c{S}}im{\c{s}}ek, van Brummelen, and van~der Zee}]{Shokrpour2018}
\bibinfo{author}{M.~Shokrpour~Roudbari},
  \bibinfo{author}{G.~{\c{S}}im{\c{s}}ek}, \bibinfo{author}{E.~H. van
  Brummelen}, \bibinfo{author}{K.~G. van~der Zee},
  \bibinfo{journal}{Mathematical Models and Methods in Applied Sciences}
  \bibinfo{volume}{28} (\bibinfo{year}{2018}) \bibinfo{pages}{733--770}.
  \DOIprefix\doi{10.1142/S0218202518500197}.
\bibitem[{Xu et~al.(2020)Xu, Gao, Lofquist, Fernando, Hsu, Sundar, and
  Ganapathysubramanian}]{Xu2020}
\bibinfo{author}{S.~Xu}, \bibinfo{author}{B.~Gao},
  \bibinfo{author}{A.~Lofquist}, \bibinfo{author}{M.~Fernando},
  \bibinfo{author}{M.-C. Hsu}, \bibinfo{author}{H.~Sundar},
  \bibinfo{author}{B.~Ganapathysubramanian}, \bibinfo{journal}{Submitted to
  Computers and Fluids}  (\bibinfo{year}{2020}).
\bibitem[{Saurabh et~al.(2020)Saurabh, Gao, Fernando, Xu, Khara, Khanwale, Hsu,
  Krishnamurthy, Sundar, and Ganapathysubramanian}]{Saurabh2020}
\bibinfo{author}{K.~Saurabh}, \bibinfo{author}{B.~Gao},
  \bibinfo{author}{M.~Fernando}, \bibinfo{author}{S.~Xu},
  \bibinfo{author}{B.~Khara}, \bibinfo{author}{M.~A. Khanwale},
  \bibinfo{author}{M.-C. Hsu}, \bibinfo{author}{A.~Krishnamurthy},
  \bibinfo{author}{H.~Sundar}, \bibinfo{author}{B.~Ganapathysubramanian},
  \bibinfo{journal}{arXiv preprint arXiv:2009.00706}  (\bibinfo{year}{2020}).
\bibitem[{Volker(2016)}]{Volker2016}
\bibinfo{author}{J.~Volker}, \bibinfo{title}{Finite Element Methods for
  Incompressible Flow Problems}, volume~\bibinfo{volume}{51} of
  \textit{\bibinfo{series}{Springer Series in Computational Mathematics Book}},
  \bibinfo{publisher}{Springer}, \bibinfo{year}{2016}.
\bibitem[{Burstedde et~al.(2011)Burstedde, Wilcox, and Ghattas}]{Burstedde2011}
\bibinfo{author}{C.~Burstedde}, \bibinfo{author}{L.~C. Wilcox},
  \bibinfo{author}{O.~Ghattas}, \bibinfo{journal}{SIAM Journal on Scientific
  Computing} \bibinfo{volume}{33} (\bibinfo{year}{2011})
  \bibinfo{pages}{1103--1133}.
\bibitem[{Sundar et~al.(2007)Sundar, Sampath, Adavani, Davatzikos, and
  Biros}]{Sundar2007}
\bibinfo{author}{H.~Sundar}, \bibinfo{author}{R.~S. Sampath},
  \bibinfo{author}{S.~S. Adavani}, \bibinfo{author}{C.~Davatzikos},
  \bibinfo{author}{G.~Biros}, in: \bibinfo{booktitle}{Proceedings of the 2007
  ACM/IEEE conference on Supercomputing}, \bibinfo{organization}{ACM},
  p.~\bibinfo{pages}{25}.
\bibitem[{Sundar et~al.(2008)Sundar, Sampath, and Biros}]{Dendro}
\bibinfo{author}{H.~Sundar}, \bibinfo{author}{R.~S. Sampath},
  \bibinfo{author}{G.~Biros}, \bibinfo{journal}{SIAM J. Sci. Comput.}
  \bibinfo{volume}{30} (\bibinfo{year}{2008}) \bibinfo{pages}{2675--2708}.
  \DOIprefix\doi{10.1137/070681727}.
\bibitem[{Fernando and Sundar(2020)}]{Dendro5}
\bibinfo{author}{M.~Fernando}, \bibinfo{author}{H.~Sundar},
  \bibinfo{title}{{paralab/Dendro-5.01: Parallel adaptive octree library for
  finite different, finite volume and finite element computations}},
  \bibinfo{year}{2020}. \URLprefix
  \url{https://doi.org/10.5281/zenodo.3879315}.
  \DOIprefix\doi{10.5281/zenodo.3879315}.
\bibitem[{Kim et~al.(2004)Kim, Kang, and Lowengrub}]{Kim2004a}
\bibinfo{author}{J.~Kim}, \bibinfo{author}{K.~Kang},
  \bibinfo{author}{J.~Lowengrub}, \bibinfo{journal}{Journal of Computational
  Physics} \bibinfo{volume}{193} (\bibinfo{year}{2004})
  \bibinfo{pages}{511--543}. \DOIprefix\doi{10.1016/j.jcp.2003.07.035}.
\bibitem[{Feng(2006)}]{Feng2006}
\bibinfo{author}{X.~Feng}, \bibinfo{journal}{SIAM Journal on Numerical
  Analysis} \bibinfo{volume}{44} (\bibinfo{year}{2006})
  \bibinfo{pages}{1049--1072}. \DOIprefix\doi{10.1137/050638333}.
\bibitem[{Shen and Yang(2010{\natexlab{a}})}]{Shen2010a}
\bibinfo{author}{J.~Shen}, \bibinfo{author}{X.~Yang}, \bibinfo{journal}{SIAM
  Journal on Scientific Computing} \bibinfo{volume}{32}
  (\bibinfo{year}{2010}{\natexlab{a}}) \bibinfo{pages}{1159--1179}.
  \DOIprefix\doi{10.1137/09075860X}.
\bibitem[{Shen and Yang(2010{\natexlab{b}})}]{Shen2010b}
\bibinfo{author}{J.~Shen}, \bibinfo{author}{X.~Yang}, \bibinfo{journal}{Chinese
  Annals of Mathematics, Series B} \bibinfo{volume}{31}
  (\bibinfo{year}{2010}{\natexlab{b}}) \bibinfo{pages}{743--758}.
  \DOIprefix\doi{10.1007/s11401-010-0599-y}.
\bibitem[{Dong(2012)}]{Dong2012}
\bibinfo{author}{S.~Dong}, \bibinfo{journal}{Computer Methods in Applied
  Mechanics and Engineering} \bibinfo{volume}{247-248} (\bibinfo{year}{2012})
  \bibinfo{pages}{179 -- 200}.
  \DOIprefix\doi{https://doi.org/10.1016/j.cma.2012.07.023}.
\bibitem[{Chen and Shen(2016)}]{Chen2016}
\bibinfo{author}{Y.~Chen}, \bibinfo{author}{J.~Shen}, \bibinfo{journal}{Journal
  of Computational Physics} \bibinfo{volume}{308} (\bibinfo{year}{2016})
  \bibinfo{pages}{40--56}. \DOIprefix\doi{10.1016/j.jcp.2015.12.006}.
\bibitem[{Dong(2017)}]{Dong2016}
\bibinfo{author}{S.~Dong}, \bibinfo{journal}{Journal of Computational Physics}
  \bibinfo{volume}{338} (\bibinfo{year}{2017}) \bibinfo{pages}{21--67}.
\bibitem[{Zhu et~al.(2019)Zhu, Chen, Yao, and Sun}]{Zhu2019}
\bibinfo{author}{G.~Zhu}, \bibinfo{author}{H.~Chen}, \bibinfo{author}{J.~Yao},
  \bibinfo{author}{S.~Sun}, \bibinfo{journal}{Applied Mathematical Modelling}
  \bibinfo{volume}{70} (\bibinfo{year}{2019}) \bibinfo{pages}{82--108}.
\bibitem[{Chen and Yang(2021)}]{Chen2021}
\bibinfo{author}{C.~Chen}, \bibinfo{author}{X.~Yang}, \bibinfo{journal}{ESAIM:
  Mathematical Modelling and Numerical Analysis} \bibinfo{volume}{55}
  (\bibinfo{year}{2021}) \bibinfo{pages}{2323--2347}.
  \DOIprefix\doi{10.1051/m2an/2021056}.
\bibitem[{Han and Wang(2015)}]{Han2015}
\bibinfo{author}{D.~Han}, \bibinfo{author}{X.~Wang}, \bibinfo{journal}{Journal
  of Computational Physics} \bibinfo{volume}{290} (\bibinfo{year}{2015})
  \bibinfo{pages}{139--156}. \DOIprefix\doi{10.1016/j.jcp.2015.02.046}.
  \href{http://arxiv.org/abs/1407.7048}{\tt arXiv:1407.7048}.
\bibitem[{Fu and Han(2021)}]{Fu2021}
\bibinfo{author}{G.~Fu}, \bibinfo{author}{D.~Han}, \bibinfo{journal}{Computer
  Methods in Applied Mechanics and Engineering} \bibinfo{volume}{387}
  (\bibinfo{year}{2021}) \bibinfo{pages}{114186}.
  \DOIprefix\doi{10.1016/j.cma.2021.114186}.
\bibitem[{Hughes et~al.(2018)Hughes, Scovazzi, and
  Franca}]{hughes2018multiscale}
\bibinfo{author}{T.~J. Hughes}, \bibinfo{author}{G.~Scovazzi},
  \bibinfo{author}{L.~P. Franca}, \bibinfo{journal}{Encyclopedia of
  Computational Mechanics Second Edition}  (\bibinfo{year}{2018})
  \bibinfo{pages}{1--64}.
\bibitem[{Hughes(1995)}]{Hughes1995}
\bibinfo{author}{T.~J. Hughes}, \bibinfo{journal}{Computer Methods in Applied
  Mechanics and Engineering} \bibinfo{volume}{127} (\bibinfo{year}{1995})
  \bibinfo{pages}{387--401}.
\bibitem[{Bazilevs et~al.(2007)Bazilevs, Calo, Cottrell, Hughes, Reali, and
  Scovazzi}]{Bazilevs2007}
\bibinfo{author}{Y.~Bazilevs}, \bibinfo{author}{V.~Calo},
  \bibinfo{author}{J.~Cottrell}, \bibinfo{author}{T.~Hughes},
  \bibinfo{author}{A.~Reali}, \bibinfo{author}{G.~Scovazzi},
  \bibinfo{journal}{Computer Methods in Applied Mechanics and Engineering}
  \bibinfo{volume}{197} (\bibinfo{year}{2007}) \bibinfo{pages}{173--201}.
\bibitem[{Mutlu et~al.(2018)Mutlu, Edd, and Toner}]{Mutlu2018}
\bibinfo{author}{B.~R. Mutlu}, \bibinfo{author}{J.~F. Edd},
  \bibinfo{author}{M.~Toner}, \bibinfo{journal}{Proceedings of the National
  Academy of Sciences} \bibinfo{volume}{115} (\bibinfo{year}{2018})
  \bibinfo{pages}{7682--7687}.
\bibitem[{Stoecklein and Di~Carlo(2018)}]{Stoecklein2018}
\bibinfo{author}{D.~Stoecklein}, \bibinfo{author}{D.~Di~Carlo},
  \bibinfo{journal}{Analytical chemistry} \bibinfo{volume}{91}
  (\bibinfo{year}{2018}) \bibinfo{pages}{296--314}.
\bibitem[{Shen and Yang(2015)}]{Shen2015}
\bibinfo{author}{J.~Shen}, \bibinfo{author}{X.~Yang}, \bibinfo{journal}{SIAM
  Journal on Numerical Analysis} \bibinfo{volume}{53} (\bibinfo{year}{2015})
  \bibinfo{pages}{279--296}. \DOIprefix\doi{10.1137/140971154}.
\bibitem[{Hughes et~al.(2000)Hughes, Mazzei, and Jansen}]{Hughes2000}
\bibinfo{author}{T.~J. Hughes}, \bibinfo{author}{L.~Mazzei},
  \bibinfo{author}{K.~E. Jansen}, \bibinfo{journal}{Computing and Visualization
  in Science} \bibinfo{volume}{3} (\bibinfo{year}{2000})
  \bibinfo{pages}{47--59}.
\bibitem[{Ahmed et~al.(2017)Ahmed, Chac{\'o}n~Rebollo, John, and
  Rubino}]{Ahmed2017}
\bibinfo{author}{N.~Ahmed}, \bibinfo{author}{T.~Chac{\'o}n~Rebollo},
  \bibinfo{author}{V.~John}, \bibinfo{author}{S.~Rubino},
  \bibinfo{journal}{Archives of Computational Methods in Engineering}
  \bibinfo{volume}{24} (\bibinfo{year}{2017}) \bibinfo{pages}{115--164}.
  \DOIprefix\doi{10.1007/s11831-015-9161-0}.
\bibitem[{Coupez and Hachem(2013)}]{Coupez2013}
\bibinfo{author}{T.~Coupez}, \bibinfo{author}{E.~Hachem},
  \bibinfo{journal}{Computer Methods in Applied Mechanics and Engineering}
  \bibinfo{volume}{267} (\bibinfo{year}{2013}) \bibinfo{pages}{65--85}.
  \DOIprefix\doi{10.1016/j.cma.2013.08.004}.
\bibitem[{Hachem et~al.(2013)Hachem, Feghali, Codina, and Coupez}]{Hachem2013}
\bibinfo{author}{E.~Hachem}, \bibinfo{author}{S.~Feghali},
  \bibinfo{author}{R.~Codina}, \bibinfo{author}{T.~Coupez},
  \bibinfo{journal}{Computers and Structures} \bibinfo{volume}{122}
  (\bibinfo{year}{2013}) \bibinfo{pages}{88--100}.
  \DOIprefix\doi{10.1016/j.compstruc.2012.12.004}.
\bibitem[{Hachem et~al.(2016)Hachem, Khalloufi, Bruchon, Valette, and
  Mesri}]{Hachem2016}
\bibinfo{author}{E.~Hachem}, \bibinfo{author}{M.~Khalloufi},
  \bibinfo{author}{J.~Bruchon}, \bibinfo{author}{R.~Valette},
  \bibinfo{author}{Y.~Mesri}, \bibinfo{journal}{Computer Methods in Applied
  Mechanics and Engineering} \bibinfo{volume}{308} (\bibinfo{year}{2016})
  \bibinfo{pages}{238--255}. \DOIprefix\doi{10.1016/j.cma.2016.05.022}.
\bibitem[{Ishii et~al.(2019)Ishii, Fernando, Saurabh, Khara,
  Ganapathysubramanian, and Sundar}]{MasadoDKT}
\bibinfo{author}{M.~Ishii}, \bibinfo{author}{M.~Fernando},
  \bibinfo{author}{K.~Saurabh}, \bibinfo{author}{B.~Khara},
  \bibinfo{author}{B.~Ganapathysubramanian}, \bibinfo{author}{H.~Sundar}, in:
  \bibinfo{booktitle}{SC'19: Proceedings of the International Conference for
  High Performance Computing, Networking, Storage and Analysis}, SC '19,
  \bibinfo{publisher}{Association for Computing Machinery},
  \bibinfo{address}{New York, NY, USA}, \bibinfo{year}{2019}, pp.
  \bibinfo{pages}{1--61}. \URLprefix
  \url{https://doi.org/10.1145/3295500.3356198}.
  \DOIprefix\doi{10.1145/3295500.3356198}.
\bibitem[{Bao et~al.(1998)Bao, Bielak, Ghattas, Kallivokas, O'Hallaron,
  Shewchuk, and Xu}]{mantle}
\bibinfo{author}{H.~Bao}, \bibinfo{author}{J.~Bielak},
  \bibinfo{author}{O.~Ghattas}, \bibinfo{author}{L.~F. Kallivokas},
  \bibinfo{author}{D.~R. O'Hallaron}, \bibinfo{author}{J.~R. Shewchuk},
  \bibinfo{author}{J.~Xu}, \bibinfo{journal}{Computer methods in applied
  mechanics and engineering} \bibinfo{volume}{152} (\bibinfo{year}{1998})
  \bibinfo{pages}{85--102}.
\bibitem[{Fernando et~al.(2019)Fernando, Neilsen, Lim, Hirschmann, and
  Sundar}]{Fernando2018_GR}
\bibinfo{author}{M.~Fernando}, \bibinfo{author}{D.~Neilsen},
  \bibinfo{author}{H.~Lim}, \bibinfo{author}{E.~Hirschmann},
  \bibinfo{author}{H.~Sundar}, \bibinfo{journal}{SIAM Journal on Scientific
  Computing} \bibinfo{volume}{41} (\bibinfo{year}{2019})
  \bibinfo{pages}{C97--C138}. \URLprefix
  \url{https://doi.org/10.1137/18M1196972}. \DOIprefix\doi{10.1137/18M1196972}.
  \href{http://arxiv.org/abs/https://doi.org/10.1137/18M1196972}{\tt
  arXiv:https://doi.org/10.1137/18M1196972}.
\bibitem[{Sundar et~al.(2008)Sundar, Sampath, and Biros}]{Sundar2008}
\bibinfo{author}{H.~Sundar}, \bibinfo{author}{R.~S. Sampath},
  \bibinfo{author}{G.~Biros}, \bibinfo{journal}{SIAM Journal on Scientific
  Computing} \bibinfo{volume}{30} (\bibinfo{year}{2008})
  \bibinfo{pages}{2675--2708}.
\bibitem[{Burstedde et~al.(2011)Burstedde, Wilcox, and
  Ghattas}]{BursteddeWilcoxGhattas11}
\bibinfo{author}{C.~Burstedde}, \bibinfo{author}{L.~C. Wilcox},
  \bibinfo{author}{O.~Ghattas}, \bibinfo{journal}{SIAM Journal on Scientific
  Computing} \bibinfo{volume}{33} (\bibinfo{year}{2011})
  \bibinfo{pages}{1103--1133}. \DOIprefix\doi{10.1137/100791634}.
\bibitem[{Fernando et~al.(2017)Fernando, Duplyakin, and Sundar}]{Fernando:2017}
\bibinfo{author}{M.~Fernando}, \bibinfo{author}{D.~Duplyakin},
  \bibinfo{author}{H.~Sundar}, in: \bibinfo{booktitle}{Proceedings of the 26th
  International Symposium on High-Performance Parallel and Distributed
  Computing}, HPDC '17, \bibinfo{publisher}{ACM}, \bibinfo{address}{New York,
  NY, USA}, \bibinfo{year}{2017}, pp. \bibinfo{pages}{231--242}. \URLprefix
  \url{http://doi.acm.org/10.1145/3078597.3078610}.
  \DOIprefix\doi{10.1145/3078597.3078610}.
\bibitem[{Magaletti et~al.(2013)Magaletti, Picano, Chinappi, Marino, and
  Casciola}]{Magaletti2013}
\bibinfo{author}{F.~Magaletti}, \bibinfo{author}{F.~Picano},
  \bibinfo{author}{M.~Chinappi}, \bibinfo{author}{L.~Marino},
  \bibinfo{author}{C.~M. Casciola}, \bibinfo{journal}{Journal of Fluid
  Mechanics} \bibinfo{volume}{714} (\bibinfo{year}{2013})
  \bibinfo{pages}{95--126}. \DOIprefix\doi{10.1017/jfm.2012.461}.
\bibitem[{Conti and Giorgini(2018)}]{Conti2018}
\bibinfo{author}{M.~Conti}, \bibinfo{author}{A.~Giorgini}, \bibinfo{title}{{The
  three-dimensional Cahn-Hilliard-Brinkman system with unmatched viscosities}},
  \bibinfo{year}{2018}.
  \bibinfo{note}{\url{https://hal.archives-ouvertes.fr/hal-01559179}}.
\bibitem[{Giorgini et~al.(2019)Giorgini, Miranville, and Temam}]{Giorgini2019}
\bibinfo{author}{A.~Giorgini}, \bibinfo{author}{A.~Miranville},
  \bibinfo{author}{R.~Temam}, \bibinfo{journal}{SIAM Journal on Mathematical
  Analysis} \bibinfo{volume}{51} (\bibinfo{year}{2019})
  \bibinfo{pages}{2535--2574}.
\bibitem[{Zeidler(1985)}]{Zeidler1985IIb}
\bibinfo{author}{E.~Zeidler}, \bibinfo{journal}{Part II/B (Nonlinear Monotone
  Operators) Springer, Berlin}  (\bibinfo{year}{1985}).
\bibitem[{Tezduyar et~al.(1992)Tezduyar, Mittal, Ray, and
  Shih}]{article:TezMitRayShi92}
\bibinfo{author}{T.~Tezduyar}, \bibinfo{author}{S.~Mittal},
  \bibinfo{author}{S.~Ray}, \bibinfo{author}{R.~Shih},
  \bibinfo{journal}{Computer Methods in Applied Mechanics and Engineering}
  \bibinfo{volume}{95} (\bibinfo{year}{1992}) \bibinfo{pages}{221--242}.
\bibitem[{Brooks and Hughes(1982)}]{article:BrooksHughes1982}
\bibinfo{author}{A.~Brooks}, \bibinfo{author}{T.~Hughes},
  \bibinfo{journal}{Computer Methods in Applied Mechanics and Engineering}
  \bibinfo{volume}{32} (\bibinfo{year}{1982}) \bibinfo{pages}{199--259}.
\bibitem[{Hughes and Sangalli(2007)}]{Hughes2007}
\bibinfo{author}{T.~J. Hughes}, \bibinfo{author}{G.~Sangalli},
  \bibinfo{journal}{SIAM Journal on Numerical Analysis} \bibinfo{volume}{45}
  (\bibinfo{year}{2007}) \bibinfo{pages}{539--557}.
\bibitem[{Balay et~al.(1997)Balay, Gropp, McInnes, and Smith}]{petsc-efficient}
\bibinfo{author}{S.~Balay}, \bibinfo{author}{W.~D. Gropp},
  \bibinfo{author}{L.~C. McInnes}, \bibinfo{author}{B.~F. Smith}, in:
  \bibinfo{editor}{E.~Arge}, \bibinfo{editor}{A.~M. Bruaset},
  \bibinfo{editor}{H.~P. Langtangen} (Eds.), \bibinfo{booktitle}{Modern
  Software Tools in Scientific Computing}, \bibinfo{publisher}{Birkh{\"{a}}user
  Press}, \bibinfo{year}{1997}, pp. \bibinfo{pages}{163--202}.
\bibitem[{Balay et~al.(2019{\natexlab{a}})Balay, Abhyankar, Adams, Brown,
  Brune, Buschelman, Dalcin, Dener, Eijkhout, Gropp, Karpeyev, Kaushik,
  Knepley, May, McInnes, Mills, Munson, Rupp, Sanan, Smith, Zampini, Zhang, and
  Zhang}]{petsc-web-page}
\bibinfo{author}{S.~Balay}, \bibinfo{author}{S.~Abhyankar},
  \bibinfo{author}{M.~F. Adams}, \bibinfo{author}{J.~Brown},
  \bibinfo{author}{P.~Brune}, \bibinfo{author}{K.~Buschelman},
  \bibinfo{author}{L.~Dalcin}, \bibinfo{author}{A.~Dener},
  \bibinfo{author}{V.~Eijkhout}, \bibinfo{author}{W.~D. Gropp},
  \bibinfo{author}{D.~Karpeyev}, \bibinfo{author}{D.~Kaushik},
  \bibinfo{author}{M.~G. Knepley}, \bibinfo{author}{D.~A. May},
  \bibinfo{author}{L.~C. McInnes}, \bibinfo{author}{R.~T. Mills},
  \bibinfo{author}{T.~Munson}, \bibinfo{author}{K.~Rupp},
  \bibinfo{author}{P.~Sanan}, \bibinfo{author}{B.~F. Smith},
  \bibinfo{author}{S.~Zampini}, \bibinfo{author}{H.~Zhang},
  \bibinfo{author}{H.~Zhang}, \bibinfo{title}{{PETS}c {W}eb page},
  \bibinfo{howpublished}{\url{https://www.mcs.anl.gov/petsc}},
  \bibinfo{year}{2019}{\natexlab{a}}.
\bibitem[{Balay et~al.(2019{\natexlab{b}})Balay, Abhyankar, Adams, Brown,
  Brune, Buschelman, Dalcin, Dener, Eijkhout, Gropp, Karpeyev, Kaushik,
  Knepley, May, McInnes, Mills, Munson, Rupp, Sanan, Smith, Zampini, Zhang, and
  Zhang}]{petsc-user-ref}
\bibinfo{author}{S.~Balay}, \bibinfo{author}{S.~Abhyankar},
  \bibinfo{author}{M.~F. Adams}, \bibinfo{author}{J.~Brown},
  \bibinfo{author}{P.~Brune}, \bibinfo{author}{K.~Buschelman},
  \bibinfo{author}{L.~Dalcin}, \bibinfo{author}{A.~Dener},
  \bibinfo{author}{V.~Eijkhout}, \bibinfo{author}{W.~D. Gropp},
  \bibinfo{author}{D.~Karpeyev}, \bibinfo{author}{D.~Kaushik},
  \bibinfo{author}{M.~G. Knepley}, \bibinfo{author}{D.~A. May},
  \bibinfo{author}{L.~C. McInnes}, \bibinfo{author}{R.~T. Mills},
  \bibinfo{author}{T.~Munson}, \bibinfo{author}{K.~Rupp},
  \bibinfo{author}{P.~Sanan}, \bibinfo{author}{B.~F. Smith},
  \bibinfo{author}{S.~Zampini}, \bibinfo{author}{H.~Zhang},
  \bibinfo{author}{H.~Zhang}, \bibinfo{title}{{PETS}c Users Manual},
  \bibinfo{type}{Technical Report} \bibinfo{number}{ANL-95/11 - Revision 3.11},
  Argonne National Laboratory, \bibinfo{year}{2019}{\natexlab{b}}.
\bibitem[{Sundar et~al.(2008)Sundar, Sampath, and Biros}]{SundarSampathBiros08}
\bibinfo{author}{H.~Sundar}, \bibinfo{author}{R.~Sampath},
  \bibinfo{author}{G.~Biros}, \bibinfo{journal}{SIAM Journal on Scientific
  Computing} \bibinfo{volume}{30} (\bibinfo{year}{2008})
  \bibinfo{pages}{2675--2708}. \DOIprefix\doi{10.1137/070681727}.
\bibitem[{Bern et~al.(1999)Bern, Eppstein, and Teng}]{bern1999parallel}
\bibinfo{author}{M.~Bern}, \bibinfo{author}{D.~Eppstein},
  \bibinfo{author}{S.-H. Teng}, \bibinfo{journal}{International Journal of
  Computational Geometry \& Applications} \bibinfo{volume}{9}
  (\bibinfo{year}{1999}) \bibinfo{pages}{517--532}.
\bibitem[{Sundar et~al.(2007)Sundar, Sampath, Adavani, Davatzikos, and
  Biros}]{SundarSampathAdavaniEtAl07}
\bibinfo{author}{H.~Sundar}, \bibinfo{author}{R.~S. Sampath},
  \bibinfo{author}{S.~S. Adavani}, \bibinfo{author}{C.~Davatzikos},
  \bibinfo{author}{G.~Biros}, in: \bibinfo{booktitle}{SC'07: Proceedings of the
  International Conference for High Performance Computing, Networking, Storage,
  and Analysis}, \bibinfo{publisher}{ACM/IEEE}, \bibinfo{year}{2007}, pp.
  \bibinfo{pages}{1--12}.
\bibitem[{Hysing et~al.(2009)Hysing, Turek, Kuzmin, Parolini, Burman, Ganesan,
  and Tobiska}]{Hysing2009}
\bibinfo{author}{S.-R. Hysing}, \bibinfo{author}{S.~Turek},
  \bibinfo{author}{D.~Kuzmin}, \bibinfo{author}{N.~Parolini},
  \bibinfo{author}{E.~Burman}, \bibinfo{author}{S.~Ganesan},
  \bibinfo{author}{L.~Tobiska}, \bibinfo{journal}{International Journal for
  Numerical Methods in Fluids} \bibinfo{volume}{60} (\bibinfo{year}{2009})
  \bibinfo{pages}{1259--1288}.
\bibitem[{Aland and Voigt(2012)}]{Aland2012}
\bibinfo{author}{S.~Aland}, \bibinfo{author}{A.~Voigt},
  \bibinfo{journal}{International Journal for Numerical Methods in Fluids}
  \bibinfo{volume}{69} (\bibinfo{year}{2012}) \bibinfo{pages}{747--761}.
  \URLprefix \url{http://doi.wiley.com/10.1002/fld.2611}.
  \DOIprefix\doi{10.1002/fld.2611}.
\bibitem[{Yuan et~al.(2017)Yuan, Chen, Shu, Wang, Niu, and Shu}]{Yuan2017}
\bibinfo{author}{H.~Z. Yuan}, \bibinfo{author}{Z.~Chen},
  \bibinfo{author}{C.~Shu}, \bibinfo{author}{Y.~Wang}, \bibinfo{author}{X.~D.
  Niu}, \bibinfo{author}{S.~Shu}, \bibinfo{journal}{Journal of Computational
  Physics} \bibinfo{volume}{345} (\bibinfo{year}{2017})
  \bibinfo{pages}{404--426}. \URLprefix
  \url{http://dx.doi.org/10.1016/j.jcp.2017.05.020}.
  \DOIprefix\doi{10.1016/j.jcp.2017.05.020}.
\bibitem[{Xie et~al.(2015)Xie, Wodo, and Ganapathysubramanian}]{Xie2015}
\bibinfo{author}{Y.~Xie}, \bibinfo{author}{O.~Wodo},
  \bibinfo{author}{B.~Ganapathysubramanian}, \bibinfo{journal}{Computers and
  Fluids} \bibinfo{volume}{141} (\bibinfo{year}{2015})
  \bibinfo{pages}{223--234}. \DOIprefix\doi{10.1016/j.compfluid.2016.04.011}.
\bibitem[{Tryggvason and Unverdi(1990)}]{Tryggvason1990}
\bibinfo{author}{G.~Tryggvason}, \bibinfo{author}{S.~O. Unverdi},
  \bibinfo{journal}{Physics of Fluids A: Fluid Dynamics} \bibinfo{volume}{2}
  (\bibinfo{year}{1990}) \bibinfo{pages}{656--659}.
\bibitem[{Li et~al.(1996)Li, Jin, and Glimm}]{Li1996}
\bibinfo{author}{X.~Li}, \bibinfo{author}{B.~Jin}, \bibinfo{author}{J.~Glimm},
  \bibinfo{journal}{Journal of Computational Physics} \bibinfo{volume}{126}
  (\bibinfo{year}{1996}) \bibinfo{pages}{343--355}.
\bibitem[{Guermond and Quartapelle(2000)}]{Guermond2000}
\bibinfo{author}{J.-L. Guermond}, \bibinfo{author}{L.~Quartapelle},
  \bibinfo{journal}{Journal of Computational Physics} \bibinfo{volume}{165}
  (\bibinfo{year}{2000}) \bibinfo{pages}{167--188}.
\bibitem[{Tryggvason(1988)}]{Tryggvason1988}
\bibinfo{author}{G.~Tryggvason}, \bibinfo{journal}{Journal of Computational
  Physics} \bibinfo{volume}{75} (\bibinfo{year}{1988})
  \bibinfo{pages}{253--282}.
\bibitem[{Ding et~al.(2007)Ding, Spelt, and Shu}]{Ding2007}
\bibinfo{author}{H.~Ding}, \bibinfo{author}{P.~D. Spelt},
  \bibinfo{author}{C.~Shu}, \bibinfo{journal}{Journal of Computational Physics}
  \bibinfo{volume}{226} (\bibinfo{year}{2007}) \bibinfo{pages}{2078--2095}.
\bibitem[{Waddell et~al.(2001)Waddell, Niederhaus, and Jacobs}]{Waddell2001}
\bibinfo{author}{J.~T. Waddell}, \bibinfo{author}{C.~E. Niederhaus},
  \bibinfo{author}{J.~W. Jacobs}, \bibinfo{journal}{Physics of Fluids}
  \bibinfo{volume}{13} (\bibinfo{year}{2001}) \bibinfo{pages}{1263--1273}.
  \DOIprefix\doi{10.1063/1.1359762}.
\bibitem[{Hunt et~al.(1988)Hunt, Wray, and Moin}]{Hunt1988}
\bibinfo{author}{J.~Hunt}, \bibinfo{author}{A.~Wray},
  \bibinfo{author}{P.~Moin}, \bibinfo{journal}{Center for turbulence research
  report CTR-S88}  (\bibinfo{year}{1988}) \bibinfo{pages}{193--208}.
\bibitem[{Liang et~al.(2016)Liang, Li, Shi, and Chai}]{Liang2016}
\bibinfo{author}{H.~Liang}, \bibinfo{author}{Q.~X. Li}, \bibinfo{author}{B.~C.
  Shi}, \bibinfo{author}{Z.~H. Chai}, \bibinfo{journal}{Physical Review E}
  \bibinfo{volume}{93} (\bibinfo{year}{2016}) \bibinfo{pages}{1--11}.
  \DOIprefix\doi{10.1103/PhysRevE.93.033113}.
\bibitem[{Jain et~al.(2020)Jain, Mani, and Moin}]{Jain2020}
\bibinfo{author}{S.~S. Jain}, \bibinfo{author}{A.~Mani},
  \bibinfo{author}{P.~Moin}, \bibinfo{journal}{Journal of Computational
  Physics} \bibinfo{volume}{418} (\bibinfo{year}{2020})
  \bibinfo{pages}{109606}. \URLprefix
  \url{https://doi.org/10.1016/j.jcp.2020.109606
  https://linkinghub.elsevier.com/retrieve/pii/S0021999120303806}.
  \DOIprefix\doi{10.1016/j.jcp.2020.109606}.
  \href{http://arxiv.org/abs/1911.03619}{\tt arXiv:1911.03619}.
\bibitem[{Chakravarthy and Ottino(1996)}]{Chakravarthy1996}
\bibinfo{author}{V.~Chakravarthy}, \bibinfo{author}{J.~Ottino},
  \bibinfo{journal}{Chemical Engineering Science} \bibinfo{volume}{51}
  (\bibinfo{year}{1996}) \bibinfo{pages}{3613--3622}. \URLprefix
  \url{https://linkinghub.elsevier.com/retrieve/pii/0009250996000073}.
  \DOIprefix\doi{10.1016/0009-2509(96)00007-3}.
\bibitem[{Chella and Vi{\~{n}}als(1996)}]{Chella1996}
\bibinfo{author}{R.~Chella}, \bibinfo{author}{J.~Vi{\~{n}}als},
  \bibinfo{journal}{Physical Review E} \bibinfo{volume}{53}
  (\bibinfo{year}{1996}) \bibinfo{pages}{3832--3840}. \URLprefix
  \url{https://link.aps.org/doi/10.1103/PhysRevE.53.3832}.
  \DOIprefix\doi{10.1103/PhysRevE.53.3832}.
\bibitem[{Park et~al.(2016)Park, Dorao, and Fernandino}]{Park2016}
\bibinfo{author}{K.~Park}, \bibinfo{author}{C.~A. Dorao},
  \bibinfo{author}{M.~Fernandino}, in: \bibinfo{booktitle}{Volume 1B, Symposia:
  Fluid Mechanics (Fundamental Issues and Perspectives; Industrial and
  Environmental Applications); Multiphase Flow and Systems (Multiscale Methods;
  Noninvasive Measurements; Numerical Methods; Heat Transfer; Performance);
  Transport Phe}, \bibinfo{publisher}{American Society of Mechanical
  Engineers}, \bibinfo{year}{2016}, pp. \bibinfo{pages}{1--9}. \URLprefix
  \url{https://asmedigitalcollection.asme.org/FEDSM/proceedings/FEDSM2016/50299/Washington,
  DC, USA/233438}. \DOIprefix\doi{10.1115/FEDSM2016-1008}.
\bibitem[{Ishii et~al.(2019)Ishii, Fernando, Saurabh, Khara,
  Ganapathysubramanian, and Sundar}]{Ishii2019}
\bibinfo{author}{M.~Ishii}, \bibinfo{author}{M.~Fernando},
  \bibinfo{author}{K.~Saurabh}, \bibinfo{author}{B.~Khara},
  \bibinfo{author}{B.~Ganapathysubramanian}, \bibinfo{author}{H.~Sundar}, in:
  \bibinfo{booktitle}{Proceedings of the International Conference for High
  Performance Computing, Networking, Storage and Analysis}, pp.
  \bibinfo{pages}{1--61}.

\end{thebibliography}

\newpage

\appendix

\section{Some elementary propositions}
For completeness we recall an important proposition and one its corollaries from \citep{Khanwale2020}.
\begin{proposition}{}{} 
	The following identity holds:
		\begin{equation}
		\pd{\tilde{\phi}^{k}}{x_j}\left(\pd{}{x_j}\left({\pd{\tilde{\phi}^{k}}{x_i}}\right)\right) = \frac{1}{2}\pd{}{x_j}\left(  \pd{\tilde{\phi}^{k}}{x_i} \pd{\tilde{\phi}^{k}}{x_i} \right)
		\end{equation}
			$\forall \;\; \widetilde{\phi}^k$, $\in H^1(\Omega)$, where $\phi^k, \phi^{k+1}, \mu^{k},\mu^{k+1}, \vec{v}^k, \vec{v}^{k+1}$ solves \cref{eqn:nav_stokes_var_semi_disc} -- \cref{eqn:phi_eqn_var_semi_disc}.
\label{prop:forcing_aside}
\end{proposition}


\begin{remark}
	The advection term in \cref{eqn:nav_stokes} can be defined in the divergence form as (see lemma 6.10 of section 6.1.2 of \citep{Volker2016} for details):
	\begin{align}
	B_1(v_i, v_j) = \rho v_j \pd{v_i}{x_j} + \frac{1}{2}  v_i \pd{\left(\rho v_j\right)}{x_j}
	\quad \text{and} \quad 
	B_2(v_i, v_j) = J_j\pd{v_i}{x_j} + \frac{1}{2}  v_i \pd{J_j}{x_j},
	\end{align}
	The divergence form induces a trilinear form when weakened:
	\begin{align}
	b_1(v_i, v_j, w_i) &= \Bigl(B_1(v_i, v_j),w_i\Bigr) = \left(\rho v_j \pd{v_i}{x_j}, w_i\right) + \frac{1}{2} \left( v_i \pd{\left(\rho v_j\right)}{x_j}, w_i\right) , \\
	b_2(v_i, J_j, w_i) &= \Bigl(B_1(v_i, J_j),w_i\Bigr) = \left(J_j \pd{v_i}{x_j}, w_i\right) + \frac{1}{2}\left(v_i \pd{J_j}{x_j}, w_i\right).
	\end{align}
	Using the above proposition and using integration by parts then yields: 
	\begin{align}
	b_1(v_i, v_j, v_i) &= (B_1(v_i, v_j),v_i) = 0,\label{eqn:trilinear_zero_vel}\\
	b_2(v_i, J_j, v_i) &= (B_2(v_i, J_j),v_i) = 0\label{eqn:trilinear_zero_massflux}.
	\end{align}
	See (see lemma 6.10 of section 6.1.2 of \citep{Volker2016} for proof.  
	For the equal density case $J_j$ is zero, and $\pd{\left(\rho v_j\right)}{x_j} = \rho \pd{v_j}{x_j} = 0$
	then, 
	\begin{align}
	b_1(v_i, v_j, v_i) &= \Bigl(B_1(v_i, v_j),v_i\Bigr) = \left(\rho v_j \pd{v_i}{x_j}, v_i\right) = \rho \left(v_j \pd{v_i}{x_j}, v_i\right) + \rho \frac{1}{2} \left( v_i \pd{v_j}{x_j}, v_i\right) = 0, \\
	b_2(v_i, J_j, v_i) &= \Bigl(B_1(v_i, J_j),v_i\Bigr) = \left(J_j \pd{v_i}{x_j}, v_i\right) = \left(J_j \pd{v_i}{x_j}, v_i\right) + \frac{1}{2}\left(v_i \pd{J_j}{x_j}, v_i\right) = 0,
	\end{align} 
	by skew-symmetry which is used in~\citep{Khanwale2020}.
	\label{rem:skew-symmetry}
\end{remark}

\begin{corollary}{Strong equivalence of forcing} {}
	If we have the following equivalence in the weak sense:
	\begin{equation}
	\frac{Cn}{We}\left( \pd{}{x_j}\left({\pd{\tphi^k}{x_i}\pd{\tphi^{k}}{x_j}}\right), \delta t  \, \tvi^k\right) = \frac{\delta t}{WeCn}\left(\tphi^k \pd{\tmu^{k}}{x_i},\tvi^k\right),
	\end{equation}
	$\forall \;\; \tphi^k$, $\tmu^{k} \in  H^1(\Omega)$, and $\forall \;\; \widetilde{\vec{v}}^k \in  \vec{H}_{0}^1(\Omega)$, and $\widetilde{\vec{v}}^k$ is weakly divergence-free, where $\widetilde{\vec{v}}^k, \widetilde{\vec{v}}^{k+1}, p^k, p^{k+1}, \phi^k, \phi^{k+1}, \mu^{k},\mu^{k+1}$ satisfy \cref{eqn:nav_stokes_var_semi_disc} -- \cref{eqn:phi_eqn_var_semi_disc}, then 
	the following equivalence also holds in the strong sense:
	\begin{equation}
	\frac{Cn}{We}\pd{}{x_j}\left({\pd{\tphi^k}{x_i}\pd{\tphi^{k}}{x_j}}\right) = \frac{1}{WeCn}\tphi^k\pd{\tmu^{k}}{x_i},
	\end{equation} 
	if $\tphi^k$, $\tmu^{k} \in H^1(\Omega) \bigcap C^{\infty}_{c}(\Omega)$, and $\widetilde{\vec{v}}^k \in \vec{H}_{0}^1(\Omega)\bigcap \vec{C}^{\infty}_{c}(\Omega)$.
\end{corollary}

\section{Details of solver selection for the numerical experiments}
\label{sec:app_linear_solve}
For the cases presented in \cref{subsec:single_rising_drop_2D,subsec:rayleigh_taylor,subsec:ldc_3d} we use the BiCGStab linear solver (a Krylov space solver) with additive Schwarz-based preconditioning.  For better reproduction, the command line options we provide {\sc petsc} are given below which include some commands used for printing some norms as well.
\begin{lstlisting}
-ksp_type bcgs
-pc_type asm
-sub_pc_type lu 0
#For monitoring residuals
-snes_monitor 
-snes_converged_reason 
-ksp_converged_reason
\end{lstlisting}

For the Rayleigh-Taylor instability case (\cref{subsec:rayleigh_taylor_2D}), we used an algebraic multigrid (AMG) linear solver with an successive over relaxation preconditioner with a GMRES at each level as a smoother.  The options used for the \petsc setup is given below. 
\begin{lstlisting}
solver_options_ns = {
	snes_atol = 1e-5
	snes_rtol = 1e-6
	snes_stol = 1e-5
	snes_max_it = 40
	ksp_rtol = 1e-5
	ksp_atol = 1e-6
	ksp_diagonal_scale = True
	ksp_diagonal_scale_fix = True

	#multigrid

		#solver selection
		ksp_type = "fgmres"
		pc_type = "gamg"
		pc_gamg_asm_use_agg = True
		mg_levels_ksp_type = "gmres"
		mg_levels_pc_type = "sor"
	  	
	#performance options
		mattransposematmult_via = "matmatmult"
		pc_gamg_reuse_interpolation = "True"
		mg_levels_ksp_max_it = 40
};

\end{lstlisting}

The linear systems we handle are fairly ill-conditioned, therefore, the smoothers we need to use are fairly expensive.  The ASM/LU based smoother is more expensive compared to other smoothers like block Jacobi, however ASM/LU is more robust (better convergence).  This setup works very well with a relatively constant number of Krylov iterations as the number of processes are increased in the massively parallel setting. The scaling results we present use the same setup of solvers,  but there is substantial room for improvement in this area of the code where fieldsplit preconditioners using Schur complement can be used as smoothers to improve speed of the AMG solver. 


\end{document}